\documentclass[11pt, twoside, leqno]{article}
\usepackage{amsfonts}
\usepackage{mathtools}
\usepackage{amssymb}
\usepackage{amscd}
\usepackage{amsthm}
\usepackage[dvips]{graphicx}
\usepackage[all,cmtip]{xy}
\usepackage[colorlinks,linkcolor=blue,anchorcolor=blue, citecolor=red]{hyperref}
\usepackage{yhmath}
\usepackage{amsmath}
\usepackage{color}
\usepackage{mathrsfs}
\usepackage{txfonts}
\usepackage{amsfonts}
\usepackage{tikz-cd}

\usepackage{indentfirst}
\usepackage{ulem}
\usepackage[all]{xy}
\usepackage{mathdots}
\usepackage{leftidx}
\usepackage{mathrsfs}
\usepackage{amsmath,amssymb, amsthm}
\usepackage{color}
\usepackage{tikz}
\usepackage{picture}
\usepackage{enumerate}
\usepackage{makecell}
\usepackage{extarrows}
\usepackage{graphicx}
\usepackage{caption}
\usepackage{subcaption}
\usepackage[pagewise]{lineno}
\allowdisplaybreaks

\numberwithin{equation}{section}
\newtheorem{Thm}{Theorem}[section]
\newtheorem{Lem}[Thm]{Lemma}
\newtheorem{Def}[Thm]{Definition}
\newtheorem{Cor}[Thm]{Corollary}
\newtheorem{Prop}[Thm]{Proposition}

\newtheorem{Rem}[Thm]{Remark}
\newtheorem{Conj}[Thm]{Conjecture}

\newtheorem{Que}[Thm]{Question}

\newcommand{\RNum}[1]{\uppercase\expandafter{\romannumeral #1\relax}}
\newcommand{\cat}[1]{\mathscr{#1}}
\newcommand{\psb}[1]{[ \! [#1] \! ]}
\newcommand{\nc}{\newcommand}
\nc{\dmo}{\DeclareMathOperator}
\dmo{\SH}{SH}
\dmo{\Spc}{Spc}
\dmo{\Spec}{Spec}
\dmo{\Res}{res}
\dmo{\End}{End}
\dmo{\cone}{cone}
\nc{\SHc}{{\SH^c}}
\nc{\SHGc}{\SHG^c}
\nc{\SHG}{\SH(G)}
\renewcommand{\appendix}{\par
\setcounter{section}{0}%
\setcounter{subsection}{0}%
\setcounter{subsubsection}{0}%
\gdef\thesection{\@Alph\c@section}%
\gdef\thesubsection{\@Alph\c@section.\@arabic\c@subsection}%
\gdef\theHsection{\@Alph\c@section.}%
\gdef\theHsubsection{\@Alph\c@section.\@arabic\c@subsection}%
\csname appendixmore\endcsname
}
\title{\bf\Large A General Blue-Shift Phenomenon
\footnotetext{\hspace{-0.35cm} 2020 {\it
Mathematics Subject Classification}. Primary 55N22; Secondary 55N20, 55P42, 55P91, 55Q10, 55R40.
\endgraf {\it Key words and phrases.}
 General blue-shift phenomenon; Generalized Tate construction; Equivariant stable homotopy; Bousfield class; Relationship between roots and coefficients; $n$-tuple; $p^j$-series.
\endgraf }}
\author{Yangyang Ruan}
\date{}
\numberwithin{equation}{section}

\begin{document}
\begin{sloppypar}
\arraycolsep=1pt

\maketitle

\vspace{-0.3cm}

\begin{center}
\begin{minipage}{13cm}
{\small {\bf Abstract}\quad
In chromatic homotopy theory, there is a well-known conjecture called blue-shift phenomenon (BSP). In this paper, we propose a general blue-shift phenomenon (GBSP) which unifies BSP and a new variant of BSP introduced by Balmer--Sanders under one framework. To explain GBSP, we use the roots of $p^j$-series of the formal group law of a complex-oriented spectrum $E$ in the homotopy group of the generalized Tate spectrum of $E$. We also incorporate the relationship between roots and coefficients of a polynomial in any commutative ring. With this fresh perspective, we successfully achieve our goal of explaining GBSP for certain abelian cases, which provides the first example of Tate blue-shift with height-shifting at arbitrary positive integer in this setting. Additionally, we establish that the generalized Tate construction lowers Bousfield class, along with numerous Tate vanishing results. These findings strengthen and extend previous theorems of Balmer--Sanders and Ando--Morava--Sadofsky, and reproduce a result of Barthel--Hausmann--Naumann--Nikolaus--Noel--Stapleton. Furthermore, our approach simplifies the original proof of a result of Bonventre--Guillou--Stapleton, indicating that its applications are not limited to GBSP.
Our work pioneers the use of commutative algebra to explain the chromatic height-shifting behavior in the blue-shift phenomenon. }
\end{minipage}
\end{center}

\newpage
\renewcommand\contentsname{\begin{center}\textbf{CONTENTS}\end{center}}
\tableofcontents

\section{Introduction\label{s1}}
Chromatic homotopy theory studies a filtration of the stable homotopy category when localized at a prime $p$, and this filtration is closely related to a complete invariant, the height of formal group laws, for classifying formal group laws over a field of characteristic $p$. At filtration $0$, one sees rational cohomology theory related to the additive formal group law. At filtration up to $1$, one sees real or complex $K$-theory related to the multiplicative formal group law.
  At filtration up to $2$, one sees topological modular forms related to a formal group law arising from a generalized Weierstrass equation. Generally in filtration up to $n$, one sees $n$-th Johnson--Wilson
 theory $E(n)$; the intermediate filtrations see the $n$-th Morava $K$-theory $K(n)$, this giving the layers between $E(n-1)$ and $E(n)$. In the limit, one sees mod $p$ cohomology theory related to a formal group law of height infinity.

This paper is concerned with a phenomenon in which cohomology theories of lower height arise
from theories of higher height: the blue-shift phenomenon in Tate cohomology. Roughly speaking, for a finite group $G$, applying the categorical $G$-fixed point functor $(-)^G$ for the \textit{classical} \textit{Tate} \textit{construction} $t_{G}({\rm inf}^G_{e}(E))$\footnote{This is in the sense of Greenlees--May \cite{GM95}, see also Section \ref {Mot} for details.} of a non-equivariant $v_n$-\textit{periodic}\footnote{Usually $v_n$-periodic means that $v_n$ is a unit in the homotopy ring $\pi_*(E)$, but in this paper, we choose a less restrictive definition due to Hovey \cite{Ho95}, see also Definition \ref{De1}.} spectrum $E$, one obtains a new spectrum $t_{G}({\rm inf}^G_{e}(E))^G$.
The blue-shift results obtained by far abounds, we summarise various blue-shift phenomena into the following conjecture.
\begin{Conj}[$\mathbf{Classical}$ $\mathbf{blue}$-$\mathbf{shift}$ $\mathbf{phenomenon}$]\label{Cbsp}
$t_{G}({\rm inf}^G_{e}(E))^G$ is $v_{n-s_{G;E}}$-periodic for some positive integer $s_{G;E}$. To make Tate vanishing results fit into this framework, especially when $s_{G;E}>n$, the $v_{n-s_{G;E}}$-periodic ring spectrum denotes the contractible spectrum $*$. We call $s_{G;E}$ \textit{blue}-\textit{shift} \textit{number}.
 \end{Conj}
\subsection{Main results}
For a finite group $G$, let $\SHG$ denote the $G$-equivariant stable homotopy category and $\SHGc$\footnote{It is also called the category of compact genuine $G$-spectra, and ``genuine" means that each $G$-spectrum has a complete $G$-universe.} denote its full subcategory that consists of all compact objects\footnote{ Naively ``compact objects" are finite $G$-spectra with finite $G$-CW decompositions.} of $\SHG$. Balmer--Sanders in their 2017 paper \cite{BS17} established a connection between the classical blue-shift phenomenon for $G=\mathbb{Z}/p$ with any prime $p$ and the Zariski topology of the Balmer spectrum $\Spc(\SH(\mathbb{Z}/p)^c)$ of $\SH(\mathbb{Z}/p)^c$. This Balmer spectrum is a $\mathbb{Z}/p$-equivariant counterpart of the work by Devinatz--Hopkins--Smith \cite{DHS,HS98}. Besides, to compute the Zariski topology of $\Spc(\SH(G)^c)$, Balmer--Sanders introduced a new construction $\Phi^G(t_{G}({\rm inf}^G_{e}(-)))$ that replaces the functor $(-)^G$ in the classical blue-shift construction $t_{G}({\rm inf}^G_{e}(-))^G$ with the geometric fixed point functor $\Phi^G(-)$. This gave rise to a \textit{new} \textit{blue}-\textit{shift} \textit{phenomenon}. In 2019, Barthel--Hausmann--Naumann--Nikolaus--Noel--Stapleton \cite{BHNNNS} further investigated this new blue-shift phenomenon to obtain the Zariski topology of $\Spc(\SH(A)^c)$ for any abelian group $A$. To unify the classical and the new blue-shift phenomena under one framework, we propose a general blue-shift phenomenon. Specifically, we consider a finite group $G$, and a normal subgroup $N$ of $G$. We introduce the relative geometric $N$-fixed point functor $\tilde{\Phi}^{N}(-):\SHG\rightarrow\SH(G/N)$. With this setup, we define a more general functor, denoted as $(\tilde{\Phi}^{N}(t_{G}({\rm inf}^{G}_e(-))))^{G/N}$. This functor is obtained by replacing the functor $(-)^G$ in the classical blue-shift construction $t_{G}({\rm inf}^G_{e}(-))^G$ with the functor $(\tilde{\Phi}^{N}(-))^{G/N}$. For convenience, we refer to this functor as $\cat T_{G,N}(-)$. The functor $\cat T_{G,N}(-)$ maps non-equivariant spectra to themselves. In this paper, we call $\cat T_{G,N}(-)$ the \textit{generalized} \textit{Tate} \textit{construction} for non-equivariant spectra. And for a non-equivariant spectrum $E$, we call $\cat T_{G,N}(E)$ the $\textit{generalized}$ $\textit{Tate}$ $\textit{spectrum}$ of $E$. The general blue-shift phenomenon can be stated as follows:
 \begin{Conj}[$\mathbf{General}$ $\mathbf{blue}$-$\mathbf{shift}$ $\mathbf{phenomenon}$]\label{Gbsp}
 The functor $\cat T_{G,N}(-)$ maps a $v_n$-periodic spectrum $E$ to a $v_{n-s_{G,N;E}}$-periodic spectrum $\cat T_{G,N}(E)$ for some positive integer $s_{G,N;E}$. In other words, this generalized Tate construction reduces chromatic periodicity.
 \end{Conj}
\begin{Rem}
\begin{enumerate}
\item[{\rm (i)}] When $N=G$, $\cat T_{G,N}(-)$ is the construction $\Phi^{G}(t_{G}({\rm inf}^{G}_{e}(-)))$ in the new blue-shift phenomenon of Balmer--Sanders, details see Proposition \ref{Shgg}.
\item[{\rm (ii)}] When the family subgroups of $G$ which do not contain $N$ are $\{e\}$, one special case is that $G=\mathbb{Z}/p^j$ and $N=\mathbb{Z}/p$ for any positive integer $j$, $\cat T_{G,N}(-)$ is the construction $t_{G}({\rm inf}^{G}_{e}(-))^{G}$ in the classical blue-shift phenomenon, details see Proposition \ref{Sfy}.
\end{enumerate}
\end{Rem}
The goal of this paper is to study this general blue-shift phenomenon, namely Conjecture \ref{Gbsp}, and a consequence of our main theorem (Theorem \ref{Gtclbc}) gives a partial answer for abelian cases. To state our main theorem, we need to introduce some notations. For a finite abelian $p$-group $A$, the $p$-$\textit{rank}$ of $A$ is the number of $\mathbb{Z}/p$ factors in the maximal elementary abelian subgroup of $A$, and it is denoted by ${\rm rank}_p(A)$. Let $\langle E\rangle$ denote Bousfield class of $E$, See \cite{Bo79} or Section \ref{Mot} for details.  Here is our main theorem (a more general version is Theorem \ref{GTglbc}),
\begin{Thm}[$\mathbf{Generalized}$ $\mathbf{Tate}$ $\mathbf{construction}$ $\mathbf{lowers}$ $\mathbf{Bousfield}$ $\mathbf{class}$]\label{Gtclbc}
Let $E$ be a $p$-complete, complex oriented spectrum with an associated formal group of height $n$. Let $A$ be a finite abelian $p$-group and $C$ be its direct summand. If $E$ is $\textit{Landweber}$ $\textit{exact}$\footnote{See \cite{La76} or Proposition \ref{LEdef} for
details.}, then $\cat T_{A,C}(E)$ is Landweber exact and $v_{n-{\rm rank}_p(C)}$-periodic. Hence $\langle\cat T_{A,C}(E)\rangle=\langle E(n-{\rm rank}_p(C))\rangle$. When $k>n$, $ E(n-k)=*$.
\end{Thm}
\begin{Rem}
\begin{enumerate}
\item[{\rm (i)}] By \cite[Corollary 1.12]{Ho95}, the assumption on $E$ implies that $\langle E\rangle=\langle E(n)\rangle$.
\item[{\rm (ii)}] When $A=C=\mathbb{Z}/p$ and $E=E(n)$, this theorem implies the corresponding case of \cite[Theorem 1.2]{HS96}, and gives an upper bound of $\mathbf{\rm BS}_m(\mathbb{Z}/p;\mathbb{Z}/p,e)$, that is $\mathbf{\rm BS}_m(\mathbb{Z}/p;\mathbb{Z}/p,e)\leq 1$, which implies \cite[Proposition 7.1]{BS17}, details see Section \ref{Mot}.
\item[{\rm (iii)}] When $A=C=(\mathbb{Z}/p)^k$ and $E$ is the $n$-th Morava E-theory $E_n$, this theorem implies \cite[Proposition  3.0.1]{St12}.
\item[{\rm (iv)}]A corollary is that $\langle\cat T_{A,A}(E(n))\rangle=\langle E(n-{\rm rank}_p(A))\rangle$. If $A=C=H/K$ is an abelian $p$-group, then this theorem gives an upper bound of $\mathbf{\rm BS}_m(G;H,K)$, that is $\mathbf{\rm BS}_m(G;H,K)\leq {\rm rank}_p(H/K)$, which implies \cite[Theorem 1.5]{BHNNNS}, details see Section \ref{Mot}.
\item[{\rm (v)}]If $A=C$ is any elementary abelian $p$-group and $E=E(n)$, then one way to get the upper bound of $s_{A,A;E(n)}$ is by generalizing Ando--Morava--Sadofsky's theorem \cite[Proposition 2.3]{AMS} from $\mathbb{Z}/p$ to any elementary abelian $p$-group, details see Theorem \ref{splite}.

\end{enumerate}
\end{Rem}

\subsection{Background of the blue-shift phenomenon and New tools to settle Conjecture \ref{Gbsp} }
 As far as we know, the classical blue-shift phenomenon, namely Conjecture \ref{Cbsp}, was discovered by Davis--Mahowald \cite{DM84} in 1984.
They found that if $G$ is a cyclic group of order $2$, denoted by $\mathbb{Z}/2$, then the construction
$t_{\mathbb{Z}/2}({\rm inf}^{\mathbb{Z}/2}_{e}(-))^{\mathbb{Z}/2}$ maps the $v_1$-periodic $2$-local ring spectra both $bo$ (representing connected
real K-theory) and $bu$ (representing connected
complex K-theory) to a wedge of suspensions of the $v_{0}$-periodic spectrum $K(\mathbb{Z}_2)$ (representing the Eilenberg-Maclane spectrum for 2-adic integers). Building upon this finding, they formulated a conjecture that extended this result to replace $bu$ with the $2$-local spectrum
$BP\langle n\rangle$ of \cite{JW73} and $K(\mathbb{Z}_2)$ with $BP\langle n-1\rangle$. Later,
in 1986 Davis--Johnson--Klippenstein--Mahowald--Wegmann   \cite{DJKJMW} proved Davis--Mahowald's conjecture for $n=2$ and a generalization to every prime, which motivated Davis--Mahowald's conjecture for each prime. In 1994, Greenlees--Sadofsky
\cite[Theorem1.1]{GS96} investigated the behavior of $t_{G}({\rm inf}^G_{e}(K(n)))^{G}$ and
they found that it is equivalent to the trivial spectrum $*$ for any $p$-group $G$. In 1996, Hovey--Sadofsky \cite{HS96} explored the case
when $G$ is the cyclic group $\mathbb{Z}/p$, $E$ is $v_n$-periodic and Landweber exact. In this scenario, they discovered that the blue-shift number $s_{\mathbb{Z}/p;E}$ is
always 1, regardless of the prime $p$. Further contributions to the understanding of the classical blue-shift phenomenon came in 1998 when Ando--Morava--Sadofsky \cite{AMS} confirmed the correctness of Davis--Mahowald's conjecture for every prime. In 2004, Kuhn \cite{Ku04} made an important advancement by proving that $t_{G}({\rm inf}^G_{e}(T(n)))^G$ is equivalent to the trivial spectrum $*$ for any $p$-group $G$. Here, $T(n)$ represents the telescope of any $v_n$-self map of a finite complex of $\textit{type}$ $n$, details see Subsection \ref{Ap}. For outside of the complex oriented setting, there are some further developments \cite{BR19,LLQ22}. It is worthwhile to mention that ``blue-shift" was not in use at the time of these results except \cite{LLQ22}, actually the introduction of this terminology into algebraic topology is due to Rognes \cite{R00}\footnote{Around 1999 Rognes coined use of the word ``red-shift" for the phenomenon that circle Tate constructions of topological Hochschild homology, and algebraic K-theory, increase chromatic complexity, and formulated a red-shift problem for topological cyclic homology at an Oberwolfach lecture \cite{R00} in 2000. Several years later, the expression blue-shift was introduced, to emphasize that the shift goes in the opposite direction of red-shift.}.

 With the exception of the vanishing results mentioned above, the chromatic height-shift observed in the blue-shift phenomenon is always 1. Our main theorem provides the first known examples where this shift occurs by an arbitrary positive integer in this setting.

In this paper, we find an idea that could explain both the classical and the new blue-shift a under the framework of the general blue-shift phenomenon. $\textit{Our}$ $\textit{main}$ $\textit{idea}$ is that since the homotopy group $\pi_*(\cat T_{G,N}(E))$ of the generalized Tate spectrum $\cat T_{G,N}(E)$ is a graded ring, it must be isomorphic to a quotient of a free graded ring by some relations. And we may reduce these relations like solving equations to obtain $v_{n-s_{G,N;E}}$, then we need to prove that the solution of $v_{n-s_{G,N;E}}$ is invertible in $\pi_*(\cat T_{G,N}(E))$. This idea represents the first time that commutative algebra has been used to understand the chromatic height-shifting behavior in the blue-shift phenomenon.

Inspired by Hopkins--Kuhn--Ravenel's work \cite{HKR}, we utilize the roots of $p^j$-series $[p^j]_E(-)$ of formal group law of $E$ in $\pi_*(\cat T_{G,N}(E))$ to execute our main idea. By using the Gysin sequence of $S^1 \to B\mathbb{Z}/p^j \to \mathbb{C}P^{\infty}$ and the fact that $[p^j]_E(x)$ is not a zero divisor in the formal power series ring $E^*\psb{x}$ with $x$ a complex orientation of $E$, one obtains that $E^*(B\mathbb{Z}/p^j)\cong E^*\psb{x}/([p^j]_E(x))$. Besides, $E^*(B\mathbb{Z}/p^j)$ is a Hopf algebra over $E^*$ where the coalgebra structure is induced by the multiplication map $\mu_{B\mathbb{Z}/p^j}: B\mathbb{Z}/p^j\times B\mathbb{Z}/p^j\rightarrow B\mathbb{Z}/p^j$. To calculate the roots of $[p^j]_E(-)$ in a $\textit{graded}$ $E^*$-$\textit{algebra}$ which denotes a graded Hopf algebra over $E^*$, we recall a definition due to Hopkins--Kuhn--Ravenel.
\begin{Def}(Hopkins--Kuhn--Ravenel, \cite[Definition 5.5]{HKR})\label{FG}
Let $R$ be a graded $E^*$-algebra and $j$ be a natural number. Then the set of $E^*$-algebra homomorphisms ${\rm  Hom}_{E^*\text{-alg}}(E^*\psb{x}/([p^j]_E(x)),R)$, denoted by ${}_{p^{j}\!}F(R)$, forms a group.
\end{Def}
\begin{Rem}\label{derpj}
As $f^*\in {\rm  Hom}_{E^*\text{-alg}}(E^*\psb{x}/([p^j]_E(x)),R)$ is an $E^*$-ring homomorphism, there is a one-one correspondence between $f^*$ and its image $f^*(x)$. If we identify $f^*$ with its image $f^*(x)$, since $f^*([p^j]_E(x))=[p^j]_E(f^*(x))=0$, then $f^*$ is viewed as a root of $[p^j]_E(-)$ in $R$. And ${}_{p^{j}\!}F(R)$ is viewed as a set of roots of $[p^j]_E(-)$ in $R$.
\end{Rem}
If $\pi_*(\cat T_{G,N}(E))$ possesses an $E^*$-algebra structure, we can view ${}_{p^{j}\!}F(\pi_*(\cat T_{G,N}(E)))$  as a set of roots of $[p^j]_E(-)$ in $\pi_*(\cat T_{G,N}(E))$, as remarked in Remark \ref{derpj}. After simplifying the construction of $\cat T_{G,N}(-)$, we can identify the homotopy group $\pi_*(\cat T_{G,N}(E))$ with the $G/N$-equivariant homotopy group
$\pi^{G/N}_*(\tilde{\Phi}^N(F(EG,{\rm inf}_e^G(E))))$ of a $G/N$-spectrum $\tilde{\Phi}^N(F(EG,{\rm inf}_e^G(E)))$, as detailed in Proposition \ref{Sfy}. Combining this with Costenoble's Theorem \cite[Chapter \RNum{2} Proposition 9.13]{LMS} (see also Theorem \ref{Cos}), we can identify $\pi^{G/N}_*(\tilde{\Phi}^N(F(EG,{\rm inf}_e^G(E))))$ with $L^{-1}_NE^*(BG)$, where $L_N$ is a multiplicatively closed set generated by the set
$$M_N=\{\chi_V \in E^*(BG)\mid \text{$V$ is any complex representation of $G$ such that}~V^N=0\}$$
of Euler classes. The work \cite{HKR} is regarded as one of the most significant and profound results in the study of the generalized cohomology of $BG$. They demonstrated that for an abelian group $G$, $E^*(BG)$ can be computed and represented by a beautiful $E^*$-algebra. However, for a general non-abelian group $G$, there is no known method to compute $E^*(BG)$. One of the primary challenges might lie in the fact that $BG$ may not have an $H$-space structure for non-abelian groups, which implies that $E^*(BG)$ may not possess a coalgebra structure. As the $E^*$-algebra structure is crucial, in this study, we focus on the case where $G$ is an abelian group $A$. Since $BG$ is homotopy equivalent to the classifying space of the $p$-Sylow group of $G$ after localizing at $p$ for a prime $p$, without loss of generality, we can work $p$-locally and assume that $A$ is an abelian $p$-group. We consider $N$ as a subgroup $C$ of $A$. Based on Costenoble's Theorem and the work of $E^*(BA)$ in \cite{HKR}, we calculate the homotopy group $\pi_*(\cat T_{A,C}(E))\cong L^{-1}_CE^*(BA)$ explicitly in the sense that we determine those inverted Euler classes in $E^*(BA)$, see for Theorem \ref{bgcoh}.

As ${}_{p^{j}\!}F(\pi_*(\cat T_{A,C}(E)))$ is well-defined, then by Weierstrass Preparation Theorem \ref{WPT}, we have an $E^*$-algebra isomorphism
$$\eta: E^*\psb{x}/([p^j]_E(x))\rightarrow E^*[x]/(g_j(x)),$$
where $g_j(x)$ is the Weierstrass polynomial of $[p^j]_E(x)$, which identifies the power series $[p^j]_E(x)$ with the polynomial $g_j(x)$ and their corresponding roots in $\pi_*(\cat T_{A,C}(E))$. To determinate the periodicity of $\cat T_{A,C}(E)$, we study the relationship between roots and coefficients of $g_j(x)$ in $\pi_*(\cat T_{A,C}(E))$.

Let $R$ be a commutative ring with 1 and $f(x)$ be a polynomial of degree $m$ over $R$. A polynomial $f(x)$ in $R[x]$ can viewed as a polynomial map from $R$ to $R$, which maps $r\in R$ to $f(r)\in R$. To identify $f(x)$ with its corresponding polynomial map, we propose a notion of
$n$-$\textit{tuple}$ $\textit{of}$ $f(x)$ in Section \ref{GRRC}. Recall that an $n$-tuple $\{r_1,r_2,\cdots, r_n\}$ of $f(x)$ is a subset of $R$ such that $f(r_i)=0$ and $r_i-r_j$ is not zero or zero-divisor  for each $1\leq i\neq j \leq n$. By using this notion, we generalize the relationship between roots and coefficients of a polynomial over the complex field to any commutative ring.
\begin{Thm}($\mathbf{Generalized}$ $\mathbf{relations}$ $\mathbf{between}$ $\mathbf{roots}$ $\mathbf{and}$ $\mathbf{coefficients}$ $\mathbf{of}$ $\mathbf{a}$ $\mathbf{polynomial}$)\label{Fac-tuple}
Let $R$ be a commutative ring with 1 and $f(x)=a_0+a_1x+\cdots+a_mx^m$ be a polynomial over $R$. Suppose that $R$ has an $n$-tuple $\{r_1,r_2,\cdots, r_n\}$ of $f(x)$.
\begin{enumerate}
\item[{\rm (i)}] If $n>m$, then $a_i=0$ in $R$ for $0\leq i\leq m$;
\item[{\rm (ii)}] if $n=m$, then
$$a_i=(-1)^{n-i}a_n\sum_{1\leq k_1\neq k_2\neq\cdots\neq k_{n-i}\leq n}r_{k_1}r_{k_2}\cdots r_{k_{n-i}}~\text{in}~R~\text{for}~0\leq i\leq n-1~\text{and}~\text{hence}~  f(x)=a_n\prod^n_{i=1}(x-r_i);$$
\item[{\rm (iii)}] if $n\leq m$, then $a_i=\frac{\det(\alpha_0,\cdots,\alpha_{i-1},\beta, \alpha_{i+1},\cdots,\alpha_{n-1})}{\det(\alpha_0,\alpha_1,\cdots,\alpha_{n-1})}$ in $R$ for $0\leq i\leq n-1$, where $\alpha_i$ denotes the column $R$-vector $(r_1^i,r_2^i,\cdots,r_n^i)^T$ for $0\leq i\leq n-1$ and $\beta$ denote the column $R$-vector $(-\sum^m_{i=n}a_{i}r_1^{i},-\sum^m_{i=n}a_{i}r_2^{i},\cdots, -\sum^m_{i=n}a_{i}r_n^{i})^T$.
\end{enumerate}
\end{Thm}
\begin{Rem}\label{Fac-R}
\begin{enumerate}
\item[{\rm (i)}] It is impossible for a nonzero polynomial over a field to have the number of roots more than its degree, whereas it is possible for a nonzero polynomial over a commutative ring, such as the nonzero polynomial $x^2$ over $\mathbb{Z}[x_1,x_2]/(x_1^2,x_2^2)$.
\item[{\rm (ii)}]To some extent, this theorem is a generalization of polynomial factorization. It is easy to see that the first two cases of this theorem imply that $f(x)$ has a polynomial factorization. The third case just showed that if $n\leq m$, one can obtain a factorization $f(x)=a_n\prod^n_{i=1}(x-r_i)$ in
$R[x]/(a_{m-n+1},a_{m-n+2},\cdots,a_{m})$.
\end{enumerate}
\end{Rem}
The following corollary of Theorem \ref{Fac-tuple} gives a sufficient yet useful condition to guarantee the vanishing of a commutative ring.
\begin{Cor}($\mathbf{Vanishing}$ $\mathbf{ring}$ $\mathbf{condition}$)\label{vanish}
Let $f(x)=a_0+a_1x+\cdots+a_mx^m$ be a polynomial over a commutative ring $R$ with 1. $R$ has an $n$-tuple $\{r_1,r_2,\cdots, r_n\}$ of $f(x)$ under the assumption that $R\neq 0$.
\begin{enumerate}
\item[{\rm (i)}] If $n>m$ and 1 belongs to the ideal $(a_0,a_1,\cdots, a_n)$ of $R$, then $R=0$;
\item[{\rm (ii)}] if $n\leq m$ and 1 belongs to the ideal $(a_0-\frac{\det(\beta, \alpha_{1},\alpha_{2},\cdots,
\alpha_{n-1})}{\det(\alpha_0,\alpha_1,\cdots,\alpha_{n-1})},a_1-\frac{\det(\alpha_0,\beta,\alpha_{2} \cdots,\alpha_{n-1})}{\det(\alpha_0,\alpha_1,\cdots,\alpha_{n-1})},\cdots, a_n-\frac{\det(\alpha_0,\cdots,\alpha_{i-1},\beta, \alpha_{i+1},\cdots,\alpha_{n-1})}{\det(\alpha_0,\alpha_1,\cdots,\alpha_{n-1})})$ of $R$, then $R=0$.
\end{enumerate}
\end{Cor}
\begin{Rem}
While this corollary may seem to follow immediately from Theorem \ref{Fac-tuple}, it in fact provides a completely new method for proving the vanishing of a ring. In equivariant stable homotopy theory, it is common to claim computations of the homotopy groups of the $H$-geometric fixed point of the Borel-equivariant $G$-spectrum $F(EG,\inf^G_eE)$ arising from some complex oriented spectrum $E$. However, we often do not even know if these homotopy groups are non-trivial, since they are obtained by inverting certain Euler classes in $E^*(BG)$. Most previously known methods for proving the vanishing of these homotopy groups involve recognizing that one of the Euler classes is nilpotent, as in the case $E=K(n), G=\mathbb{Z}/p^j$. However, if all the Euler classes to be inverted are not zero-divisors in $E^*(BG)$, then these classical methods fail. Our new approach remains effective in such cases,
provided we can find an $n$-tuple of $[p^j]_E(x)$ in the multiplicatively closed subset satisfying $\deg_W [p^j]_E(x)<n$; even if this condition is not met, it can still yield valuable homological information about these homotopy groups. When the condition holds, our corollary guarantees their vanishing.
\end{Rem}
The usefulness of Corollary \ref{vanish} can be seen in Corollary \ref{usevan} which includes new proofs of Tate vanishing result \cite[Theorem 1.1]{GS96} of Morava $K$-theory and, the vanishing result \cite[Proposition 3.10]{BGS} of the geometric $H$-fixed point of $G$-equivariant complex K-theory for a $p$-group $G$ and a non-cyclic subgroup $H$. And our approach greatly simplifies those original proofs. Besides, the most important application of  Corollary \ref{vanish} lies in explaining the general blue-shift phenomenon.

\subsection{Proof strategy of Theorem \ref{Gtclbc}}
The crux of comprehending the general blue-shift phenomenon lies in understanding the blue-shift number $s_{G,N;E}$. Since computing $s_{G,N;E}$ is tantamount to determining the periodicity of $\cat T_{G,N}(E)$, the central question becomes how to characterize the periodicity of $\cat T_{G,N}(E)$. This necessitates a thorough grasp of the $v_n$-periodic spectrum. To our knowledge, there exist at least two definitions of $v_n$-periodic, as elaborated in Section \ref{GRRC}. However, in this paper, we opt for Hovey's definition and provide a recap of it.

\begin{Def}(Hovey's $v_n$-periodic, \cite{Ho95})\label{De1}
Let $E$ be a $p$-local and complex oriented spectrum. Let $I_{n}$ denote the ideal of the homotopy group $\pi_*(E)=E^*$ generated by $v_0, v_1, \cdots, v_{n-1}$. The spectrum $E$ is called $v_n$-$\textit{periodic}$ if $v_n$ is a unit of $E^*/I_{n}\neq 0$.
\end{Def}
\begin{Rem}
If $E$ is a $p$-local and complex oriented spectrum, then there are a formal group law over $\pi_*(E)$ and a ring homomorphism from the homotopy group $\pi_*(BP)=\mathbb{Z}_{(p)}[v_1,v_2,\cdots]$ of the Brown-Peterson spectrum $BP$ to $E^*$ which classifies this formal group law. Then $I_{n}$ is the ideal of $E^*$ generated by the image of $v_0=p, v_1, \cdots, v_{n-1}$ under this ring homomorphism, and we still use $v_i$ denote its image.
\end{Rem}
To give a purely algebraic description of the periodicity of $\cat T_{G,N}(E)$, we refine Hovey's definition in Definition \ref{Rd} and hence
find that a spectrum $E$ is $v_n$-periodic if and only if $E^*/I_{n+1}=0, E^*/I_n\neq 0$. In Theorem \ref{Gtclbc}, we specialize to
the case where $G$ is a finite abelian $p$-group $A$ and $N$ is a subgroup $C$ of $A$. Additionally, $E^*$ is considered a local ring with the
maximal ideal $I_n$. By calculating $\pi_*(\cat T_{A,C}(E))$ in Theorem \ref{bgcoh}, we observe that $\pi_*(\cat T_{A,C}(E))$ is an
$E^*$-module. Consequently, we define an integer $\mathbf{s}_{A,C;E}$ to characterize the periodicity of $\cat T_{A,C}(E)$.
\begin{Def}\label{Cri}
There is an ascending chain of ideals
$$I_{-1}=\emptyset\subseteq I_{0}=(0)\subseteq I_1 \subseteq \cdots \subseteq I_{n+1-q} \subseteq \cdots \subseteq I_{n+1}=\pi_{*}(\cat T_{A,C}(E)),$$
then $\mathbf{s}_{A,C;E}$ is the $\textit{maximal}$ integer $q$ such that $I_{n+1-q}=\pi_{*}(\cat T_{A,C}(E))$ and also is the $\textit{minimal}$ integer $q$ such that $I_{n-q}\subsetneq\pi_{*}(\cat T_{A,C}(E))$, which is equivalent to $$\pi_{*}(\cat T_{A,C}(E))/I_{n+1-q}=\left\{\begin{array}{ll}
0\hspace{1.1cm} \text{if $0\leq q \leq\mathbf{s}_{A,C;E}$},\\
\neq 0\hspace{0.7cm} \text{if $\mathbf{s}_{A,C;E}<  q$}.
\end{array}\right.$$
\end{Def}
By Definition \ref{Rd}, it is easy to see that
\begin{Lem}\label{Max}
 $$s_{A,C;E}=\mathbf{s}_{A,C;E}.$$
\end{Lem}
The integer $\mathbf{s}_{A,C;E}$ can be elucidated in terms of Homology algebra. According to Lemma \ref{LE}, ${\cat T}_{A,C}(E)$ inherits the Landweber exactness property of $E$. Consequently, $v_0, v_1, \cdots, v_{n-\mathbf{s}_{A,C;E}}$ constitute a maximal regular $\pi_{*}(\cat T_{A,C}(E))$-sequence within $I_{n}$ of $E^*$. In Homology algebra, the maximal length of a $\pi_{*}(\cat T_{A,C}(E))$-regular sequence in the maximal ideal $I_{n}$ of $E^*$ measures the $I_n$-depth of $\pi_{*}(\cat T_{A,C}(E))$ as an $E^*$-module. This depth is defined by the minimum integer $d$ such that ${\rm Ext}^d_{E^*}(E^*/I_n, \pi_{*}(\cat T_{A,C}(E)))\neq 0$.

Let ${\rm pd}_{E^*}(\pi_{*}(\cat T_{A,C}(E)))$ denote the projective dimension of $\pi_{*}(\cat T_{A,C}(E))$ as an $E^*$-module. This dimension is defined as the minimum length among all finite projective resolutions of $\pi_{*}(\cat T_{A,C}(E))$ as an $E^*$-module. Notably, the $I_n$-depth of $E^*$ is $n$. Hence, by the Auslander-Buchsbaum formula \cite[Theorem 3.7]{AB57}, we have:
\begin{Prop}\label{bs-pd}
$$\mathbf{s}_{A,C;E}={\rm pd}_{E^*}(\pi_{*}(\cat T_{A,C}(E)))=n- \min\{d\mid {\rm Ext}^d_{E^*}(E^*/I_n, \pi_{*}(\cat T_{A,C}(E)))\neq 0\}.$$
\end{Prop}
Proposition \ref{bs-pd} offers a purely algebraic characterization of $\mathbf{s}_{A,C;E}$, which also extends to provide the same characterization for the blue-shift number $s_{A,C;E}$. However, from a computational standpoint, we employ Definition \ref{Cri} instead of Proposition \ref{bs-pd} to compute $\mathbf{s}_{A,C;E}$. By utilizing Corollary \ref{vanish}, if we find some-tuple of $p^j$-series $[p^j]_{E}(x)$ in $\pi_*(\cat T_{A,C}(E))$, we can establish an upper bound for $\mathbf{s}_{A,C;E}$. Moreover, by leveraging Lemma \ref{ind} inductively and assuming $E^*/I_n\neq 0$, we derive a lower bound for $\mathbf{s}_{A,C;E}$. This approach constitutes our strategy to prove Theorem \ref{Gtclbc} and Theorem \ref{GTglbc}.

\subsection{Furture work: some ideas to settle the non-abelian cases of Conjecture \ref{Gbsp}}
We do not anticipate that our method, which employs $n$-tuples of $[p^j]_E(x)$, will provide a complete solution to the general blue-shift phenomenon for arbitrary non-abelian groups. Nevertheless, we are optimistic that it can be adapted to certain non-abelian cases, an endeavor that will require substantial further work.

Addressing non-abelian cases requires solving a key problem: computing the roots of $[p^j]_E(-)$ in the homotopy groups of $\cat T_{G,N}(E)$, which is equivalent to finding these roots in $E^*(BG)$. For abelian groups, we define a homomorphism $\psi^{p^j}_{G}: G \rightarrow G$ by $\psi^{p^j}_{G}(g) = g^{p^j}$. Using the functoriality of the classifying space functor $B$, we obtain a map $B\psi^{p^j}_{G} = \psi^{p^j}_{BG}$, making the induced map $\psi^{p^j,*}_{BG}: E^*(BG) \rightarrow E^*(BG)$ an $E^*$-algebra homomorphism. Crucially, the restriction of $\psi^{p^j,2}_{BG}$ to Euler classes coincides with the operation $[p^j]_E(-)$. This key insight allows us to compute the roots of $[p^j]_E(-)$ in $E^*(BG)$ directly at the group level, with full details in Theorem \ref{fpjr}. For non-abelian groups, however, a fundamental question arises:
\begin{Que}\label{NHP}
If $G$ is a non-abelian $p$-group, the map $\psi^{p^j}_{G}$ may fail to be a homomorphism. Consequently, the functoriality of $B$ cannot be invoked to obtain a self-map of $BG$.
\end{Que}
In the theory of finite $p$-groups, a group $G$ is termed $p^j$-abelian if the $p^j$-th power map $\psi^{p^j}_G: G \rightarrow G$ is a homomorphism. This generalizes the classical fact that a $p$-group is abelian precisely when it is $2$-abelian, thereby offering a potential pathway to resolve Question \ref{NHP}.

The case of $p^j$-abelian groups prompts a generalization of the Hopkins--Kuhn--Ravenel definition of formal groups on a graded Hopf algebra (Definition \ref{FG}), to calculate the roots of $[p^j]_E(-)$ in $E^*(BG)$. Although the algebra structure on $E^*(BG)$ is needed to identify a homomorphism $f^* \in {\rm  Hom}_{E^*\text{-alg}}(E^*\psb{x}/([p^j]_E(x)), E^*(BG))$ with its image $f(x)$, the coalgebra structure can be weakened. The goal is to define a hom-set ${\rm Hom}_{?}(E^*\psb{x}/([p^j]_E(x)), E^*(BG))$ that still forms a root set of $[p^j]_E(-)$. Consider an $E^*$-algebra $R$ with a map $[p^j]_R(-): R \rightarrow R$. We require that any $f$ in this hom-set is an $E^*$-ring homomorphism satisfying $f([p^j]_E(x)) = [p^j]_R(f(x))$.

In the abelian case, the group ${\rm  Hom}_{E^*\text{-alg}}(E^*\psb{x}/([p^j]_E(x)),E^*(BG))$ is computed by Theorem \ref{fpjr} (based on \cite{LT65}) and is isomorphic to ${\rm Hom}(G,\mathbb{Z}/p^j)$. For a $p^j$-abelian $G$, it is easy to see that ${\rm Hom}(G,\mathbb{Z}/p^j)$ is a subset
of ${\rm Hom}_{?}(E^*\psb{x}/([p^j]_E(x)),E^*(BG))$, which leads to the following question
\begin{Que}
 Let $G$ be a finite $p^j$-abelian $p$-group. Is it true that
 $${\rm Hom}(G,\mathbb{Z}/p^j)= {\rm Hom}_{?}(E^*\psb{x}/([p^j]_E(x)),E^*(BG))?$$
\end{Que}
However, any attempt to generalize the theorem from \cite{LT65} to answer this question must confront the requirement that $E^*(BG)$ be a polynomial or power series algebra.

\begin{Conj}
Let $G$ be a finite $p^j$-abelian $p$-group and $E$ be a $p$-complete complex-oriented spectrum with an associated formal group of height $n$. Then the induced map $\psi^{p^j,*}_{BG}$, when restricted to $2$-dimensional Euler classes, has the power series expansion
$$\psi^{p^j,*}_{BG}(x)=v_0^jx+\cdots+v_{n}^{1+p^{n}+\cdots+p^{(j-1)n}}x^{p^{jn}}.$$
\end{Conj}
\begin{Rem}
This conjecture may be connected to Ando's results in \cite{And95}.
\end{Rem}

Our paper is $\mathbf{organized}$ as follows. In Section 2, we review the computation of the Zariski topology of Balmer spectrum $\Spc(\SH(G)^c)$ and this is our motivation to study the general blue-shift phenomenon; In Section 3, we calculate the homotopy group of the generalized Tate spectrum $\mathscr {T}_{A,C}(E)$; In Section 4, we prove Theorem \ref{Fac-tuple} and give two applications of Corollary \ref{vanish}; In Section 5, we recall the definition of algebraic periodicity and Landweber exactness for a spectrum; Note that Theorem \ref{Gtclbc} is a corollary of Theorem \ref{GTglbc}, we give a detailed proof of Theorem \ref{GTglbc} in Section 6.

\textbf{Acknowledgement}:
Firstly, I thank Professor Nick Kuhn for introducing me the problem of computing Balmer spectrum in the International Workshop on Algebraic Topology at Fudan University in 2019. Secondly, I thank Professor Stefan Schwede for teaching me lots of knowledge about the $G$-equivariant stable homotopy category. Thirdly, as most my work is based on my PhD thesis \cite{RYY21}, I thank Professor Xu-an Zhao for his carefully reading my PhD thesis and making me correct some vague arguments. Then I thank Professor John Rognes for sharing the origin of blue-shift terminology in algebraic topology. Finally, I also thank Professor Peter May, Zhouli Xu, Hana Jia Kong, Long Huang and Ran Wang for carefully reading my draft and suggesting lots of improvements.

\section{Towards computing the Zariski topology of $\Spc(\SH(G)^c)$ \label{Mot} }
Our work is motivated by computing the Zariski topology of Balmer spectrum, this leads us to Conjecture \ref{Gbsp} and Theorem \ref{Gtclbc}. So let us illustrate how Theorem \ref{Gtclbc} can be applied to compute the Balmer spectrum.

\subsection{Review of the computation of the Zariski topology of $\Spc(\SH(G)^c)$ }
The category $\SH(G)^c$ has a symmetric monoidal structure, where the tensor product is the smash product of $G$-spectra, and the unit object is the $G$-sphere spectrum $S_G$. This structure makes $\SH(G)^c$ resemble a commutative ring with a unit. Therefore, methods from algebraic geometry can be introduced, allowing us to define concepts like ``prime ideal" and ``spectrum" for this category. In 2005, Balmer \cite{Ba05} defined the spectrum $\Spc(\SH(G)^c)$, which is analogous to the spectrum of a commutative ring with a unit. It consists of all proper ``prime ideals" and is equipped with the Zariski topology. This spectrum is now known as the $\textit{Balmer}$ $\textit{spectrum}$. When the group $G$ is the trivial group $e$, the category $\SHG$ reduces to the classical stable homotopy category $\SH(e)$. Hopkins--Smith \cite{HS98} classified all thick subcategories of $\SH(e)^c$ by building on the work of Ravenel \cite{Ra84} and Mitchell \cite{Mi85}. In essence, they determined the Balmer spectrum $\Spc(\SH(e)^c)$. In this context, the proper ``prime ideals" of $\SH(e)^c$ are given by the thick subcategories
$$\cat C_{p,m}= \{ X \in \SH(e)^c \mid K(m-1)_{*}(X)=0\}$$
for primes $p$ and positive integers $m$, where $K(0)$ and $K(\infty)$ denote the rational and mod $p$ Eilenberg-Maclane spectra ($K(\mathbb{Q})$ and $K(\mathbb{Z}/p)$ respectively). For each prime $p$, there is a descending chain
\begin{equation*}
\cat C_{p,1}\supsetneq \cat C_{p,2}\supsetneq \cdots \supsetneq\cat C_{p,\infty}
\end{equation*}
due to \cite{Ra84,Mi85}. The topology space $\Spc(\SH(e)^c)$ can be
described by the following diagram:
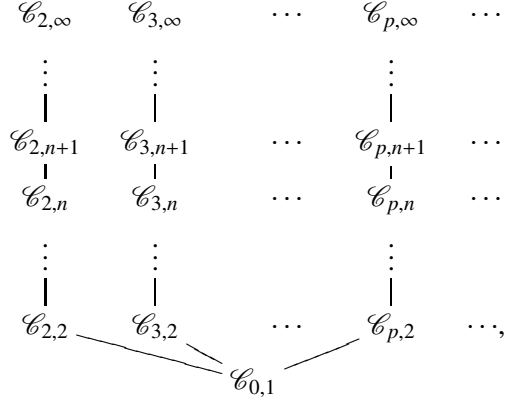
\begin{equation*}
\label{eq:Spc(SH)}
\vcenter{\xymatrix@C=.8em @R=.4em{
&&\cat C_{2,\infty} \ar@{-}[d]
&\cat C_{3,\infty} \ar@{-}[d]
&& \kern-2em{\cdots}
&\cat C_{p,\infty} \ar@{-}[d]
& {\cdots}
\\
&&{\vdots} \ar@{-}[d]
& {\vdots} \ar@{-}[d]
&&& {\vdots} \ar@{-}[d]
\\
&&\cat C_{2,n+1} \ar@{-}[d]
& \cat C_{3,n+1} \ar@{-}[d]
&& \kern-2em{\cdots}
& \cat C_{p,n+1} \ar@{-}[d]
& {\cdots}
\\
&&\cat C_{2,n} \ar@{-}[d]
& \cat C_{3,n} \ar@{-}[d]
&& \kern-2em{\cdots}
& \cat C_{p,n} \ar@{-}[d]
& {\cdots}
\\
&&{\vdots} \ar@{-}[d]
& {\vdots} \ar@{-}[d]
&&& {\vdots} \ar@{-}[d]
\\
&&\cat C_{2,2} \ar@{-}[rrd]
& \cat C_{3,2} \ar@{-}[rd]
&& \kern-2em{\cdots}
& \cat C_{p,2} \ar@{-}[lld]
& {\cdots},
\\
&&&& \cat C_{0,1}
\\\\
}}
\end{equation*}
where the line between any two points denotes that there is an inclusion relation between the two proper ``prime ideals".

The computation of $\Spc(\SH(e)^c)$ is one of the main tools used in applications of the nilpotence theorem of Devinatz--Hopkins--Smith \cite{DHS,HS98} to global questions in stable homotopy theory. Strickland \cite{Stri12} tried to generalize the non-equivariant case to the $G$-equivariant case. For any subgroup $H$ of a finite group $G$, Strickland employed the geometric $H$-fixed point functor $\Phi^H(-):\SH(G)\rightarrow \SH(e)$, which exhibits similarities to a ring homomorphism, to pull back $\cat C_{p,m}$ and hence obtained the $G$-equivariant proper ``prime ideals"
$$\cat P_G(H,p,m)=(\Phi^H)^{-1}(\cat C_{p,m})=\{ X \in \SHGc \mid K(m-1)_{*}\Phi^H(X)=0\}.$$
In 2017, Balmer--Sanders \cite[Theorem 4.9 and Theorem 4.14]{BS17} confirmed that all $G$-equivariant proper ``prime ideals" of $\SHGc$ are obtained in this manner, effectively determining the set structure of the Balmer spectrum $\Spc(\SH(G)^c)$. To compute the Zarisiki topology of
$\Spc(\SH(G)^c)$, it suffices to give an equivalent condition for any two proper ``prime ideals" $\cat P_G(K,q,l),\cat P_G(H,p,m)$ of $\SHGc$
to have an inclusion relation $\cat P_G(K,q,l)\subseteq\cat P_G(H,p,m)$. Balmer--Sanders \cite[Corollary 4.12 and Corollary 6.4]{BS17} derived two
necessary conditions for this inclusion: one is $p=q$; the other is that $K$ is a subgroup of $H$ up to $G$-conjucate, which is denoted by
$K\leq_G H$. Consequently, the determination of Zariski topology of $\Spc(\SHGc)$ can be reduced to the computation of the following number
$$\mathbf{\rm BS}_m(G;H,K):=\min\{l-m=i\in\mathbb{Z}\mid\ {\cat P_G}(K,p,l)\subseteq {\cat P_G}(H,p,m)\}.$$
An important observation made by Kuhn--Lloyd \cite{KL20} is that $l\geq m$. Therefore, it suffices to prove that for each $l<m$, there is a finite $G$-spectrum $X$ such that $X\in {\cat P_G}(K,p,l)$ and $X\notin {\cat P_G}(H,p,m)$. By Mitchell's work \cite{Mi85}, there exists a non-equivariant finite spectrum $Y$ such that $Y\in {\cat C}_{p,m}$ but $Y\notin {\cat C}_{p,m+1}$. Taking $X$ to be the $G$-spectrum $Y$ with the trivial $G$-action completes the proof.

To determine $\mathbf{\rm BS}_m(G;H,K)$, it would be helpful to gain some intuition for the inclusion relation ${\cat P_G}(K,p,l)\subseteq {\cat P_G}(H,p,m)$. From the descending chain
\begin{equation*}
\cat C_{p,1}\supsetneq \cat C_{p,2}\supsetneq \cdots \supsetneq\cat C_{p,\infty}
\end{equation*}
and the fact that $\Phi^K(X)\in\SH(e)^c$, we can deduce the following equivalence:
$$K(m-1)\otimes \Phi^K(X)=0 \Leftrightarrow \bigvee_{i=0}^{m-1}K(i)\otimes \Phi^K(X)=0.$$
To make this equation more convenient for analysis, let us recall a definition for any non-equivariant spectrum $E$ due to Bousfield \cite{Bo79}, where $\langle E\rangle$ denotes the equivalence class of $E$: $E \sim F$ if for any spectrum $X\in \SH(e)$, $E_*X = 0 \Leftrightarrow F_*X = 0$. And $\langle E\rangle$ is called $\textit{Bousfield}$ $\textit{class}$ of $E$.
Due to Ravenel \cite[Theorem 2.1]{Ra84}, the Bousfield class $\langle\bigvee_{i=0}^{n}K(i)\rangle$ equals to the Bousfield class $\langle E(n)\rangle$. Then we have for $X\in \SH(G)^c$,
$$\bigvee_{i=0}^{m-1}K(i)\otimes \Phi^K(X)=0 \Leftrightarrow E(m-1)\otimes \Phi^K(X)=0.$$
Thus for $X\in \SH(G)^c$,
$$K(m-1)\otimes \Phi^K(X)=0 \Leftrightarrow  E(m-1)\otimes \Phi^K(X)=0.$$
Hence ${\cat P_G}(K,p,l)\subseteq {\cat P_G}(H,p,m)$ is equivalent to the fact that for $X\in \SH(G)^c$, $E(l-1)_{*}\Phi^K(X)=0$ implies
$E(m-1)_{*}\Phi^H(X)=0$.

The inclusion $H \hookrightarrow G$ provides a $\textit{restriction}$ functor $\Res^G_H : \SH(G) \to \SH(H)$. Assume that $K \trianglelefteq G$, the surjective homomorphism $G\rightarrow G/K$ induces an $\textit{inflation}$ functor ${\rm inf}^G_{G/K}: \SH(G/K)\rightarrow \SHG$. Let $\tilde{\Phi}^{K}$ be the relative geometric $K$-fixed point functor from $\SHG$ to $\SH(G/K)$. By \cite[Chapter \textrm{II}. \S 9]{LMS}, we have $\Res^{G/K}_e {\mkern-9mu} \circ \tilde{\Phi}^{K} \cong \Phi^{K}$ and hence
\begin{align*}
0=E(l-1)\otimes \Phi^K(X) =E(l-1)\otimes \Res^{G/K}_{e}\circ\tilde{\Phi}^K(X)=\Res^{G/K}_{e}({\rm inf}^{G/K}_e(E(l-1))\otimes \tilde{\Phi}^K(X)).
\end{align*}
Let ${G/K}_+$ denote the disjoint union of the coset $G/K$ and a point. By \cite[1.1 Theorem]{BDS}, we get $\Res^{G/K}_{e}(-)\cong {G/K}_+\otimes(-)$ and hence
$$0=\Res^{G/K}_{e}({\rm inf}^{G/K}_e(E(l-1))\otimes \tilde{\Phi}^K(X))={G/K}_+\otimes {\rm inf}^{G/K}_e(E(l-1))\otimes \tilde{\Phi}^K(X).$$
Let $E(G/K)$ denote the Milnor construction, which is an infinite join $G/K*G/K*\cdots*G/K$, for the group $G/K$. Then
$$0=E(G/K)_+\otimes {\rm inf}^{G/K}_e(E(l-1))\otimes \tilde{\Phi}^K(X).$$
Let $\widetilde{E}(G/K)$ be the unreduced suspension of $E(G/K)$ with one of the cone points as basepoint, then we have
\begin{align}\label{Gco}
0=&F({\widetilde{E}(G/K)},\Sigma {E(G/K)}_+\otimes {\rm inf}^{G/K}_e(E(l-1))\otimes \tilde{\Phi}^K(X)).
\end{align}
By \cite[Corollary B.5]{Gr94}, we have
$$F({\widetilde{E}G},\Sigma {EG}_+\otimes -)\cong F(EG_+,-)\otimes \widetilde{E}G.$$
Actually $t_{G}(k_G):=F(EG_+,k_G)\otimes \widetilde{E}G$ is so-called $\textit{classical}$ $\textit{Tate}$ $\textit{construction}$ in the sense of Greenlees--May \cite{GM95} for a $G$-spectrum $k_G$. Assume that $K \trianglelefteq H$, we apply geometric $H/K$-fixed point functor ${\Phi}^{H/K}(-)$ to Formula \ref{Gco}. Since ${\Phi}^{H/K}(-)$ preserves weak equivalences, we obtain
$$0={\Phi}^{H/K}(t_{G/K}({\rm inf}^{G/K}_e(E(l-1))\otimes \tilde{\Phi}^K(X))).$$
Note that for $X\in \SHG$, $Y \in \SHGc$, $t_G(X)\otimes Y \cong t_G(X\otimes Y)$ (details see \cite[Remark 5.8]{BS17}), we have
$$0={\Phi}^{H/K}(t_{G/K}({\rm inf}^{G/K}_e(E(l-1)))\otimes \tilde{\Phi}^K(X)).$$
From the facts that for any $G/K$-spectra $X$ and $Y$, $\Phi^{H/K}(X\otimes Y)=\Phi^{H/K}(X)\otimes \Phi^{H/K}(Y),$ and $\Phi^{H/K} \circ \tilde{\Phi}^{K} \cong \Phi^{H}$, it follows that
\begin{align*}
0=&{\Phi}^{H/K}(t_{G/K}({\rm inf}^{G/K}_e(E(l-1)))\otimes \tilde{\Phi}^K(X))\\
=&{\Phi}^{H/K}(t_{G/K}({\rm inf}^{G/K}_e(E(l-1))))\otimes {\Phi}^{H/K}\circ\tilde{\Phi}^K(X)\\
=&{\Phi}^{H/K}(t_{G/K}({\rm inf}^{G/K}_e(E(l-1))))\otimes {\Phi}^{H}(X).
\end{align*}
For the sake of convenience, let $T_{G/K,H/K}(-)$ denote the functor $\Phi^{H/K}(t_{G/K}({\rm inf}^{G/K}_e(-)))$, and by Proposition \ref{Shgg} we have $T_{G/K,H/K}(-)=\cat {T}_{H/K,H/K}(-)$. If $\langle T_{G/K,H/K}(E(l-1))\rangle$ is equal to the Bousfield class of some Johnson-Wilson theory, this would give us an upper bound for $\mathbf{\rm BS}_m(G;H,K)$.

\subsection{Comparison between our new approach and the previous approach}
The idea of the above reduction is inspired by Balmer--Sanders' computation \cite[Proposition 7.1]{BS17} of the Zariski topology of the Balmer spectrum $\Spc(\SH(\mathbb{Z}/p)^c)$. They used the result from Hovey--Sadofsky \cite{HS96} and Kuhn \cite{Ku04}:
$$\langle T_{\mathbb{Z}/p,\mathbb{Z}/p}(E(l-1))\footnote{Actually their construction is $t_{\mathbb{Z}/p}({\rm inf}^{\mathbb{Z}/p}_{e}(-))^{\mathbb{Z}/p}$, but by Proposition \ref{Sfy} and Proposition \ref{Shgg}, $t_{\mathbb{Z}/p}({\rm inf}^{\mathbb{Z}/p}_{e}(-))^{\mathbb{Z}/p}$ and $T_{\mathbb{Z}/p,\mathbb{Z}/p}(-)$ are the same construction.}\rangle=\langle E(l-2)\rangle.$$
This result led them to conclude that $\mathbf{\rm BS}_m(\mathbb{Z}/p;\mathbb{Z}/p,e)\leq 1$. In fact, $\mathbf{\rm BS}_m(\mathbb{Z}/p;\mathbb{Z}/p,e)=1$, which means that the determination of $\langle T_{G/K,H/K}(E(l-1))\rangle$ might give us the least upper bound of $\mathbf{\rm BS}_m(G;H,K)$. If $H/K$ is a finite abelian $p$-group, then Theorem \ref{Gtclbc} confirms that
$$\langle T_{G/K,H/K}(E(l-1))\rangle=\langle E(l-1-{\rm rank}_p(H/K))\rangle.$$
In 2019, Barthel--Hausmann--Naumann--Nikolaus--Noel--Stapleton \cite{BHNNNS} showed that when $G$ is a finite abelian $p$-group, $\mathbf{\rm BS}_m(G;H,K)$ is exactly ${\rm rank}_p(H/K)$. Interestingly, they did not use the Bousfield class $\langle T_{G/K,H/K}(E(l-1))\rangle$ to determine the upper bound of $\mathbf{\rm BS}_m(G;H,K)$; instead, they employed the method \cite{MNN19} of derived defect base by recognizing $T_{G/K,H/K}(E(l-1))$ as suitable sections of the structure sheaf on a certain non-connective derived scheme. There must be some beautiful mathematics behind such an elegant result. In order to make this problem more approachable to a broader audience, we present a new approach that is by use of Theorem \ref{Gtclbc} to give an upper bound of $\mathbf{\rm BS}_m(G;H,K)$.

The earlier approach described in \cite{BHNNNS} uses the chromatic height, as defined in \cite[Definition 3.1]{BHNNNS}, of  $T_{G/K,H/K}(E(l-1))$ to establish an upper bound for $\mathbf{\rm BS}_m(G;H,K)$. In some respects, the chromatic height of  $T_{G/K,H/K}(E(l-1))$ in \cite{BHNNNS} serves a role similar to the periodicity of $T_{G/K,H/K}(E(l-1))$ in our case, albeit with differing definitions. Consequently, the primary challenge addressed in \cite{BHNNNS} lies in determining this chromatic height.

Despite similarities, there are several significant differences between our new approach and the earlier approach in \cite{BHNNNS}:
\begin{enumerate}
\item[{\rm (i)}] Uniqueness: the approach to determine the chromatic height of $T_{G/K,H/K}(E(l-1))$ in \cite{BHNNNS} is by directly analyzing some properties of $T_{G/K,H/K}(E(l-1))$, but our approach to determine the periodicity of $T_{G/K,H/K}(E(l-1))$ is by analyzing certain properties of $\pi_*(T_{G/K,H/K}(E(l-1)))$. We call these two kinds of properties geometric properties and algebraic properties. The authors in \cite{BHNNNS} used the results of \cite{MNN17,MNN19} to study these geometric properties. We also develop some new tools including Theorem \ref{Fac-tuple} to study these algebraic properties, and this is the uniqueness of our new approach.
\item[{\rm (ii)}]Conceptual clarity: our new approach offers a more intuitive and conceptual explanation of the general blue-shift phenomenon, leading to its successful establishment. This clarity can be particularly valuable when dealing with non-abelian groups $G$, where the behavior of $\mathbf{\rm BS}_m(G;H,K)$ is not fully known.
\item[{\rm (iii)}]Simplicity of tools used: in contrast to the derived algebraic geometry and the geometry of the stack of formal groups used in \cite{BHNNNS}, our approach relies on the use of some-tuple of the $p^j$-series in $\pi_*(T_{H/K,H/K}(E(l-1)))$ and standard linear algebra techniques. This makes our approach more accessible and easier to apply.
\end{enumerate}
Overall, our new approach provides a fresh perspective on the general blue-shift phenomenon and may bring more intuition to the challenging problem of determining $\mathbf{\rm BS}_m(G;H,K)$ for non-abelian groups.

\section{The homotopy groups $\pi_*(\cat {T}_{A,C}(E))$ and their maps \label{hggts}}
Follow the notion of \cite[Section 5]{HKR}, in this section we assume that $E$ is a complex oriented cohomology theory, particularly $p$-complete theory with an associated formal group of height $n$. In this context, the homotopy group of the classical Tate construction $t_{A}({\rm inf}^A_{e}(E))^A$ for any finite abelian $p$-group $A$ has been calculated in \cite{GS98}. Additionally, experts in the field have been aware of the homotopy group of the generalized Tate spectrum $\cat {T}_{A,C}(E)$ for several years. However, a version of this information that offers sufficiently detailed proofs has been absent. In the present section, we endeavor to furnish a comprehensive proof for Theorem \ref{bgcoh}.

It is worth noting that the functor $T_{G,N}(-)$ bears a connection to $\cat {T}_{G,N}(-)$, a relationship that is delineated by the following proposition.
\begin{Prop}\label{Shgg}
Let $G$ be a finite $p$-group or $T^m=\underbrace{U(1)\times\cdots\times U(1)}_m$ for any positive integer $m$, and $N$ be its normal subgroup. Then
$T_{G,N}(-)=\cat {T}_{N,N}(-).$
\end{Prop}
\begin{proof}
By definition, ${\Phi}^{N}(-)=\tilde{\Phi}^{N}\circ\Res^{G}_{N}(-)$, combining with
the fact that
$$\Res^{G}_{N}(t_{G}({\rm inf}^{G}_e(-))=t_{N}(\Res^{G}_{N}\circ{\rm inf}^{G}_e(-)=t_{N}({\rm inf}^{N}_e(-)),$$
details see \cite[Example 5. 18]{BS17}, we have ${\Phi}^{N}(t_{G}({\rm inf}^{G}_e(-)))= \cat T_{N,N}(-)$.
\end{proof}
To begin, let us revisit the definition provided in the work \cite{LMS} for the concept of the $\textit{relative}$ $\textit{geometric}$ $N$-$\textit{fixed}$ $\textit{point}$ functor, denoted as $\tilde{\Phi}^{N}(-)$, which maps from the category $\SH(G)$ to $\SH(G/N)$.
For a family $\mathcal{F}$ of subgroups of $G$ that is closed under $G$-conjugacy, a universal space $E\mathcal{F}$ is defined based on its
fixed point properties. Specifically, the space $E\mathcal{F}^K$ is contractible if $K\in\mathcal{F}$ and empty if $K\notin\mathcal{F}$. A map
$E\mathcal{F}_+\rightarrow S^0$ is induced by the mapping $E\mathcal{F}\rightarrow *$, and the cofiber of this map is denoted as
$\widetilde{E}{\mathcal{F}}$. Through the long exact sequence of non-equivariant homotopy groups derived from this cofiber sequence, it is
established that $\widetilde{E}{\mathcal{F}}^K$ is homotopy equivalent to $*$ if $K\in\mathcal{F}$ and $S^0$ if $K\notin\mathcal{F}$.
Consequently, it follows that $\widetilde{E}{\mathcal{F}_1}\otimes\widetilde{E}{\mathcal{F}_2}\simeq
\widetilde{E}(\mathcal{F}_1\cup\mathcal{F}_2)$, where $\simeq$ denotes a homotopy equivalence. Let ${\mathcal{F}}[N]$ represent the
family of subgroups of $G$ that do not contain $N$, and the definition of $\tilde{\Phi}^{N}(-)$ involves the construction
$(\widetilde{E}{\mathcal{F}}[N]\otimes(-))^N$. Here, $\widetilde{E}G$ refers to $\widetilde{E}{\mathcal{F}}$, where $\mathcal{F}$ denotes the
family of subgroups solely containing the trivial subgroup $\{e\}$.

To calculate $\pi_*(\cat T_{G,N}(E))$, we give it an equivalent description.
\begin{Prop}\label{Sfy}
Let $G$ be a finite $p$-group or $T^m$, and $N$ be its normal subgroup. Let $E$ be a non-equivariant spectrum. Then
$$\cat T_{G,N}(E)\simeq (\tilde{\Phi}^N(F(EG_+,{\rm inf}_e^G(E))))^{G/N}~~\text{and}~~\pi_*\cat (\cat T_{G,N}(E))\cong \pi^{G/N}_*(\tilde{\Phi}^N(F(EG_+,{\rm inf}_e^G(E)))),$$
where $G/N$-equivariant homotopy group is defined by a complete $G/N$-universe in the sense of Lewis--May--Steinberger \cite{LMS}.
If the family subgroups of $G$ which do not contain $N$ are $\{e\}$, then $\cat T_{G,N}(-)=t_{G}({\rm inf}^{G}_{e}(-))^{G}$.
\end{Prop}
\begin{proof}
Since $\widetilde{E}{\mathcal{F}}[N]\otimes \widetilde{E}G\simeq\widetilde{E}{\mathcal{F}}[N]$, we have
\begin{align*}
\cat T_{G,N}(E)=&(\tilde{\Phi}^N(t_G({\rm inf}_e^G(E))))^{G/N}\\
=&((\widetilde{E}{\mathcal{F}}[N]\otimes \widetilde{E}G\otimes F(EG_+,{\rm inf}_e^G(E)))^N)^{G/N}\\
\simeq&((\widetilde{E}{\mathcal{F}}[N]\otimes F(EG_+,{\rm inf}_e^G(E)))^N)^{G/N}=(\tilde{\Phi}^N(F(EG_+,{\rm inf}_e^G(E))))^{G/N}.
\end{align*}
By the adjunction $[S^n,(\tilde{\Phi}^N(F(EG_+,{\rm inf}_e^G(E))))^{G/N}]\cong [{\rm inf}_e^{G/N}(S^n),\tilde{\Phi}^N(F(EG_+,{\rm inf}_e^G(E)))]^{G/N}$, we
identify the homotopy group $\pi_*(\tilde{\Phi}^N(F(EG_+,{\rm inf}_e^G(E))))^{G/N}$ with the $G/N$-equivariant homotopy group $\pi^{G/N}_*(\tilde{\Phi}^N(F(EG_+,{\rm inf}_e^G(E))))$.

If the family subgroups of $G$ which do not contain $N$ are $\{e\}$, then $\widetilde{E}{\mathcal{F}}[N]=\widetilde{E}G$ and $\cat T_{G,N}(-)=t_{G}({\rm inf}^{G}_{e}(-))^{G}$.
\end{proof}

Consider a normal subgroup $N$ of the group $G$. In this context, the ensuing theorem, attributed to Costenoble, delineates how the relative geometric $N$-fixed point functor $\tilde{\Phi}^{N}(-)$ operates on the homotopy group.
\begin{Thm}(Costenoble, \cite[Chapter \RNum{2} Proposition 9.13]{LMS})\label{Cos}
Let $k_G$ be a ring $G$-spectrum and set $k_{G/N}=\tilde{\Phi}^{N}(k_G)$. Then for a finite $G/N$-CW spectrum
$X$, $k_{G/N}^*(X)$ is the localization of $k_G^*({\rm inf}_{G/N}^G(X))$ obtained by inverting the Euler classes $\chi_V \in k_G^V(S^0)$ of those representations $V$ of $G$ such that $V^N=0$.
\end{Thm}
Proposition \ref{Sfy} and Theorem \ref{Cos} combine to reveal that in order to calculate $\pi_*(\cat T_{G,N}(E))$, the key lies in computing $\pi^G_*(F(EG_+,{\rm inf}_e^G(E)))$. Once this is done, it is a matter of inverting the Euler classes $\chi_V \in F(EG_+,{\rm inf}_e^G(E))^{V}(S^0)$ corresponding to complex representations $V$ of $G$ where $V^N=0$.

Leveraging the equivariant suspension isomorphism, we establish a correspondence:
$$\chi_V \in F(EG_+,{\rm inf}_e^G(E))^{V}(S^0)\cong F(EG_+,{\rm inf}_e^G(E))^{|V|}(S^{|V|-V}),$$
with $|V|$ representing the real dimension of $V$.

Applying Theorem \ref{Cos} and making use of the observation below:
\begin{align*}
\pi^G_*(F(EG_+,{\rm inf}_e^G(E)))=&\pi_*(G/G_+\wedge S^0,F(EG_+,{\rm inf}_e^G(E)))^G\\
=&\pi_*(S^0,F(EG_+,{\rm inf}_e^G(E))^G)\\
\cong&\pi_*(BG_+,E)=E^*(BG_+),
\end{align*}
we successfully equate the $G$-equivariant homotopy group $\pi^G_*(F(EG_+,{\rm inf}_e^G(E)))$ with $E^*(BG_+)$. This identification provides a key insight into solving for $\pi_*(\cat T_{G,N}(E))$.

\subsection{The $E^*$-cohomology of the classifying space of a finite abelian $p$-group }
Recall that a ring spectrum $E$ is $\textit{complex}$ $\textit{oriented}$ if there exists an element $x\in E^2(\mathbb{C}P^{\infty})$ such that the image $i^*(x)$ of the map
$i^*:E^2(\mathbb{C}P^{\infty})\rightarrow E^2(\mathbb{C}P^1)$ induced by $i:S^2\cong\mathbb{C}P^1\hookrightarrow \mathbb{C}P^{\infty}$ is the
canonical generator of $E^2(S^2)\cong \pi_0E$. Such a class $x$ is called a $\textit{complex}$ $\textit{orientation}$ of $E$. The complex orientated $E$ with the multiplication map $\mu_{\mathbb{C}P^{\infty}}:\mathbb{C}P^{\infty}\times\mathbb{C}P^{\infty}\rightarrow \mathbb{C}P^{\infty}$ gives an associated formal group law $F$ over $E^*$:
$$x_1+_F x_2 = F(x_1,x_2) = \mu_{\mathbb{C}P^{\infty}}^*(x)\in  E^{*}(\mathbb{C}P^{\infty }\times \mathbb{C}P^{\infty })=E^{*}\psb{x_1,x_2}.$$
For any integer $m$, the $m$-series of $F$ is the formal power series $[m]_E(x)=\underbrace{x+_F x+_F \cdots+_F x}_m \in E^*\psb{x}$. This formal group law is classified by a ring homomorphism $f$ from the homotopy group $MU^*$ of the complex corbordism spectrum to $E^*$. If $E^*$ is a local ring with the maximal ideal $I$, then there are a quotient map $\pi: E^*\to E^*/I$ and a formal group law $F_0$ over $E^*/I$ which is classified by the ring homomorphism $\pi\circ f$. Let $\tilde{v}_n$ denote the coefficient of $x^{p^n}$ in $[p]_{F_0}(x)$. Say that $F_0$
\begin{enumerate}
\item[\rm (i)] has $\textit{height}$ $\textit{at}$ $\textit{least}$ $n$ if $\tilde{v}_i=0$ for $i<n$;
\item[\rm (ii)] has $\textit{height}$ $\textit{exactly}$ $n$ if it has height at least $n$ and $\tilde{v}_n$ is non-zero in $E^*/I$.
\end{enumerate}
When localized at $p$, such formal group laws are classified by height.

Now we introduce the Weierstrass Preparation Theorem.
\begin{Thm}(Weierstrass Preparation Theorem, \cite{Ma71, Lang78, ZS60})\label{WPT}
Let $R$ be a graded local commutative ring, complete in the topology defined by the powers of an ideal $\mathbf{m}$. Suppose
$$\alpha(x)=\sum_{i=0}^{\infty}a_ix^i\in R\psb{x}$$
satisfies $\alpha(x)\equiv a_nx^n \mod (\mathbf{m}, x^{n+1})$ with $a_n \in R$ a unit. Then
\begin{enumerate}
\item[{\rm (i)}] (Euclidean algorithm) Given $f(x) \in R\psb{x}$, there exist unique
$r(x) \in R[x]$ and $q(x) \in R\psb{x}$ such that $f(x) = r(x) + \alpha(x)q(x)$ with $\deg r(x)\leq n-1$.
\item[{\rm (ii)}]The ring $R\psb{x}/(\alpha(x))$ is a free R-module with basis $\{1, x, \cdots, x^{n-1}\}$.
\item[{\rm (iii)}] (Factorization) There is a unique factorization
$\alpha(x) = \varepsilon(x)g(x)$ with $\varepsilon(x)$ a unit and $g(x)$ a monic polynomial of degree $n$.
\end{enumerate}
\end{Thm}
We call $g(x)$ the $\textit{Weierstrass}$ $\textit{polynomial}$ of $\alpha(x)$. The number $n$ is called the $\textit{Weierstrass}$ $\textit{degree}$ of $\alpha(x)$ and denoted by $\deg_W \alpha(x)$.

Recall some basic properties of the associated formal group law $F$ over $E^*$.
\begin{Prop}\label{pps}
Let $E$ be a $p$-complete, complex oriented spectrum with an associated formal group of height $n$. Let $I_n$ denote the maximal ideal of $E^*$. Then for any integer $m$, the $m$-series of $F$ satisfies
\begin{enumerate}
\item[{\rm (i)}]   $[m]_E(x)\equiv mx \mod (x^2)$;
\item[{\rm (ii)}]  $[mk]_E(x)=[m]_E([k]_E(x))$;
\item[{\rm (iii)}]  $[p]_E(x)=v_nx^{p^n} \mod I_n$;
\item[{\rm (iv)}]  $[m-k]_E(x)=[m]_E(x)-_F [k]_E(x)=([m]_E(x)-[k]_E(x))\cdot \varepsilon([m]_E(x),[k]_E(x))$, where $\varepsilon([m]_E(x),[k]_E(x))$ is a unit in ${E^*\psb{x}}$.
 \end{enumerate}
 \end{Prop}
\begin{Lem}\label{pjwp}
Let $g_j(x)$ denote the Weierstrass polynomial of $[p^j]_{E}(x)$, and $g_1^j(x)=g_1(g_1^{j-1}(x))$. Then
$g_j(x)=g_1^j(x)$.
\end{Lem}
\begin{proof}
Suppose that $[p]_{E}(x)=px+a_2x^2+\cdots+a_{p^n-1}x^{p^n-1}+v_nx^{p^n}\mod (x^{p^n+1})$, and we apply Theorem \ref{WPT} to $[p]_{E}(x)\in E^*\psb{x}$, then $[p]_{E}(x)=\varepsilon(x)g_1(x)$ with $\varepsilon(x)$ a unit and $g_1(x)=px+a_2x^2+\cdots+a_{p^n-1}x^{p^n-1}+v_nx^{p^n}$. And we apply this theorem \ref{WPT} to $[p^j]_{E}(x)\in E^*\psb{x}$, by the fact that $[p^j]_{E}(x)=[p]_{E}([p^{j-1}]_{E}(x))$, then $[p^j]_{E}(x)=\varepsilon_j(x)g_j(x)$ with $\varepsilon_j(x)$ a unit. By the uniqueness of factorization \ref{WPT} and the fact that $g_1^j(x)=[p^j]_{E}(x)= v_{n}^{1+p^{n}+\cdots+p^{(j-1)n}}x^{p^{jn}}\mod I_n$, then $g_j(x)=g_1^j(x)$.
\end{proof}

The following lemma gives the computation of $E^*(BA_+)$.
\begin{Lem}\label{GCA}
Let $E$ be a $p$-complete, complex oriented spectrum with an associated formal group of height $n$. If $A$ is an abelian $p$-group of form $\mathbb{Z}/{p^{i_1}}\oplus \cdots \oplus \mathbb{Z}/{p^{i_m}}$, then
$$E^*(BA_+)\cong E^{*}\psb{x_1,\cdots,x_{m}}/([p^{i_1}]_{E}(x_1),\cdots, [p^{i_m}]_{E}(x_{m})).$$
\end{Lem}
\begin{proof}
If $A=\mathbb{Z}/p^j$, then there is a fiber sequence:
$$S^1 \to B\mathbb{Z}/p^j \to \mathbb{C}P^{\infty}\stackrel{\psi^{p^j}}\to \mathbb{C}P^{\infty}.$$
Note that the Euler class of the Gysin sequence of $S^1 \to B\mathbb{Z}/p^j \to \mathbb{C}P^{\infty}$ is $\psi^{p^j,2}(x)=[p^j]_E(x)\in E^{2}(\mathbb{C}P^{\infty }_+)$, then we have a long exact sequence:
$$\xymatrix@C=0.5cm{
  \cdots \ar[r] & E^*\psb{x}\ar[rr]^{\cup[p^j]_{E}(x)} && E^{*+2}\psb{x} \ar[rr] &&E^{*+2}(B\mathbb{Z}/p^j_+) \ar[r] & \cdots }.$$
Since $[p^j]_{E}(x)$ is not a zero divisor in $E^*\psb{x}$, the long exact sequence splits. Therefore, we obtain
$$E^*(B\mathbb{Z}/p^j_+)\cong E^*\psb{x}/([p^j]_E(x)).$$
As we all know, K$\ddot{\rm u}$nneth isomorphism is not always true for product spaces $X\times Y$, but if $E$-cohomology of the space $X$ or $Y$
is a finitely generated free module over $E^*$, the K$\ddot{\rm u}$nneth isomorphism is true. By Weierstrass Preparation Theorem \ref{WPT}, we have an $E^*$-ring isomorphism
$$\eta: E^*\psb{x}/([p^j]_{E}(x))\cong E^*[x]/(g_j(x))$$
that maps $f(x)$ to $r(x)$, where $g_j(x)$ is the Weierstrass polynomial of $[p^j]_E(x)$, which implies that $E^*\psb{x}/([p^j]_E(x))$ is a finite free $E^*$-module of rank $p^{jn}=\deg_W [p^j]_E(x)$. This finishes the proof.
\end{proof}

Note that $E^*(B\mathbb{Z}/p^j_+)$ is a Hopf algebra over $E^*$. And $\eta$ induces a coalgebra structure on $E^*[x]/(g_j(x))$ by the following commutative diagram:
$$\CD
  E^*\psb{x}/([p^j]_E(x)) @>\mu_{B\mathbb{Z}/p^j}^*>> E^*\psb{x}/([p^j]_E(x))\otimes_{E^*} E^*\psb{x}/([p^j]_E(x))\\
  @V\eta VV @V \eta\otimes\eta VV  \\
  E^*[x]/(g_j(x)) @>(\eta\otimes\eta)\circ\mu_{B\mathbb{Z}/p^j}^*\circ\eta^{-1}>> E^*[x]/(g_j(x))\otimes_{E^*} E^*[x]/(g_j(x)),
\endCD$$
then combining with Lemma \ref{pjwp}, we have
\begin{Prop}\label{idwp}
Let $E$ be a $p$-complete, complex oriented spectrum with an associated formal group of height $n$. Then there is an $E^*$-algebra isomorphism
$$\eta: E^*\psb{x}/([p^j]_{E}(x))\cong E^*[x]/(g_1^j(x)),$$
where the coalgebra structure on $E^*[x]/(g_1^j(x))$ is given by the map
$$\eta\circ\mu_{B\mathbb{Z}/p^j}^*\circ\eta^{-1}: E^*[x]/(g_1^j(x))\rightarrow E^*[x]/(g_1^j(x))\otimes_{E^*} E^*[x]/(g_1^j(x)).$$
\end{Prop}

\subsection{Euler classes and formal groups}
In this paper, we always identify $\mathbb{Z}/{p^{j}}$ with the set $\{0, 1, \cdots, p^j-1 \}$. Let $\rho_{\frac{w}{p^j}} : \mathbb{Z}/p^j\rightarrow U(1) $ denote the complex character that maps $h$ to $e^{\frac{2wh\pi {\rm i}}{p^j}}$ for $w\in \mathbb{Z}/p^j$. Suppose that $A$ has the form $\mathbb{Z}/{p^{i_1}}\oplus \cdots \oplus \mathbb{Z}/{p^{i_m}}$. By the representation theory of abelian groups \cite[Propositon 4.5.1]{Ste12}, $$\{\rho_{(\frac{w_1}{p^{i_1}},\cdots,\frac{w_m}{p^{i_m}})}=\mu_{U(1)}\circ(\rho_{\frac{w_1}{p^{i_1}}}\times\cdots\times\rho_{\frac{w_m}{p^{i_m}}})
=\rho_{\frac{w_1}{p^{i_1}}}\cdots\rho_{\frac{w_m}{p^{i_m}}}: A\rightarrow U(1) \mid (w_1,\cdots,w_m)\in A\}$$
formed all irreducible complex representations of $\mathbb{Z}/{p^{i_1}}\oplus \cdots \oplus \mathbb{Z}/{p^{i_m}}$.

Recall the definition \cite{GM95} of Euler classes for the $A$-spectrum $F(EA_+,{\rm inf}_e^A(E))$. Let $V$ be any complex $A$-representation  with an inner product, let $e_V : S^0 \rightarrow S^V$ send the non-basepoint to $0$, and let $\chi_V\in F(EA_+,{\rm inf}_e^A(E))^{V}(S^0)$ be the image of the unit of $F(EA_+,{\rm inf}_e^A(E))^{0}(S^0)$ under the map
$e_V^*:F(EA_+,{\rm inf}_e^A(E))^{0}(S^0)\cong F(EA_+,{\rm inf}_e^A(E))^{V}(S^V)\rightarrow F(EA_+,{\rm inf}_e^A(E))^{V}(S^0)$.

Since any finite abelian $p$-group $A$ with ${\rm rank}_p(A)=m$ is isomorphic to a subgroup of $T^m$, we first show how to specifically identify $E^*(BU(1)_+)\cong E^*\psb{x}$ with $\pi^{U(1)}_*(F(EU(1)_+,{\rm inf}_e^{U(1)}(E)))$.
%According to the group representation theory, we have a one-one correspondence between $G$-representations $\rho: G\rightarrow GL(V)$ and $G$-spaces $V$. In this paper, we do not distinguish this two notions.
Let $R$ denote the $U(1)$-spectrum $F(EU(1)_+,{\rm inf}_e^{U(1)}(E))$. We may assume that $E$ is a homotopy commutative ring spectrum, and by
\cite[Theorem 6.23]{BH15} $F(EU(1)_+,{\rm inf}_e^{U(1)}(E))$ is a homotopy commutative $U(1)$-ring spectrum. Firstly, recall the definition \cite[Definition 5.1]{MNN19} of the $\textit{Thom}$ $\textit{class}$ $\mu_{V}: S^{V-|V|}\rightarrow R$ for $V$ with respect to $R$, $\mu_{V}$ is a map of $U(1)$-spectra such that its canonical extension to an $R$-module map
$$\CD
  R\otimes S^{V-|V|}@>{\rm id}_R\otimes\mu_{V}>> R\otimes R@>\mu>> R
\endCD$$
is an equivalence, where $\mu$ denotes the multiplication map of the ring spectrum $R$. Secondly, we will find the Thom class $\mu_{V}$. Since all irreducible complex representations of abelian groups are complex one-dimensional, we may choose $V$ to be $\mathbb{C}$. For the principal $U(1)$-bundle $\mathbb{C}\to \mathbb{C} \to *$, we have a Thom space $S^{\mathbb{C}}$, which gives a Thom isomorphism
$$\phi_{\mathbb{C}}:F(EU(1)_+,{\rm inf}_e^{U(1)}(E))^*(S^0)\rightarrow F(EU(1)_+,{\rm inf}_e^{U(1)}(E))^{*+2}(S^{\mathbb{C}}),$$
by the equivariant suspension isomorphism, we can rewrite $\phi_{\mathbb{C}}$ as an isomorphism
$$\pi^{U(1)}_*(F(EU(1)_+,{\rm inf}_e^{U(1)}(E)))\cong \pi^{U(1)}_*(F(EU(1)_+,{\rm inf}_e^{U(1)}(E))\otimes S^{2-\mathbb{C}}).$$
By \cite[Remark 5.2]{MNN19}, this Thom isomorphism $\phi_{\mathbb{C}}$ gives rise to such a Thom class
$\mu_{\mathbb{C}}:S^{\mathbb{C}-2}\rightarrow F(EU(1)_+,{\rm inf}_e^{U(1)}(E))$ for $\mathbb{C}$ with respect to
$F(EU(1)_+,{\rm inf}_e^{U(1)}(E))$. Follow the notions of \cite[Remark 2.2]{GM97}, we also insist that
$\phi_{\mathbb{C}}(y)=y\cdot\mu_{\mathbb{C}}$ for all $y\in F(EU(1)_+,{\rm inf}_e^{U(1)}(E))^*(S^0)$. Since
$\chi_V:S^{-|V|}\stackrel{e_V}\to S^{V-|V|}\stackrel{\mu_{V}}\to F(EU(1)_+,{\rm inf}_e^{U(1)}(E))$, we have
$$\chi_{\mathbb{C}}=\phi_{\mathbb{C}}(e_{\mathbb{C}})=e_{\mathbb{C}}\cdot \mu_{\mathbb{C}}=e_{\mathbb{C}}^*(\mu_{\mathbb{C}}).$$
For the universal principal $U(1)$-bundle $U(1) \to EU(1) \to BU(1)$, we have a Thom space $MU(1)\simeq BU(1)$, which gives a Thom isomorphism
$\cup x:E^*(BU(1)_+)\rightarrow E^{*+2}(BU(1)_+)$, and it corresponds to $\cdot\chi_{\mathbb{C}}$ under the following identification
$$\CD
  F(EU(1)_+,{\rm inf}_e^{U(1)}(E))^*(S^0) @>\cdot\mu_{\mathbb{C}}>>F(EU(1)_+,{\rm inf}_e^{U(1)}(E))^{*+2}(S^{\mathbb{C}})\\
  @V \cong VV @V e_{\mathbb{C}}^* VV  \\
E^*(BU(1)_+)@>\cup x>> F(EU(1)_+,{\rm inf}_e^{U(1)}(E))^{*+2}(S^0)\cong E^{*+2}(BU(1)_+).
\endCD$$

Then $x$ corresponds to $\chi_{\mathbb{C}}$ under the isomorphism between $F(EU(1)_+,{\rm inf}_e^{U(1)}(E))^*(S^0)$ and $E^*(BU(1)_+)$.

\begin{Lem}
Let $\rho_{\frac{w}{p^j}}$ be an irreducible complex $\mathbb{Z}/{p^{j}}$-representation with $w\in \mathbb{Z}/{p^{j}}$. Let $\rho_{\frac{w}{p^j}}^{\#}$ be the map $F(EU(1)_+,{\rm inf}_e^{U(1)}(E))^{*}(S^0)\rightarrow F(E\mathbb{Z}/p^{j}_+,{\rm inf}_e^{\mathbb{Z}/{p^{j}}}(E))^{*}(S^0)$ induced by $\rho_{\frac{w}{p^j}}$. Then $B\rho_{\frac{w}{p^j}}^*(x)=[p^j]_E(x)$ corresponds to $\chi_{\rho_{\frac{w}{p^j}}}=\rho_{\frac{w}{p^j}}^{\#}(\mu_{\mathbb{C}})$ under the isomorphism between $\pi^{\mathbb{Z}/{p^{j}}}_*(F(E\mathbb{Z}/p^{j}_+,{\rm inf}_e^{\mathbb{Z}/p^{j}}(E)))$ and $E^*(B\mathbb{Z}/p^{j}_+)$.
\end{Lem}
\begin{proof}
We take $V$ to be $\mathbb{C}$ and identify the following two diagrams.
$$\xymatrix{
F(EU(1)_+,{\rm inf}_e^{U(1)}(E))^{*}(S^0)\ar[d]_{\cdot\chi_{\mathbb{C}}} \ar[r]^{\rho_{\frac{w}{p^j}}^{\#}} &F(E\mathbb{Z}/p^{j}_+,{\rm inf}_e^{\mathbb{Z}/{p^{j}}}(E))^{*}(S^0)\ar[d]_{\cdot\rho_{\frac{w}{p^j}}^{\#}(\chi_{\mathbb{C}})} \\
F(EU(1)_+,{\rm inf}_e^{U(1)}(E))^{*+2}(S^0) \ar[r]^{\rho_{\frac{w}{p^j}}^{\#}} & F(E\mathbb{Z}/p^{j}_+,{\rm inf}_e^{\mathbb{Z}/{p^{j}}}(E))^{*+2}(S^0),}
\xymatrix{
E^*(BU(1)_+)\ar[d]_{\cup x} \ar[r]^{B\rho_{\frac{w}{p^j}}^*} &E^*(B\mathbb{Z}/p^{j}_+)\ar[d]_{\cup B\rho_{\frac{w}{p^j}}^*(x)} \\
E^{*+2}(BU(1)_+)\ar[r]^{B\rho_{\frac{w}{p^j}}^{*+2}} & E^{*+2}(B\mathbb{Z}/p^{j}_+),}$$
which finishes the proof.
\end{proof}

\begin{Lem}\label{ideu}
Let $A$ be an abelian $p$-group of form $\mathbb{Z}/{p^{i_1}}\oplus \cdots \oplus \mathbb{Z}/{p^{i_m}}$ and $\rho_{(\frac{w_1}{p^{i_1}},\cdots,\frac{w_m}{p^{i_m}})}$ be an irreducible complex $A$-representation with $(w_1,\cdots,w_m)\in A$. Let $\rho_{(\frac{w_1}{p^{i_1}},\cdots,\frac{w_m}{p^{i_m}})}^{\#}$ be the map
$F(EU(1)_+,{\rm inf}_e^{U(1)}(E))^{*}(S^0)\rightarrow F(EA_+,{\rm inf}_e^A(E))^{*}(S^0)$ induced by $\rho_{(\frac{w_1}{p^{i_1}},\cdots,\frac{w_m}{p^{i_m}})}$. Then $B\rho_{(\frac{w_1}{p^{i_1}},\cdots,\frac{w_m}{p^{i_m}})}^*(x)=[w_1]_{E}(x_1)+_F \cdots +_F [w_m]_{E}(x_m)$, corresponds to $\chi_{\rho_{(\frac{w_1}{p^{i_1}},\cdots,\frac{w_m}{p^{i_m}})}}=\rho_{(\frac{w_1}{p^{i_1}},\cdots,\frac{w_m}{p^{i_m}})}^{\#}(\chi_{\mathbb{C}})$ under the isomorphism between $\pi^A_*(F(EA_+,{\rm inf}_e^A(E)))$ and $E^*(BA_+)$.
\end{Lem}
\begin{proof}
Since $\rho_{(\frac{w_1}{p^{i_1}},\cdots,\frac{w_m}{p^{i_m}})}:A\rightarrow U(1)$ is the composition map
$$\CD
 \mathbb{Z}/{p^{i_1}}\oplus \cdots \oplus \mathbb{Z}/{p^{i_m}} @>\rho_{\frac{w_1}{p^{i_1}}}\times\cdots\times\rho_{\frac{w_m}{p^{i_m}}}>> T^m @>\mu^m_{U(1)}>> U(1),
\endCD$$
where $\mu^m_{U(1)}$ denotes the $m$-th composition of the multiplication map of $U(1)$. This map induces the composition of $E^*$-algebra homomorphisms
$$\CD
E^*(BU(1)_+)@>B\mu_{U(1)}^{m,*}>> E^*(BT^m_+) @>B(\rho_{\frac{w_1}{p^{i_1}}}\times\cdots\times\rho_{\frac{w_m}{p^{i_m}}})^*>> E^*(BA_+).
\endCD$$
Note that $B\mu_{U(1)}^{m,*}(x)=x_1+_F\cdots+_F x_m$, then we have
\begin{align*}
B\rho_{(\frac{w_1}{p^{i_1}},\cdots,\frac{w_m}{p^{i_m}})}^*(x)&=B(\rho_{\frac{w_1}{p^{i_1}}}\times\cdots\times\rho_{\frac{w_m}{p^{i_m}}})^*\circ B\mu_{U(1)}^{m,*}(x)\\
&=B(\rho_{\frac{w_1}{p^{i_1}}}\times\cdots\times\rho_{\frac{w_m}{p^{i_m}}})^*(x_1+_F\cdots+_F x_m)\\
&=[w_1]_{E}(x_1)+_F \cdots +_F [w_m]_{E}(x_m).
\end{align*}
This finishes the proof.
\end{proof}

\begin{Thm}(Lubin--Tate, \cite{LT65})\label{LT}
For each integer $k$ and each nature number $j$, there exists a unique series $[k]_E(x)\in E^*\psb{x}$ such that
$$[k]_E(x)\equiv kx \mod (x^2)~~\text{and}~~[k]_E([p^j]_E(x))=[p^j]_E([k]_E(x)).$$
\end{Thm}

For convenience, we denote $[w_1]_{E}(x_1)+_F \cdots +_F [w_m]_{E}(x_m)$ by $\alpha_{(w_1,\cdots,w_m)}$.
\begin{Lem}\label{setpjr}
Let $j$ be a nature number and $E$ be a $p$-complete, complex oriented spectrum with an associated formal group of height $n$. If $A$ is a finite abelian $p$-group of form $\mathbb{Z}/{p^{i_1}}\oplus \cdots \oplus \mathbb{Z}/{p^{i_m}}$, then there is a bijection
\begin{align*}
\omega:{}_{p^{j}\!}F(E^*(BA_+))&\rightarrow\{\alpha_{(w_1,\cdots,w_m)}\in E^*(BA_+)\mid(p^jw_1,\cdots,p^jw_m)=0, (w_1,\cdots,w_m)\in A\}\\
f^*&\mapsto \omega(f^*)=f^*(x).
\end{align*}
\end{Lem}
\begin{proof}
First suppose that $A=\mathbb{Z}/{p^{i}}$. For
$$f^*\in{}_{p^{j}\!}F(E^*(B\mathbb{Z}/p^{i}_+))={\rm  Hom}_{E^*\text{-alg}}(E^*\psb{x}/([p^j]_E(x)),E^*(B\mathbb{Z}/p^{i}_+)),$$
we can identify $f^*$ with $f^*(x)$ since $f^*$ is an $E^*$-ring homomorphism, which means that $\omega$ is injective. Then we have to prove that $\omega$ is well-defined, namely
$$f^*(x)\in\{\alpha_{(w_1,\cdots,w_m)}\in E^*(BA_+)\mid(p^jw_1,\cdots,p^jw_m)=0, (w_1,\cdots,w_m)\in A\}.$$
As $f^*$ is a graded $E^*$-algebra homomorphism and $\deg x=2$, we have
$$0=f^*([p^j]_E(x))=[p^j]_E(f^*(x))\in E^2(B\mathbb{Z}/p^{i}_+)\cong E^2\psb{x}/([p^i]_E(x)).$$
Notice that $[p^j]_E(x)\equiv p^jx \mod (x^2)$, then the constant term of $f^*(x)$ must be zero. Since $f^*(x)\in E^2(B\mathbb{Z}/p^{i}_+)$, we may suppose that $f^*(x)\equiv kx\mod (x^2)$, and by Lubin and Tate's theorem \ref{LT}, we have $f^*(x)=[k]_E(x)$. By the property that $[n_1]_{E}([n_2]_{E}(x))=[n_1n_2]_{E}(x)$, we have $[p^j]_E([k]_E(x))=[kp^j]_E(x)$. Then $f^*\in{\rm  Hom}_{E^*\text{-alg}}(E^*\psb{x}/([p^j]_E(x)),E^*(B\mathbb{Z}/p^{i}_+))$ implies that
$$f^*(x)\in\{[w]_E(x)\in E^2\psb{x}/([p^i]_E(x)) \mid p^jw=0, w\in \mathbb{Z}/p^{i} \},$$
so $\omega$ is well-defined. Note that for each $[w]_E(x)\in E^2\psb{x}/([p^i]_E(x))$ with $p^jw=0$, there is a group homomorphism $\rho_w:\mathbb{Z}/p^{i}\rightarrow \mathbb{Z}/p^{j}$ that maps $1$ to $w$ and $B\rho_{w}^*(x)=[w]_E(x)$, so $B\rho_{w}^*$ is an $E^*$-algebra homomorphism, so $\omega$ is surjective. Therefore, $\omega$ is a well-defined bijection.

For $A=\mathbb{Z}/{p^{i_1}}\oplus \cdots \oplus \mathbb{Z}/{p^{i_m}}$, there are group inclusions $\iota_k: \mathbb{Z}/{p^{i_k}}\rightarrow A$ that maps $w\in \mathbb{Z}/{p^{i_k}}$ to $(0,\cdots,0,w,0,\cdots,0)\in\mathbb{Z}/{p^{i_1}}\oplus \cdots \oplus\mathbb{Z}/{p^{i_{k-1}}}\oplus\mathbb{Z}/{p^{i_k}}\oplus\mathbb{Z}/{p^{i_{k-1}}}\oplus \cdots\oplus \mathbb{Z}/{p^{i_m}}$. By Lemma \ref{GCA}, we have
$$E^*(BA_+)\cong E^{*}\psb{x_1}/([p^{i_1}]_{E}(x_1))\otimes_{E^*}\cdots\otimes_{E^*}E^{*}\psb{x_m}/([p^{i_m}]_{E}(x_m)).$$
There is an isomorphism:
\begin{align*}
{\rm  Hom}_{E^*\text{-alg}}(E^*\psb{x}/([p^j]_E(x)),E^*(BA_+))&\rightarrow \bigotimes^{m}_{k=1}{\rm  Hom}_{E^*\text{-alg}}(E^*\psb{x}/([p^j]_E(x)),E^{*}\psb{x_1}/([p^{i_k}]_{E}(x_k)))\\
f^*&\mapsto B\iota_1^*\circ f^*\otimes\cdots\otimes B\iota_m^*\circ f^*.
\end{align*}
We can identify $f^*\in{\rm  Hom}_{E^*\text{-alg}}(E^*\psb{x}/([p^j]_E(x)),E^*(BA_+))$ with $f^*(x)\in E^2(BA_+)$. Then the rest proof is similar to the case of $A=\mathbb{Z}/{p^{i}}$, we omit it here.
\end{proof}

\begin{Lem}\label{chm}
Let $A$ be a finite abelian $p$-group. If $G$ is a finite abelian $p$-group or $U(1)$, then the map $E^*(B(-)):{\rm Hom}(A,G)\rightarrow {\rm  Hom}_{E^*\text{-alg}}(E^*(BG_+),E^*(BA_+))$
defined by $f\mapsto E^*(Bf)=Bf^*$ is a group isomorphism.
\end{Lem}
\begin{proof}
By Lemma \ref{setpjr}, it is easy to check that $E^*(B(-))$ is a bijection. Then the remaining thing is to prove that $E^*(B(-))$ is a group homomorphism.
Let $[BA_+,BG_+]$ denote the homotopy class from $BA_+$ to $BG_+$. Since $G$ is abelian, we have ${\rm Hom}(A,G)/{\rm Inn}G={\rm Hom}(A,G)$. Note that $A$ is a finite abelian $p$-group, by Dwyer--Zabrodsky's Theorem \cite{DZ87} or Notbohm's Theorem \cite{No91}, there is a bijection
 \begin{align*}
B:{\rm Hom}(A,G)&\rightarrow[BA_+,BG_+]\\
\rho &\mapsto B\rho.
\end{align*}
For a topological space $X$, let $\Delta_X$ denote the diagonal map $X\rightarrow X\times X$, then for any $\rho_1, \rho_2\in {\rm Hom}(A,G)$, there are products $\mu_{G}\circ(\rho_1\times\rho_2)\circ\Delta_A $ and $\mu_{BG}\circ(B\rho_1\times B\rho_2)\circ\Delta_{BA} $. By the functorial property of $B$, $B$ preserves the product, namely
$$B(\mu_{G}\circ(\rho_1\times\rho_2)\circ\Delta_A)=\mu_{BG}\circ(B\rho_1\times B\rho_2)\circ\Delta_{BA}.$$
Similarly, By the functorial property of $E^*(-)$, $E^*(-)$ preserves the product, namely
$$E^*(\mu_{BG}\circ(B\rho_1\times B\rho_2)\circ\Delta_{BA})=\Delta_{BA}^*\circ(B\rho_1\times B\rho_2)^*\circ\mu_{BG}^*.$$
This finishes our proof.
\end{proof}

By Lemma \ref{setpjr} and Lemma \ref{chm}, we have
\begin{Thm}\label{fpjr}
Let $j$ be a nature number and $E$ be a $p$-complete, complex oriented spectrum with an associated formal group of height $n$. If $A$ is a finite abelian $p$-group, then there are group isomorphisms
\begin{align*}
{}_{p^{j}\!}F(E^*(BA_+))&\cong\{\alpha_{(w_1,\cdots,w_m)}\in E^*(BA_+)\mid(p^jw_1,\cdots,p^jw_m)=0, (w_1,\cdots,w_m)\in A\}\\
                       &\cong {\rm Hom}(A,\mathbb{Z}/p^{j})\cong V(p^j|A).
\end{align*}
Furthermore,
\begin{align*}
{}_{p^{\infty}\!}F(E^*(BA_+))\cong {\rm Hom}(A,U(1))\cong A.
\end{align*}
\end{Thm}

\subsection{Maps between $E^*$-cohomology of classifying spaces}
Let $A_1$ and $A_2$ be two abelian $p$-groups $\mathbb{Z}/{p^{i_1}}\oplus \cdots \oplus \mathbb{Z}/{p^{i_{m}}}$ and $\mathbb{Z}/{p^{j_1}}\oplus \cdots \oplus \mathbb{Z}/{p^{j_{k}}}$. Then any homomorphism $h\in{\rm Hom}(A_1,A_2)$ is determined by an integer $m\times k$-matrix
$H\in M_{m\times k}(\mathbf{Z}_{(p)})$. Since each nature number $i$ can be identified with a self-map of $U(1)$ of degree $i$, $H$ can be identified with a map from $T^m$ to $T^k$, and there are two commutative diagrams:
$$\CD
 A_1 @>\rho_{\frac{1}{p^{i_1}}}\times\cdots\times\rho_{\frac{1}{p^{i_m}}}>> T^m \\
  @V h VV @V H VV  \\
 A_2 @>\rho_{\frac{1}{p^{j_1}}}\times\cdots\times\rho_{\frac{1}{p^{j_k}}}>> T^k,
\endCD~~~~~~
\CD
 BA_1 @>B(\rho_{\frac{1}{p^{i_1}}}\times\cdots\times\rho_{\frac{1}{p^{i_m}}})>> BT^m \\
  @V Bh VV @V BH VV  \\
 BA_2 @>B(\rho_{\frac{1}{p^{j_1}}}\times\cdots\times\rho_{\frac{1}{p^{j_k}}})>> BT^k.
\endCD$$
Besides $A_1$ and $A_2$ are associated with the following two fibrations
$$\CD
 T^m/A_1\cong T^{m} @>>>BA_1@>B(\rho_{\frac{1}{p^{i_1}}}\times\cdots\times\rho_{\frac{1}{p^{i_m}}})>>BT^m,
\endCD~~\CD
 T^k/A_2\cong T^{k} @>>>BA_2@>B(\rho_{\frac{1}{p^{j_1}}}\times\cdots\times\rho_{\frac{1}{p^{j_k}}})>>BT^k.
\endCD$$

\begin{Lem}
Let $E$ be a $p$-complete, complex oriented spectrum with an associated formal group of height $n$. Then
there is a Leray-Serre spectral sequence of $T^m \to ET^m \to BT^m$ with the $E_2$-page $H^s(BT^m; E^t(T^m))\cong H^s(BT^m;\mathbb{Z}/p)\otimes E^t(T^m)\cong\mathbb{Z}/p\psb{x_1,x_2,\cdots,x_m}\otimes\wedge_{E^*}[y_1,y_2,\cdots,y_m]$, and its only nontrivial differential is $d_2(1\otimes y_i)=x_i$ for $1\leq i\leq m$, which implies that it collapses at $E_3$-page.
\end{Lem}
\begin{proof}
Since $ET^m$ is contractible, then the only possible differential is $d_2(1\otimes y_i)=x_i$ for $1\leq i\leq m$.
\end{proof}

\begin{Lem}
Let $E$ be a $p$-complete, complex oriented spectrum with an associated formal group of height $n$. Then
there is a Leray-Serre spectral sequences of $T^m \to BA_1 \to BT^m $ with the $E_2$-page $H^s(BT^m; E^t(T^m))\cong H^s(BT^m;\mathbb{Z}/p)\otimes E^t(T^m)\cong\mathbb{Z}/p\psb{x_1,x_2,\cdots,x_m}\otimes\wedge_{E^*}[y_1,y_2,\cdots,y_m]$, and its only nontrivial differential is $d_2(1\otimes y_i)=[p^{i_j}]_{E}(x_j)$ for $1\leq j\leq m$, which implies that it collapses at $E_3$-page.
\end{Lem}
\begin{proof}
The following commutative diagram
$$\CD
 BA_1 @>B(\rho_{\frac{1}{p^{i_1}}}\times\cdots\times\rho_{\frac{1}{p^{i_m}}})>> BT^m \\
  @V  VV @V 1_{BT^m} VV  \\
 ET^m @>>> BT^m
\endCD$$
induces a map of Leray-Serre spectral sequences, which gives differentials $d_2(1\otimes y_i)=[p^{i_j}]_{E}(x_j)$ for $1\leq j\leq m$. Then by Lemma \ref{GCA}, we conclude that it collapses at $E_3$-page.
\end{proof}

\begin{Thm}\label{Mapab}
Let $E$ be a $p$-complete, complex oriented spectrum with an associated formal group of height $n$. Let $A_1$ and $A_2$ be two abelian $p$-groups $\mathbb{Z}/{p^{i_1}}\oplus \cdots \oplus \mathbb{Z}/{p^{i_{m}}}$ and $\mathbb{Z}/{p^{j_1}}\oplus \cdots \oplus \mathbb{Z}/{p^{j_{k}}}$. Then any abelian group homomorphism $h\in{\rm Hom}(A_1,A_2)$ is determined by an integer $m\times k$-matrix
$H\in M_{m\times k}(\mathbf{Z}_{(p)})$, and the homomorphism $Bh^*: E^*({BA_2}_+)\rightarrow E^*({BA_1}_+)$ can be identified with the $E_3$-page map of Leray-Serre spectral sequences for two associated fibrations
$$\CD
 T^m/A_1\cong T^{m} @>>>BA_1@>B(\rho_{\frac{1}{p^{i_1}}}\times\cdots\times\rho_{\frac{1}{p^{i_m}}})>>BT^m,
\endCD~~\CD
 T^k/A_2\cong T^{k} @>>>BA_2@>B(\rho_{\frac{1}{p^{j_1}}}\times\cdots\times\rho_{\frac{1}{p^{j_k}}})>>BT^k.
\endCD$$
where the map of these two fibrations is given by the following commutative diagram:
$$\CD
 BA_1 @>B(\rho_{\frac{1}{p^{i_1}}}\times\cdots\times\rho_{\frac{1}{p^{i_m}}})>> BT^m \\
  @V Bh VV @V BH VV  \\
 BA_2 @>B(\rho_{\frac{1}{p^{j_1}}}\times\cdots\times\rho_{\frac{1}{p^{j_k}}})>> BT^k.
\endCD$$
\end{Thm}

\subsection{The homotopy groups $\pi_*(\cat {T}_{A,C}(E))$}
The following lemma determines all complex representations $V$ of a finite abelian $p$-group $A$ such that $V^C=0$ for any subgroup $C$ of $A$.
\begin{Lem}\label{Lmcs}
Let $A$ be an abelian group of form $\mathbb{Z}/{p^{i_1}}\oplus \cdots \oplus \mathbb{Z}/{p^{i_m}}$ and $C$ be its subgroup $\mathbb{Z}/{p^{j_1}}\oplus \cdots \oplus \mathbb{Z}/{p^{j_m}}$ with a group inclusion
\begin{align*}
\varphi:\mathbb{Z}/{p^{j_1}}\oplus \cdots \oplus \mathbb{Z}/{p^{j_m}}&\rightarrow \mathbb{Z}/{p^{i_1}}\oplus \cdots \oplus \mathbb{Z}/{p^{i_m}}\\
(w_1,\cdots,w_k)&\mapsto (p^{i_1-j_1}w_1,\cdots, p^{i_m-j_m}w_m).
\end{align*}
There is a group homomorphism from $A/C$ to $A$ as follows:
\begin{align*}
\phi:\mathbb{Z}/{p^{i_1-j_1}}\oplus \cdots \oplus \mathbb{Z}/{p^{i_m-j_m}}&\rightarrow \mathbb{Z}/{p^{i_1}}\oplus \cdots \oplus \mathbb{Z}/{p^{i_m}}\\
(w_1,\cdots,w_m)&\mapsto (p^{j_1}w_1,\cdots, p^{j_m}w_m).
\end{align*}
 Then
$$\{\rho_{(\frac{w_1}{p^{i_1}},\cdots,\frac{w_m}{p^{i_m}})}=\rho_{\frac{w_1}{p^{i_1}}}\cdots\rho_{\frac{w_m}{p^{i_m}}}: A\rightarrow U(1) \mid (w_1,\cdots,w_m)\in A-{\rm im}\phi(A/C)\}$$
 forms all irreducible complex representations $V$ of $A$ such that $V^{C}=0$.
\end{Lem}
\begin{proof}
Note that
$$\{\rho_{(\frac{w_1}{p^{i_1}},\cdots,\frac{w_m}{p^{i_m}})}: A\rightarrow U(1) \mid (w_1,\cdots,w_m)\in A\}$$
 formed all irreducible complex representations of $A$. Then for any $(u_1,\cdots,u_m)\in C$, we have
\begin{align*}
\rho_{(\frac{w_1}{p^{i_1}},\cdots,\frac{w_m}{p^{i_m}})}(\varphi(u_1,\cdots,u_m))&=\rho_{\frac{w_1}{p^{i_1}}}(p^{i_1-j_1}u_1)
\cdots\rho_{\frac{w_m}{p^{i_m}}}(p^{i_m-j_m}u_m)\\
&=e^{2\pi {\rm i}(\frac{w_1u_1 }{p^{j_1}}+\cdots+\frac{w_mu_m }{p^{j_m}})}\\
&=\left\{\begin{array}{ll}
1&~~\text{if $p^{j_1}|w_1, \cdots, p^{j_m}|w_m $},\\
{\rm nonconstant}\,&
~~\text{Otherwise}.
\end{array}\right.
\end{align*}
And $p^{j_1}|w_1,\cdots, p^{j_m}|w_m \Leftrightarrow (w_1, \cdots, w_m)\in {\rm im}\phi(A/C)$.
\end{proof}

Now, we calculate the homotopy group of the generalized Tate spectrum $\cat T_{A,C}(E)$.
\begin{Thm}\label{bgcoh}
Let $m$ be a positive integer and $E$ be a $p$-complete, complex oriented spectrum with an associated formal group of height $n$. Let $A$ be an abelian $p$-group of form $\mathbb{Z}/{p^{i_1}}\oplus \cdots \oplus \mathbb{Z}/{p^{i_m}}$ and $C$ be its subgroup $\mathbb{Z}/{p^{j_1}}\oplus \cdots \oplus \mathbb{Z}/{p^{j_m}}$ with $j_k\leq i_k$ for $1\leq k\leq m$. There is a group homomorphism $\phi$ from $A/C$ to $A$ as follows:
\begin{align*}
\phi:\mathbb{Z}/{p^{i_1-j_1}}\oplus \mathbb{Z}/{p^{i_2-j_2}}\oplus \cdots \oplus \mathbb{Z}/{p^{i_m-j_m}}&\rightarrow \mathbb{Z}/{p^{i_1}}\oplus\mathbb{Z}/{p^{i_2}}\oplus\cdots\oplus \mathbb{Z}/{p^{i_m}}\\
(w_1,w_2,\cdots,w_m)&\mapsto (p^{i_1-j_1}w_1,p^{i_2-j_2}w_2,\cdots, p^{i_m-j_m}w_m).
\end{align*}
Then
$$\pi_*(\cat T_{A,C}(E))\cong L_C^{-1}E^*\psb{x_1,\cdots,x_{m}}/([p^{i_1}]_{E}(x_1),\cdots,[p^{i_m}]_{E}(x_m)),$$
where the multiplicatively closed set $L_C$ is generated by the set
\begin{align*}
M_C=\{\alpha_{(w_1,\cdots,w_m)}=[w_1]_{E}(x_1)+_F \cdots +_F [w_m]_{E}(x_m)\in E^*(BA_+)\mid(w_1,\cdots,w_m)\in A-{\rm im}\phi(A/C)\}.
\end{align*}
\end{Thm}

\begin{proof}
From Theorem \ref{Cos}, it follows that $\pi_{*}(\cat T_{A,C}(E))$ is the localization of
$\pi_*(F(EA_+,{\rm inf}_e^A(E)))\cong E^*(BA_+)$ obtained by inverting the Euler classes $\chi_V \in F(EA_+,{\rm inf}_e^A(E))^{|V|}(S^{|V|-V})$ of those complex representations $V$ of $A$ such that $V^C=0$.
By Theorem \ref{GCA}, we have
$$E^*(BA_+)\cong E^*\psb{x_1,\cdots,x_{m}}/([p^{i_1}]_{E}(x_1),\cdots,[p^{i_m}]_{E}(x_m)).$$
By Lemma \ref{Lmcs}, we have $\{\rho_{(\frac{w_1}{p^{i_1}},\cdots,\frac{w_m}{p^{i_m}})}: A\rightarrow U(1) \mid (w_1,\cdots,w_m)\in A-{\rm im}\phi(A/C)\}$ forms all irreducible complex representations $V$ of $A$ such that $V^{C}=0$. Each representation $\rho_{(\frac{w_1}{p^{i_1}},\cdots,\frac{w_m}{p^{i_m}})}: A\rightarrow U(1)$ induces a homormorphism $B\rho_{(\frac{w_1}{p^{i_1}},\cdots,\frac{w_m}{p^{i_m}})}^*: E^*(BU(1)_+)\cong E^*\psb{x}\rightarrow E^*(BA_+)$, and by Lemma \ref{ideu}, the image $B\rho_{(\frac{w_1}{p^{i_1}},\cdots,\frac{w_m}{p^{i_m}})}^*(x)$ is
the Euler class $[w_1]_{E}(x_1)+_F \cdots +_F [w_m]_{E}(x_m)=\alpha_{(w_1,\cdots,w_m)}$.
\end{proof}

\section{Generalized relations between roots and coefficients of a polynomial \label{GRRC}}
 In this section, we prove Theorem \ref{Fac-tuple} and give two applications of Corollary \ref{vanish}. Recall the definition of the root of a polynomial $f(x)$ over a commutative ring $R$. A polynomial $f(x)$ in $R[x]$ can be interpreted as a polynomial map from $R$ to $R$, where it maps $r\in R$ to $f(r)\in R$. We denote the set of such polynomial maps as ${\rm Pmap}(R,R)$. More precisely, ${\rm Pmap}(R,R)$ is the quotient $R[x]/\sim$. Let $[f(x)]$ represent the equivalence class of $f(x)$, such that $f(x)\sim g(x)$ if for any $r\in R$, $f(r)=g(r)$. An element $r\in R$ is then referred to as a $\it{root}$ of $f(x)\in R[x]$ if $r$ is a zero of the polynomial map $[f(x)]$, i.e., $f(r)=0$.  It is worth noting that if two polynomials are equal, their corresponding polynomial maps must also be equal. However, the converse may not be true, as exemplified by $x^2+x\in \mathbb{F}_2[x]$, which is unequal to $0$ as a polynomial over $\mathbb{F}_2$, but equals $0$ as a polynomial map from $\mathbb{F}_2$ to itself. So there are a map $\lambda:R[x]\rightarrow{\rm Pmap}(R,R)$ with $\lambda(f(x))=[f(x)]$ for $f(x)\in R[x]$ and a question that what conditions does $R$ satisfy with such that $\lambda$ is injective. To serve our purpose here, we restrict ourself to a narrow version of this question. Let $R[x]_n$ denote the set of polynomials of degree at most $n$ and $\lambda_{R[x]_n}$ denote the map that restricts $\lambda$ to $R[x]_n$, then what conditions does $R$ satisfy with such that $\lambda_{R[x]_n}$ is injective?

To give a sufficient condition, we take a fresh look at the equality $f(r)=0$ induced by a root $r\in R$ of $f(x)$. Without loss of generality, we may suppose that $f(x)=a_0+a_1x+\cdots+a_nx^n$ with $a_0, a_1, \cdots, a_n\in R$. $f(r)=0$ means that the ``$R$-vector" $(a_0,a_1,\cdots, a_n)$ is a solution of the homogeneous $R$-linear equation $x_0+rx_1+\cdots+r^nx_n=0$. Then we need the definition of ``$R$-vector", $R$-linear and so on.

\subsection{Basic concepts }
We first introduce the notion of ``$R$-vector space".
\begin{Def}\label{R-space}
Let $R$ be a commutative ring with 1 and $n$ be a positive integer. Let $R^n=\{(a_1,a_2,\cdots,a_n)\mid a_i\in R, 1\leq i\leq n\}$, then
for $(a_1,a_2,\cdots,a_n), (b_1,b_2,\cdots,b_n)\in R^n$,
$$(a_1,a_2,\cdots,a_n)=(b_1,b_2,\cdots,b_n)\Leftrightarrow a_i=b_i (1\leq i\leq n)\in R.$$
$R^n$ has two operations as follows, for $(a_1,a_2,\cdots,a_n), (b_1,b_2,\cdots,b_n)\in R^n$, $r\in R$, then
\begin{enumerate}
\item[{\rm (i)}]{\rm Vector addition}: $(a_1,a_2,\cdots,a_n)+(b_1,b_2,\cdots,b_n)=(a_1+b_1,a_2+b_2,\cdots,a_n+b_n)$;
\item[{\rm (ii)}]{\rm Scalar multiplication}: $r(a_1,a_2,\cdots,a_n)=(ra_1,ra_2,\cdots,ra_n)$.
\end{enumerate}
These two operations on $R^n$ satisfy the following eight rules. For any $\mathbf{a},\mathbf{b},\mathbf{c}\in R^n$, $r,k\in R$,

\begin{enumerate}
\item$\mathbf{a}+\mathbf{b}=\mathbf{b}+\mathbf{a}$;
\item$(\mathbf{a}+\mathbf{b})+\mathbf{c}=\mathbf{a}+(\mathbf{b}+\mathbf{c})$;
\item there is a unique vector $\mathbf{0}=(0,0,\cdots,0)$ in $R^n$ such that $\mathbf{0}+\mathbf{a}=\mathbf{a}+\mathbf{0}=\mathbf{a}$,
then $\mathbf{0}$ is called the $\mathbf{zero}$ $\mathbf{vector}$ of $R^n$;
\item for any $\mathbf{a}=(a_1,a_2,\cdots,a_n)\in R^n$, there is a vector $-\mathbf{a}=(-a_1,-a_2,\cdots,-a_n)\in R^n$,
called the $\mathbf{negative}$ of $\mathbf{a}$, such that $\mathbf{a}+(-\mathbf{a})=(-\mathbf{a})+\mathbf{a}=\mathbf{0}$;
\item $1(\mathbf{a})=\mathbf{a}$;
\item $(kr)\mathbf{a}=k(r\mathbf{a})$;
\item $(k+r)\mathbf{a}=k\mathbf{a}+r\mathbf{a}$;
\item$r(\mathbf{a}+\mathbf{b})=r\mathbf{a}+r\mathbf{b}$.
\end{enumerate}
Then $R^n$ is called an $n$-$\mathbf{dimensional}$ $R$-$\mathbf{vector}$ $\mathbf{space}$ or $R$-$\mathbf{linear}$ $\mathbf{space}$, and any $\mathbf{a}\in R^n$ is called an $n$-$\mathbf{dimensional}$ $R$-$\mathbf{vector}$.
\end{Def}
And we have the notion of subspace.
\begin{Def}\label{Rsub}
If a nonempty subset $U$ of $R^n$ satisfies that
\begin{enumerate}
\item[{\rm (i)}]$\mathbf{a},\mathbf{b}\in U\Rightarrow \mathbf{a}+\mathbf{b}\in U$;
\item[{\rm (ii)}]$\mathbf{a}\in U, r\in R \Rightarrow r\mathbf{a}\in U$.
Then $U$ is called an $R$-$\mathbf{vector}$ $\mathbf{subspace}$ of $R^n$.
\end{enumerate}
\end{Def}

\begin{Prop}\label{Rle-Rss}
Let $R$ be a commutative ring with 1. For $t_1,t_2,\cdots,t_n\in R$, if there is a system of homogeneous $R$-linear equations
\begin{align}\label{nRlss}
\begin{cases}
x_0+t_1x_1+t_1^2x_2+\cdots+t_1^{n-1}x_{n-1}=0\\
x_0+t_2x_1+t_2^2x_2+\cdots+t_2^{n-1}x_{n-1}=0\\
~~~~~~~~~~~~~~~\vdots\\
x_0+t_nx_1+t_n^2x_2+\cdots+t_n^{n-1}x_{n-1}=0
\end{cases}
\end{align}
with variables $x_0, x_1, \cdots,x_{n-1}$. Then the solution of Equations \ref{nRlss} is an $R$-vector subspace of $R^n$.
\end{Prop}

\subsection{$n$-tuple of a polynomial over a commutative ring}
Now, we give a sufficient condition such that the solution of Equations \ref{nRlss} is unique.
\begin{Lem}\label{eqnt}
Let $R$ be a commutative ring with 1. For $t_1,t_2,\cdots,t_n\in R$, any $1\leq i\neq j\leq n$, $t_i-t_j$ is not zero or zero-divisor. If there is a system of homogeneous $R$-linear equations \ref{nRlss} with variables $x_0, x_1, \cdots,x_{n-1}$, then the solution of Equations \ref{nRlss}
is the subspace $\{\mathbf{0}\}$ of $R^n$.
\end{Lem}
\begin{proof}
For constants $c_0, c_1,\cdots, c_{n-1},d \in R$, if $t$ is not zero or zero-divisor, then the solutions of
$c_0x_0+c_1x_1+\cdots+c_{n-1}x_{n-1}=d$ and  $tc_0x_0+tc_1x_1+\cdots+tc_{n-1}x_{n-1}=td$ are the same, that is
$$c_0x_0+c_1x_1+\cdots+c_{n-1}x_{n-1}=d \Leftrightarrow tc_0x_0+tc_1x_1+\cdots+tc_{n-1}x_{n-1}=td.$$
We use Gaussian elimination to solve the $R$-linear equations:
$$\left(
  \begin{array}{ccccc}
   1     & t_1 & t_1^2 &\cdots & t_1^{n-1} \\
   1     & t_2 & t_2^2 &\cdots & t_2^{n-1} \\
   1     & t_3 & t_3^2 &\cdots & t_3^{n-1} \\
    \vdots & \vdots & \vdots& \ddots& \vdots \\
   1     & t_{n} & t_{n}^2 &\cdots & t_{n}^{n-1} \\
  \end{array}
\right)\rightarrow \left(
  \begin{array}{ccccc}
   1     & t_1 & t_1^2 &\cdots & t_1^{n-1} \\
   0     & t_2-t_1 & t_2^2-t_1^2 &\cdots & t_2^{n-1}-t_1^{n-1} \\
   0     & t_3-t_1 & t_3^2-t_1^2 &\cdots & t_3^{n-1}-t_1^{n-1} \\
    \vdots & \vdots & \vdots& \ddots& \vdots \\
   0     & t_{n}-t_1 & t_{n}^2-t_1^2 &\cdots & t_{n}^{n-1}-t_1^{n-1} \\
  \end{array}
\right)\rightarrow \left(
  \begin{array}{ccccc}
   1     & t_1 & t_1^2 &\cdots & t_1^{n-1} \\
   0     & 1 & t_1+t_2 &\cdots & \sum^{n-2}_{i=0}t_1^{n-2-i}t_2^{i} \\
   0     & 1 & t_1+t_3 &\cdots & \sum^{n-2}_{i=0}t_1^{n-2-i}t_3^{i} \\
    \vdots & \vdots & \vdots& \ddots& \vdots \\
   0     & 1 & t_1+t_{n}&\cdots & \sum^{n-2}_{i=0}t_1^{n-2-i}t_{n}^{i} \\
  \end{array}
\right),$$
then inductively carry out the above process and finally obtain the upper triangular matrix
$$\left(
  \begin{array}{ccccc}
   1     & t_1 & t_1^2 &\cdots & t_1^{n-1} \\
   0     & 1 & t_1+t_2 &\cdots & \sum^{n-2}_{i=0}t_1^{n-2-i}t_2^{i} \\
   0     & 0 & 1 &\cdots & \sum^{n-2}_{i=1}t_1^{n-2-i}\sum^{i-1}_{j=0}t_2^{i-1-j}t_3^j \\
    \vdots & \vdots & \vdots& \ddots& \vdots \\
   0     & 0 & 0 &\cdots & 1 \\
  \end{array}
\right),$$
this finishes the proof.
\end{proof}
The coefficient matrix $V(t_1,t_2,\cdots,t_n)=(t_i^{j-1})_{1\leq i,j\leq n}$ of Equation \ref{nRlss} is a Vandermonde matrix. The determinant $\det V(t_1,t_2,\cdots,t_n)$ of $V(t_1,t_2,\cdots,t_n)$ can be calculated in the conventional manner without any specific assumptions placed on $t_1,t_2,\cdots,t_n$.
\begin{Lem}\label{DV}
Let $R$ be a commutative ring with 1. Let $t_1,t_2,\cdots,t_n\in R$. Then the determinant
$$\det V(t_1,t_2,\cdots,t_n)=\prod_{1\leq j<i\leq n}(t_i-t_j).$$
Besides for any $1\leq i\neq j\leq n$, $t_i-t_j$ is not zero or zero-divisor if and only if $\det V(t_1,t_2,\cdots,t_n)$ is not zero or zero-divisor.
\end{Lem}
\begin{proof}
By employing the established classical technique for computing the determinant of a Vandermonde matrix, we obtain
$$\det V(t_1,t_2,\cdots,t_n)=\prod_{1\leq j<i\leq n}(t_i-t_j).$$

Note that for $a,b \in R$, $a$ and $b$ are not zero or zero-divisor if and only if $ab$ is not zero or zero-divisor, so we have
for any $1\leq i\neq j\leq n$, $t_i-t_j$ is not zero or zero-divisor if and only if $\prod_{1\leq j<i\leq n}(t_i-t_j)$ is not zero or zero-divisor.
\end{proof}

Lemma \ref{eqnt} motivates us to introduce the following notions.
\begin{Def}\label{ntuple}
Let $R$ be a commutative ring with 1. we define an $n$-$\mathbf{tuple}$ $\{t_1,t_2,\cdots, t_n\}$ of $R$ such that for any $1\leq i\neq j\leq n$, $t_i-t_j$ is not zero or zero-divisor; if for any $1\leq i\neq j\leq n$, $t_i-t_j$ is invertible in $R$, we call $\{t_1,t_2,\cdots, t_n\}$ an $\mathbf{invertible}$ $n$-$\mathbf{tuple}$ of $R$. Let $f(x)$ be a polynomial over $R$, we call $\{r_1,r_2,\cdots, r_n\}$ an $n$-$\mathbf{tuple}$ of $f(x)$ if it is an $n$-tuple of $R$ and also is a subset of roots of $f(x)$.
\end{Def}

To explain the meaning of fractions of Theorem \ref{Fac-tuple}, we give the following definition.
\begin{Def}
Let $R$ be a commutative ring with 1, and $d$ is not zero or zero-divisor in $R$. For $r\in R$, we call $t$ is $\mathbf{divisible}$ by $d$ if
there is an element $t'\in R$ such that $t=dt'$.
\end{Def}

\begin{Rem}
Since $d$ is not zero or zero-divisor in $R$, for $t\in R$, the solution of $t=dx$ in $R$ is unique.
\end{Rem}

With the notion of $n$-tuple, we give a sufficient condition to address the question posted in the introduction of this section.
\begin{Prop}\label{injnt}
Let $R$ be a commutative ring with 1. If $R$ has an $n$-tuple $\{t_1,t_2,\cdots, t_n\}$, then $\lambda_{R[x]_{n-1}}$ is injective.
\end{Prop}
\begin{proof}
For any two polynomial $f_1(x)\neq f_2(x)\in R[x]_{n-1}$, without loss of generality we may suppose that $f_1(x)=\sum^{n-1}_{k=0}a_kx^k, f_2(x)=\sum^{n-1}_{k=0}b_kx^k$. Then $f_1(x)\neq f_2(x)$ implies that there is $1\leq k_0\leq n-1$ such that $a_{k_0}-b_{k_0}\neq 0$.
If $\lambda_{R[x]_{n-1}}(f_1(x))=\lambda_{R[x]_{n-1}}(f_2(x))$, that is $[f_1(y)-f_2(y)=(f_1-f_2)(y)]=[0]$, which implies that
$(f_1-f_2)(t_i)=0$ for any $1\leq i\leq n$. Then the $n$-dimensional $R$-vector $(a_0-b_0, a_1-b_1, \cdots,a_{n-1}-b_{n-1})$ is a solution
of Equations \ref{nRlss}. And by Lemma \ref{eqnt}, the solution of Equations \ref{nRlss} is $\{\mathbf{0}\}$. So $(a_0-b_0, a_1-b_1, \cdots,a_{n-1}-b_{n-1})=(0,0,\cdots,0)$, which contradicts to our assumption that $a_{k_0}-b_{k_0}\neq 0$. This finishes the proof.
\end{proof}

%\begin{Lem}\label{injint}
%Let $R$ be a commutative ring with 1. If $R$ has an invertible $n$-tuple $\{r_1,r_2,\cdots, r_n\}$, then for a polynomial $f(x)\in R[x]_{n-1}$, then
%$$f(x)=\sum^n_{j=0}f(r_j)\prod_{1\leq i\leq n, i\neq j}\frac{x-r_i}{r_j-r_i}.$$
%\end{Lem}
%\begin{proof}
%Note that $\lambda_{R[x]_{n-1}}(f(x)-\sum^n_{j=o}f(r_j)\prod_{1\leq i\leq n, i\neq %j}\frac{x-r_i}{r_j-r_i})=[f(y)-\sum^n_{j=o}f(r_j)\prod_{1\leq i\leq n, i\neq j}\frac{y-r_i}{r_j-r_i}=\sum^{n-1}_{k=0}c_ky^k]$ with $\deg g\leq %n-1$ and for any $1\leq i\leq n$, $\sum^{n-1}_{k=0}c_k{r_i}^k=0$. Similarly, by use the proof of Lemma \ref{injnt}, we have $c_k=0\in R$
%for any $0\leq k\leq n-1$, then $[f(y)-\sum^n_{j=o}f(r_j)\prod_{1\leq i\leq n, i\neq j}\frac{y-r_i}{r_j-r_i}]=[0]$. So by Lemma \ref{injnt}, we %obtain the result.
%\end{proof}
The subsequent lemma plays a pivotal role in proving the meaningfulness of the fractions stated in Theorem \ref{Fac-tuple}.
\begin{Lem}\label{DVD}
Let $R$ be a commutative ring with 1 and $R$ has an $n$-tuple $\{t_1,t_2,\cdots, t_n\}$. Let $\alpha_i$ denote the column $n$-dimensional $R$-vector $(t_1^i,t_2^i,\cdots,t_n^i)^T$. If $0\leq i_1<i_2<\cdots<i_{n-1}$, then $\det(\alpha_0,\alpha_{i_1},\alpha_{i_2},\cdots,\alpha_{i_{n-1}})$ is divisible by $\det(\alpha_0,\alpha_1,\cdots,\alpha_{n-1})=\det V(t_1, t_2, \cdots, t_n)$.
\end{Lem}
\begin{proof}
Let $R[s_1, s_2, \cdots, s_{i_{n-1}-n+1}]$ be the ring of polynomials with $i_{n-1}-n+1$ indeterminates $s_1, s_2, \cdots, s_{i_{n-1}-n+1}$ over $R$. By Lemma \ref{DV}, we obtain that the determinant
$$\det V(s_1, s_2, \cdots, s_{i_{n-1}-n+1}, t_1, t_2, \cdots, t_n) \in R[s_1, s_2, \cdots, s_{i_{n-1}-n+1}]$$
and the coefficient of its each monomial is divisible by $\det V(t_1, t_2, \cdots, t_n)$. By Laplace theorem, we expand $\det V(s_1, s_2, \cdots, s_{i_{n-1}-n+1}, t_1, t_2, \cdots, t_n)$ along the first $i_{n-1}-n+1$ rows, then $\det(\alpha_0,\alpha_{i_1},\alpha_{i_2},\cdots,\alpha_{i_{n-1}})$ is the cofactor of some $(i_{n-1}-n+1)\times (i_{n-1}-n+1)$ matrix $A$. Note that each term of $\det A$ is a unique monomial of $\det V(s_1, s_2, \cdots, s_{i_{n-1}-n+1}, t_1, t_2, \cdots, t_n)$. Therefore $\det(\alpha_0,\alpha_{i_1},\alpha_{i_2},\cdots,\alpha_{i_{n-1}})$ is divisible by $\det V(t_1, t_2, \cdots, t_n)$.
\end{proof}

The following lemma proves the last two cases of Theorem \ref{Fac-tuple}.
\begin{Lem}\label{Nohc}
Let $R$ be a commutative ring with 1 and $f(x)=a_0+a_1x+\cdots+a_mx^m$ be a polynomial over $R$. $R$ has an $n$-tuple $\{r_1,r_2,\cdots, r_n\}$ of $f(x)$ with $n\leq m$. Let $\alpha_i$ denote the column $R$-vector $(r_1^i,r_2^i,\cdots,r_n^i)^T$ for $0\leq i\leq n-1$, and let $\beta$ denote the column $R$-vector $(-\sum^m_{i=n}a_{i}r_1^{i},-\sum^m_{i=n}a_{i}r_2^{i},\cdots, -\sum^m_{i=n}a_{i}r_n^{i})^T$. Then for $0\leq i\leq n-1$, $\det(\alpha_0,\cdots,\alpha_{i-1},\beta, \alpha_{i+1},\cdots,\alpha_{n-1})$ is divisible by $\det(\alpha_0,\alpha_1,\cdots,\alpha_{n-1})$ and
$$a_i=\frac{\det(\alpha_0,\cdots,\alpha_{i-1},\beta, \alpha_{i+1},\cdots,\alpha_{n-1})}{\det(\alpha_0,\alpha_1,\cdots,\alpha_{n-1})}.$$
\end{Lem}
\begin{proof}
There is a system of non-homogeneous linear equations
$$\{x_0+r_ix_1+\cdots+r_i^{n-1}x_{n-1}=-\sum^m_{i_1=n}a_{i_1}r_i^{i_1}\mid 1\leq i\leq n\}$$
with variables $x_0, x_1, \cdots,x_{n-1}$. We use Gaussian elimination to solve these $R$-linear equations and obtain
$${\small\left(
  \begin{array}{cccccc}
   1     & r_1 & r_1^2 &\cdots & r_1^{n-1}            &-\sum^m_{i_1=n}a_{i_1}r_1^{i_1} \\
   0     & 1 & r_1+r_2 &\cdots & \sum^{n-1}_{i_1=0}r_1^{n-2-i_1}r_2^{i_1}& -\sum^m_{i_1=n}a_{i_1}\sum^{i_1-1}_{i_2=0}r_1^{i_1-1-i_2}r_2^{i_2}\\
   0     & 0 & 1 &\cdots & \sum^{n-1}_{i_1=1}r_1^{n-2-i_1}\sum^{i_1-1}_{i_2=0}r_2^{i_1-1-i_2}r_3^{i_2}& -\sum^m_{i_1=n}a_{i_1}\sum^{i_1-1}_{i_2=0}r_1^{i_1-1-i_2}\sum^{i_2-1}_{i_3=0}r_2^{i_2-1-i_3}r_3^{i_3}\\
    \vdots & \vdots & \vdots& \ddots& \vdots& \vdots\\
   0     & 0 & 0 &\cdots & 1 &-\sum^m_{i_1=n}a_{i_1}\sum^{i_1-1}_{i_2=0}r_1^{i_1-1-i_2}\sum^{i_2-1}_{i_3=0}r_2^{i_2-1-i_3}\cdots \sum^{i_{n-1}-1}_{i_n=0}r_{n-1}^{i_{n-1}-1-i_n}r_n^{i_n}\\
  \end{array}
\right)}.$$
Which implies that the solution of $\{x_0+r_ix_1+\cdots+r_i^{n-1}x_{n-1}=-(a_nr_i^n+\cdots+a_mr_i^m)\mid 1\leq i\leq n\}$ is unique. Then by Lemma \ref{DVD}, we get that
$\frac{\det(\alpha_0,\cdots,\alpha_{i-1},\beta,\alpha_{i+1},\cdots,\alpha_{n-1})}{\det(\alpha_0,\alpha_1,\cdots,\alpha_{n-1})}$
is well defined for each $i$. Let $\widetilde{\alpha}_i$ denote the row $R$-vector $(1,r_i,r_i^2,\cdots,r_i^{n-1},-\sum^m_{i_1=n}a_{i_1}r_i^{i_1})$ for $0\leq i\leq n-1$, then by using properties of determinant, we obtain that
$\det(\widetilde{\alpha}_i, \widetilde{\alpha}_0,\widetilde{\alpha}_1,\cdots,\widetilde{\alpha}_{n-1})=0$  and
$$(\frac{\det(\beta,\alpha_{1},\cdots,\alpha_{n-1})}{\det(\alpha_0,\alpha_1,\cdots,\alpha_{n-1})},\frac{\det(\alpha_{0},\beta,\alpha_{2},\cdots,\alpha_{n-1})}
{\det(\alpha_0,\alpha_1,\cdots,\alpha_{n-1})}, \cdots, \frac{\det(\alpha_0,\cdots,\alpha_{n-2},\beta)}{\det(\alpha_0,\alpha_1,\cdots,\alpha_{n-1})})$$
 is a solution of $\{x_0+r_ix_1+\cdots+r_i^{n-1}x_{n-1}=-\sum^m_{i_1=n}a_{i_1}r_i^{i_1}\mid 1\leq i\leq n\}$. This finishes the proof.
\end{proof}

\begin{proof}[Proof of Theorem \ref{Fac-tuple}]
{\rm (i)} If $n>m$, then by Lemma \ref{eqnt} the solution of Equations \ref{nRlss} is the subspace $\{\mathbf{0}\}$. Since $(a_0,a_1,\cdots, a_n)$ is a solution of Equations \ref{nRlss}, we must have $(a_0,a_1,\cdots, a_n)=\mathbf{0}$.

{\rm (ii)(iii)} If $n\leq m$, then by Lemma \ref{Nohc}, we finish the proof.
\end{proof}

The first corollary of Theorem \ref{Fac-tuple} shows that generalized relations between roots and coefficients of a polynomial can be viewed in some sense as polynomial interpolation over a commutative ring.
\begin{Cor}\label{Intf}
Let $R$ be a commutative ring with 1 and $f(x)=a_0+a_1x+\cdots+a_mx^m$ be a polynomial over $R$. If $R$ has an invertible $n$-tuple $\{r_1,r_2,\cdots, r_n\}$ of $f(x)$, then
 $$f(x)=\sum^n_{j=1}\prod_{1\leq i\leq n, i\neq j}\frac{x-r_i}{r_j-r_i}(-\sum^m_{k=n}a_{k}r_j^{k}),$$
 when $m<n$, $-\sum^m_{k=n}a_{k}r_j^{k}$ denotes $0$.
\end{Cor}

Corollary \ref{Intf} and Corollary \ref{vanish} can be easily deduced from Theorem \ref{Fac-tuple}, so we omit their proof.
\subsection{Applications of Vanishing ring condition}

In this subsection, we give two applications of Corollary \ref{vanish}. However, it is important to note that $[p^j]_{E}(x)$ is not a polynomial but a power series, which prevents us from directly using Corollary \ref{vanish}. To overcome this issue, we identify the power series $[p^j]_E(x)$ with its Weierstrass polynomial $g_j(x)$, as per Proposition \ref{idwp}. Then we could apply Corollary \ref{vanish} to the following two cases.
\begin{Cor}\label{usevan}
\begin{enumerate}
\item[\rm (i)]  If $G$ is a finite $p$-group, then $t_{G}({\rm inf}^G_{e}(K(n)))^{G}\simeq *$. (\cite[Theorem 1.1]{GS96})
\item[\rm (ii)] Let $G$ be a finite $p$-group and $H$ be a non-cyclic subgroup, then ${\Phi}^{H}(KU_G)\simeq *$.(\cite[Proposition 3.10]{BGS} )
\end{enumerate}
\end{Cor}
\begin{proof}
{\rm (i)} By the proof of \cite[Theorem 1.1]{GS96}, it suffices to prove that
$t_{\mathbb{Z}/p}({\rm inf}^{\mathbb{Z}/p}_{e}(K(n)))^{\mathbb{Z}/p}\simeq *$. Let $f(y)$ be $\frac{[p]_{K(n)}(y)}{y^{p^n-1}}=v_ny$. Note that both $0$ and $x^{p^n}$ are roots
of $f(y)$ in $\pi_*(t_{\mathbb{Z}/p}({\rm inf}^{\mathbb{Z}/p}_{e}(K(n)))^{\mathbb{Z}/p})=\pi_*(\cat
T_{\mathbb{Z}/p,\mathbb{Z}/p}(K(n)))=L^{-1}_{\mathbb{Z}/p}\mathbb{F}_p[v_n^{\pm 1}]\psb{x}/(v_nx^{p^n})$, where the multiplicatively closed set
$L_{\mathbb{Z}/p}$ is generated by all Euler classes induced by one dimensional complex representations of $\mathbb{Z}/p$. And their difference
$x^{p^n}$ is in $L_{\mathbb{Z}/p}$, hence it is not a zero divisor. By Corollary \ref{usevan}, we have
$t_{\mathbb{Z}/p}({\rm inf}^{\mathbb{Z}/p}_{e}(K(n)))^{\mathbb{Z}/p}\simeq *$.

{\rm (ii)} By the proof of \cite[Proposition 3.10]{BGS}, it suffices to prove that
${\Phi}^{\mathbb{Z}/p\times\mathbb{Z}/p}(KU_{\mathbb{Z}/p\times \mathbb{Z}/p})\simeq *$. Let $f(x)$ be $\frac{(x+1)^p-1}{x}$. Note that the
Euler classes $x_1-1, x_1^2-1, \cdots, x_1^{p-1}-1, x_2-1$ are different roots of $f(x)$ in
$\pi_*({\Phi}^{\mathbb{Z}/p\times\mathbb{Z}/p}(KU_{\mathbb{Z}/p\times\mathbb{Z}/p}))=L^{-1}_{\mathbb{Z}/p\times\mathbb{Z}/p}\mathbb{Z}[x_1,x_2]/(x^p_1-1,x^p_2-1)$,
where the multiplicatively closed set $L_{\mathbb{Z}/p\times \mathbb{Z}/p}$ is generated by all Euler classes induced by one dimensional
complex representations of $\mathbb{Z}/p\times \mathbb{Z}/p$. Note that the difference of any two roots has the forms
$(x_1^{m}-x_1^{n})=x_1^{n}(x_1^{m-n}-1)$ or $(x_2-x_1^{n})=x_1^{n}(x_1^{p-n}x_2-1)$, since $x_1^{n}$ is invertible in
$L^{-1}_{\mathbb{Z}/p\times \mathbb{Z}/p}\mathbb{Z}[x_1,x_2]/(x^p_1-1,x^p_2-1)$ and $x_1^{p-n}x_2-1$ is the Euler class in
$L_{\mathbb{Z}/p\times \mathbb{Z}/p}$, we conclude that
${\Phi}^{\mathbb{Z}/p\times \mathbb{Z}/p}(KU_{\mathbb{Z}/p\times \mathbb{Z}/p})\simeq *$ by Corollary \ref{usevan}.
\end{proof}

\section{Algebraic periodicity and Landweber exactness}
Most of this section are due to Greenlees--Sadofsky \cite{GS96} and Hovey \cite{Ho95}, we just add some details here.
\subsection{Algebraic periodicity\label{Ap}}
There are two distinct definitions of being $v_n$-periodic for a $p$-local and complex-oriented spectrum $E$, each presented by
Greenlees--Sadofsky \cite{GS96} and Hovey \cite{Ho95}, respectively. These definitions are closely related, with Hovey's version being stronger
than Greenlees--Sadofsky's. In this paper, we opt to adopt Hovey's definition as our chosen characterization of a $v_n$-periodic property for a
$p$-local and complex-oriented spectrum $E$.

Recall a finite spectrum $X$ has $\textit{type}$ $n$ if $K(n-1)_*X=0$ but $K(n)_*X\neq 0$.
\begin{Lem}(Hopkins--Smith, \cite{HS98})\label{HSle}
All finite spectrum of type $n$ have the same Bousfield class and is denoted by $F(n)$. The spectrum $F(n)$ has a $v_n$ self-map and its telescope is denoted by $T(n)$.
\end{Lem}

Let $M(p^{i_0}, v_{1}^{i_1}, \cdots, v_{n-1}^{i_{n-1}})$ be a finite spectrum with
$$\pi_*(BP\wedge M(p^{i_0}, v_{1}^{i_1}, \cdots, v_{n-1}^{i_{n-1}}))=BP_*/(p^{i_0}, v_{1}^{i_1}, \cdots, v_{n-1}^{i_{n-1}}).$$
Such spectra are of type $n$ and are called $\textit{generalized}$ $\textit{Moore}$ $\textit{spectra}$. $M(p^{i_0}, v_{1}^{i_1}, \cdots, v_{n-1}^{i_{n-1}})$ are guaranteed to exist for sufficiently large multi-indices
$I =(i_0, \cdots, i_{n-1})$ by the periodicity theorem of Smith \cite{HS98}, written up in \cite[Section 6.4]{Ra92}.

We use the notation $X^{\wedge}_{I_{n}}$ for the completion of $X$ with respect to the ideal $I_n=(p,v_1,\cdots,v_{n-1})\subset BP_*$. More
precisely, the construction is
\begin{align}\label{com}
X^{\wedge}_{I_{n}}=\lim_{\longleftarrow \atop (i_0,i_1,\cdots,i_{n-1})}(X\wedge M(p^{i_0},v_1^{i_1},\cdots,v_{n-1}^{i_{n-1}})),
\end{align}
where the inverse limit is taken over maps
$$M(p^{j_0},v_1^{j_1},\cdots,v_{n-1}^{j_{n-1}})\rightarrow M(p^{i_0},v_1^{i_1},\cdots,v_{n-1}^{i_{n-1}})$$
commuting with inclusion of the bottom cell. Such maps are easily constructed by courtesy of the nilpotence theorem of \cite{HS98} (see for example \cite[Proposition 3.7]{HS98} for existence of these maps and some uniqueness properties). By \cite[Definition 1.4]{Ra84}, for any spectrum $E$ there is an $E$-$\textit{localization}$ $\textit{functor}$ $L_E: \SH(e)\rightarrow \SH(e)$. The following theorem says that localization with respect to $F(n)$ is completion at $I_n$.
\begin{Thm}(Hovey, \cite[Theorem 2.1]{Ho95})\label{Lfn}
For any spectrum $X$, the map $X\rightarrow \underleftarrow{\lim} (X\wedge M(p^{i_0}, v_1^{i_1}, \cdots, v_{n-1}^{i_{n-1}}))$ is a $F(n)$-localization, namely $L_{F(n)}X=X^{\wedge}_{I_{n}}$.
\end{Thm}

If $E$ is $p$-local and complex oriented, then there is a unique map $f:BP\rightarrow E$ such that
$$f^*:BP^*(\mathbb{C}P^{\infty})\cong BP^*\psb{x_{BP}}\rightarrow E^*(\mathbb{C}P^{\infty})\cong E^*\psb{x_{E}}$$
maps the $BP$-orientation $x_{BP}$ to the $E$-orientation $x_{E}$. And there is a homomorphism
$${f\wedge 1_{M(p^{i_0},v_1^{i_1},\cdots,v_{n-1}^{i_{n-1}})}}_*:\pi_*(BP\wedge M(p^{i_0},v_1^{i_1},\cdots,v_{n-1}^{i_{n-1}}))\rightarrow \pi_*(E\wedge M(p^{i_0},v_1^{i_1},\cdots,v_{n-1}^{i_{n-1}}))$$
and we still use $v_i$ denote ${f\wedge 1_{M(p^{i_0},v_1^{i_1},\cdots,v_{n-1}^{i_{n-1}})}}_*(v_i)$.

\begin{Def}(Greenlees--Sadofsky's $v_n$-periodic, \cite[Definition 1.3]{GS96})\label{de1}
Let $E$ be a $p$-local and complex oriented spectrum, $E$ is called $v_n$-$\textit{periodic}$ if $v_n$ is a unit on the nontrivial spectrum
$E\wedge M(p^{i_0},v_1^{i_1},\cdots,v_{n-1}^{i_{n-1}})$ for sufficiently large multi-indices $I=(i_0,i_1,\cdots,i_{n-1})$.
\end{Def}
\begin{Rem}
\begin{enumerate}
\item[\rm (i)] The above definition is independent of the choice of multi-index $I$ and of the spectrum $M(p^{i_0},v_1^{i_1},\cdots,v_{n-1}^{i_{n-1}})$.
By Theorem \ref{Lfn}, the equivalent definition of $v_n$-periodic for $E$ is that $v_n$ is a unit on the nontrivial spectrum
$L_{F(n)}E$.
\item[\rm (ii)] If a $p$-local and complex oriented spectrum $E$ is $v_n$-periodic, then $n$ is unique.
\end{enumerate}
\end{Rem}

There is another definition of $v_n$-periodic due to Hovey \ref{De1}, and we refine the definition as follows
\begin{Def}\label{Rd}
Let $E$ be a $p$-local and complex oriented spectrum.
\begin{enumerate}
\item[\rm (i)] $E$ is called $\textit{at}$ $\textit{most}$ $v_n$-$\textit{periodic}$ if $v_n$ is a unit on $E^*/I_{n}$, by the exactness of
$$\CD E^*/I_{n} @>\cdot v_n>> E^*/I_{n} @>>>E^*/I_{n+1}\endCD,$$
which is equivalent to $E^*/I_{n+1}=0$.
\item[\rm (ii)] $E$ is called $\textit{at}$ $\textit{least}$ $v_n$-$\textit{periodic}$ if $E^*/I_{n}\neq 0$.
\end{enumerate}
$E$ is $v_n$-periodic if and only if $E^*/I_{n+1}=0$ and $E^*/I_{n}\neq 0$.
\end{Def}

If we say some spectrum is $v_n$-periodic, we mean it in the sense of Hovey's definition, namely Definition \ref{De1}.

The following proposition says that Hovey's $v_n$-$periodic$ (Definition \ref{De1}) implies that Greenlees-Sadofsky's $v_n$-$periodic$ (Definition \ref{de1}).
\begin{Prop}\label{pepr}
Let $E$ be a $p$-local and complex oriented spectrum. If $v_n$ is a unit of $E^*/I_{n}\neq 0$, then $v_n$ is a unit on some nontrivial spectrum
$E\wedge M(p^{i_0},v_1^{i_1},\cdots,v_{n-1}^{i_{n-1}})$.
\end{Prop}
\begin{proof}
Suppose $v_n\equiv u \mod I_{n}$ for some unit $u$ of $E^*/I_{n}$, then there exists an element $t\in I_{n}$ such that
$v_n= u+t$. Since $u^{-1}-u^{-2}t+u^{-3}t^2-\cdots$ is a power series that converges in $(E^*)^{\wedge}_{I_{n}}$, $v_n$ is a unit of
$(E^*)^{\wedge}_{I_{n}}$. By Theorem \ref{Lfn}, $v_n$ is a unit in $\pi_*(L_{F(n)}E)=(E^*)^{\wedge}_{I_{n}}$.

Since there exists a generalized Moore spectrum $M(p^{i_0},v_1^{i_1},\cdots,v_{n-1}^{i_{n-1}})$ of type $n$ with large enough multi-index
$I=(i_0,i_1,\cdots,i_{n-1})$, from the construction \ref{com} for $E$, it follows that $v_n$ is a unit in
$$\pi_*(E\wedge M(p^{i_0},v_1^{i_1},\cdots,v_{n-1}^{i_{n-1}}))=E^*/(p^{i_0},v_1^{i_1},\cdots,v_{n-1}^{i_{n-1}}).$$
This completes the proof.
\end{proof}

\subsection{Landweber exactness}

The Brown-Peterson spectrum $BP$ is a ring spectrum with the product map $\mu_{BP}:BP\wedge BP\rightarrow BP$ and the unit map $\eta_{BP}:S\rightarrow BP$. The spectrum $E$ is called a $BP$-$\textit{module}$ $\textit{spectrum}$ if there is a $BP$-module map $\nu:BP\wedge E\rightarrow E$ such that the following diagrams commute.
$$\xymatrix{
BP\wedge BP\wedge E\ar[d]_{1_{BP}\wedge \nu} ~~~\ar[r]^{\mu_{BP}\wedge 1_E} &BP\wedge E\ar[d]_{\nu} \\
BP\wedge E \ar[r]^{\nu} &E,}
\xymatrix{
S\wedge E\ar[d]_{\simeq} \ar[r]^{\eta_{BP}\wedge 1_E} & BP\wedge E\ar[d]_{\nu} \\
E\ar[r]^{1_E} & E.}$$
A particular good kind of $BP$-module spectrum is the Landweber exact spectrum \cite{La76}.

\begin{Prop}(The Landweber exact functor, \cite{La76})\label{LEdef}
Let $F$ be a formal group law, $p$ be a prime, and $\tilde{v}_i$ be the coefficient of $x^{p^i}$ in
$$[p]_F(x)=\tilde{v}_0x+\tilde{v}_1x^p+\cdots+\tilde{v}_ix^{p^i}+\cdots.$$
If for each $i$ multiplication by $\tilde{v}_i$ is monic on $ \mathbf{Z}_{(p)}[\tilde{v}_1,\tilde{v}_2,\cdots]/(\tilde{v}_0,\tilde{v}_1,\cdots,\tilde{v}_{i-1})$, then $F$ is $\textit{Landweber}$ $\textit{exact}$ and hence gives a cohomology theory
$E^*(-)=BP^*(-)\otimes_{BP^*} \mathbf{Z}_{(p)}[\tilde{v}_1,\tilde{v}_2,\cdots]$. By Brown representation theorem \cite{Br62}, this defines a spectrum and the spectra arising this way are called $\textit{Landweber}$ $\textit{exact}$ $\textit{spectra}$.
\end{Prop}

Recall a lemma due to Ravenel.
\begin{Lem}(Ravenel, \cite[Lemma 1.34]{Ra84})\label{Rale}
Let $X$ be a non-equivariant spectrum and $f:\sum^d X\rightarrow X$ be a self-map of $X$ with cofiber $Y$. Let $T(X)$ denote the $\textit{telescope}$ $\underrightarrow{\lim}\sum^{-id} X$ of $f$. Then
$$\langle X\rangle= \langle T(X)\rangle\vee\langle Y\rangle.$$
\end{Lem}

For two non-equivariant spectra $E$ and $F$, recall that $\langle F\rangle \leq \langle E\rangle$ if for any spectrum $X\in \SH(e)$, $E_*X = 0 \Rightarrow F_*X = 0$.
The Landweber exact spectrum with the assumption of periodicity determines its Bousfield class.
\begin{Lem}\label{}
Let $E$ be a Landweber exact spectrum.
\begin{enumerate}
\item[\rm (i)] If $E$ is at most $v_n$-periodic, then $\langle E\rangle \leq \langle E(n)\rangle$;
\item[\rm (ii)] if $E$ is at least $v_n$-periodic, then $\langle E\rangle \geq \langle E(n)\rangle$.
\end{enumerate}
\end{Lem}
\begin{proof}
Applying Lemma \ref{Rale} repeatedly using $v_n$-self map \ref{HSle}, we get
$$\langle S^0\rangle=\langle T(0)\rangle\vee\cdots \vee\langle T(n)\rangle \vee \langle F(n+1)\rangle.$$
Smashing with $E$, we have
$$\langle E\rangle=\langle E\wedge T(0)\rangle\vee\cdots \vee\langle E\wedge T(n)\rangle \vee \langle E\wedge F(n+1)\rangle.$$
Since $E$ is Landweber exact, $E$ is a $BP$-module spectrum, so $E$ is a retract of $BP\wedge E$, then
$$\langle E\rangle=\langle BP\wedge E\rangle=\langle BP\wedge E\wedge T(0)\rangle\vee\cdots \vee\langle BP\wedge E\wedge T(n)\rangle \vee \langle BP\wedge E\wedge F(n+1)\rangle.$$
By Hovey's theorem \cite[Theorem 1.9]{Ho95} that $\langle BP\wedge T(n)\rangle=\langle K(n)\rangle,$ we have
$$\langle E\rangle=\langle E\wedge K(0)\rangle\vee\cdots \vee\langle E\wedge K(n)\rangle \vee \langle BP\wedge E\wedge F(n+1)\rangle.$$

If $E$ is at most $v_n$-periodic, then by Proposition \ref{pepr}, we have $E\wedge F(n+1)=0$ and
$$\langle E\rangle=\langle E\wedge K(0)\rangle\vee\cdots \vee\langle E\wedge K(n)\rangle\leq \langle K(0)\rangle\vee\cdots \vee\langle K(n)\rangle=\langle E(n)\rangle.$$

If $E$ is at least $v_n$-periodic, that is $E^*/I_{n}\neq 0$, then we get $E^*/I_{j}\neq 0$ for $j\leq n$. And by Proposition \ref{pepr}, we have $E\wedge F(j)\neq 0$ for $j\leq n$. Since $E$ is Landweber exact, the map $E^*/I_{j}\rightarrow v_j^{-1}E^*/I_{j}$ is injective, so $v_j^{-1}E^*/I_{j}\neq 0$ and $E\wedge T(j)\neq 0$ for $j\leq n$. Note that $\langle E\wedge T(j)\rangle=\langle E\wedge K(j)\rangle$ and for any $F\in \SH(e)$, $\langle F\wedge K(j)\rangle$ is either $0$ or $\langle K(j)\rangle$, then we have $\langle E\wedge K(j)\rangle=\langle K(j)\rangle$ for $j\leq n$ and
$$\langle E\rangle=\langle K(0)\rangle\vee\cdots \vee\langle K(n)\rangle\vee \langle BP\wedge E\wedge F(n+1)\rangle\geq \langle K(0)\rangle\vee\cdots \vee\langle K(n)\rangle=\langle E(n)\rangle.$$
\end{proof}

\begin{Thm}(Hovey, \cite[Corollary 1.12]{Ho95})\label{En}
If $E$ is a $v_n$-periodic and Landweber exact spectrum, then
$$\langle E\rangle=\langle E(n)\rangle=\langle K(0)\vee\cdots \vee K(n)\rangle.$$
\end{Thm}

\begin{Lem}\label{LE}
If $E$ is Landweber exact, then ${\cat T}_{A,C}(E)$ is Landweber exact.
\end{Lem}
\begin{proof}
Note that $E^*(BA_+)$ is a finite free module over $E^*$. Since $E$ is Landweber exact, $v_0,\cdots ,v_i$ form a regular sequence of $E^*(BA_+)$ for all $p$ and $i$. Hence for all $i$ there are short exact sequences
$${\small\CD
  0 @>>>E^*(BA_+)/I_{i} @>\cdot v_{i}>> E^*(BA_+)/I_{i}@>>> E^*(BA_+)/I_{i+1} @>>> 0.
\endCD}$$

By Theorem \ref{bgcoh}, we know that
$\pi_*({\cat T}_{A,C}(E))$ is a localization of $E^*(BA_+)$. Note that $E^*(BA_+)/I_{i+1}$ is an $E^*(BA_+)/I_i$-module and the localization functor is exact, we have
short exact sequences
$${\small\CD
  0 @>>>\pi_{*}(\cat T_{A,C}(E))/I_{i} @>\cdot v_{i}>> \pi_{*}(\cat T_{A,C}(E))/I_{i}@>>> \pi_{*}(\cat T_{A, C}(E))/I_{i+1} @>>> 0.
\endCD}$$
This deduces that $v_0,\cdots ,v_i$ form a regular sequence of $\pi_*({\cat T}_{A,C}(E))$ for all $p$ and $i$. This finishes the proof.
\end{proof}
\section{Generalized Tate construction lowers Bousfield class}
In this section, we prove the following theorem.
\begin{Thm}($\mathbf{Generalized}$ $\mathbf{Tate}$ $\mathbf{construction}$ $\mathbf{lowers}$ $\mathbf{Bousfield}$ $\mathbf{class}$)\label{GTglbc}
Let $m$ be a positive integer and $E$ be a $p$-complete, complex oriented spectrum with an associated formal group of height $n$. Let $A$ be an abelian $p$-group of form $\mathbb{Z}/{p^{i_1}}\oplus \cdots \oplus \mathbb{Z}/{p^{i_m}}$ and $C$ be its subgroup $\mathbb{Z}/{p^{j_1}}\oplus \cdots \oplus \mathbb{Z}/{p^{j_m}}$ with $i_k\leq j_k$ for $1\leq k\leq m$. There is a group homomorphism $\phi$ from $A/C$ to $A$ as follows:
\begin{align*}
\phi:\mathbb{Z}/{p^{i_1-j_1}}\oplus \mathbb{Z}/{p^{i_2-j_2}}\oplus \cdots \oplus \mathbb{Z}/{p^{i_m-j_m}}&\rightarrow \mathbb{Z}/{p^{i_1}}\oplus\mathbb{Z}/{p^{i_2}}\oplus\cdots\oplus \mathbb{Z}/{p^{i_m}}\\
(w_1,w_2,\cdots,w_m)&\mapsto (p^{i_1-j_1}w_1,p^{i_2-j_2}w_2,\cdots, p^{i_m-j_m}w_m).
\end{align*} If $E$ is Landweber exact, then
\begin{enumerate}
\item[{\rm (i)}]$\cat T_{A,C}(E)$ is Landweber exact;
\item[{\rm (ii)}]$\cat T_{A,C}(E)$ is at least $v_{n-{\rm rank}_p(C)}$-periodic and at most $v_{n-t}$-periodic;
\item[{\rm (iii)}]$\langle\cat T_{A,C}(E)\rangle=\langle E(n-s_{A,C;E})\rangle$ for some integer $s_{A,C;E}$ with $t\leq s_{A,C;E}\leq {\rm rank}_p(C)$, When $k>n$, $ E(n-k)=*$.
\end{enumerate}
Where
$$t=\max_{j\in \mathbb{N}^+}{\lceil\frac{\log_p|V(p^j|A)|-\log_p|V(p^j|{\rm im}\phi(A/C))|}{j}\rceil}.$$
Especially, if $A$ is a finite abelian $p$-group and $C$ is its direct summand, then the blue-shift number $s_{A,C;E}={\rm rank}_p(C)$; $A=\mathbb{Z}/p^j$ and $C$ is a non-trivial subgroup, then the blue-shift number $s_{A,C;E}=1$. However, the upper bound $t$ does not always equal ${\rm rank}_p(C)$. For example, $A=\mathbb{Z}/p^2\oplus\mathbb{Z}/p^2\oplus\mathbb{Z}/p^2$ and
$C=\mathbb{Z}/p\oplus\mathbb{Z}/p\oplus\mathbb{Z}/p$, then $t=2$ but ${\rm rank}_p(C)=3$.
\end{Thm}
The $\rm (i)$ of Theorem \ref{GTglbc} is proved by Lemma \ref{LE}. By Theorem \ref{En}, the $\rm (i)$ and $\rm (ii)$ of Theorem \ref{GTglbc} imply the $\rm (iii)$ of Theorem \ref{GTglbc}. It remains to prove the $\rm (ii)$ of Theorem \ref{GTglbc}, and by Lemma \ref{Max}, it is equivalent to $t\leq \mathbf{s}_{A,C;E}\leq {\rm rank}_p(C)$ where
$$t=\max_{j\in \mathbb{N}^+}{\lceil\frac{\log_p|V(p^j|A)|-\log_p|V(p^j|{\rm im}\phi(A/C))|}{j}\rceil}.$$
And we divide its proof into three cases:
\begin{enumerate}
\item[(1)]$A=C$ is any elementary abelian $p$-group;
\item[(2)]$A=C$ is any general abelian $p$-group;
\item[(3)]$A$ is any general abelian $p$-group and $C$ is its proper subgroup.
\end{enumerate}
Although $(1)$ is a special case of $(2)$, the whole proof for the case $(1)$ is inspiring and the proof for the upper bound of $\mathbf{s}_{A,A;E}$ is different from the corresponding proof for the case $(2)$. For all above three cases, the key proof lies in the looking for lower bounds of $\mathbf{s}_{A,C;E}$. If we could find some-tuple of $[p^j]_E(x)$ or its Weierstrass polynomial $g_j(x)$ (In this section, we do not distinguish between $[p^j]_E(x)$ and $g_j(x)$) in $\pi_*(\cat T_{A,C}(E))$, then by Corollary \ref{vanish} we get a lower bound of $\mathbf{s}_{A,C;E}$.

\subsection{Proof for the case (1) $A=C$ is an elementary abelian $p$-group \label{Ivrop}}
Let $A$ be an elementary abelian group with ${\rm rank}_p(A)=m$. From Proposition \ref{Shgg} and Theorem \ref{bgcoh}, it
follows that
$$\pi_{*}(\cat T_{A,A}(E))\cong L_A^{-1}E^*\psb{x_1,\cdots,x_{m}}/([p]_E(x_1),\cdots,[p]_E(x_m)),$$
where the multiplicatively closed set $L_A$ is generated by the set
$$M_A=\{\alpha_{(w_1,w_2,\cdots,w_m)}\mid(w_1,w_2,\cdots,w_m)\in A-\{e\}=A^*\}.$$
And we have
\begin{align*}
\pi_{*}(\cat T_{A,A}(E))/I_{n+1-q}\cong\widetilde{L}_{A,n+1-q}^{-1}E^*/I_{n+1-q}\psb{x_1,\cdots,x_{m}}/([p]_E(x_1),\cdots,[p]_E(x_m)),
\end{align*}
where the multiplicatively closed set $\widetilde{L}_{A,n+1-q}$ is mod $I_{n+1-q}$ reduction of $L_A$ and generated by the set
$$\widetilde{M}_{A,n+1-q}=\{\widetilde{\alpha}_{(w_1,w_2,\cdots,w_m)} \mid (w_1,w_2,\cdots,w_m)\in A^*\}.$$
Note that
$$[p]_{E}(x)=v_{n+1-q}x^{p^{n+1-q}}+v_{n+2-q}x^{p^{n+2-q}}+\cdots+v_{n}x^{p^n}\in \pi_{*}(\cat T_{A,A}(E))/I_{n+1-q}[x].$$
Let $g_{1,n+1-q}(x)=v_{n+1-q}x+v_{n+2-q}x^{p}+\cdots+v_{n}x^{p^{q-1}}$, then $[p]_{E}(x)=g_{1,n+1-q}(x^{p^{n+1-q}}) \mod I_{n+1-q}$. The following lemma gives a $p^m$-tuple of $[p]_{E}(x)$ in $\pi_{*}(\cat T_{A,A}(E))$ under the assumption that $\pi_{*}(\cat T_{A,A}(E))\neq 0$.
\begin{Lem}\label{nrp}
If $\pi_{*}(\cat T_{A,A}(E))\neq 0$, then ${}_{p\!}F(\pi_{*}(\cat T_{A,A}(E)))$ is a $p^m$-tuple of $[p]_{E}(x)$ in $\pi_{*}(\cat T_{A,A}(E))$. Furthermore, follow the notation in \cite[Lemma 6.3]{HKR}, for $a, b\in\pi_{*}(\cat T_{A,A}(E))$, we will write $a\sim b$ if $a=\varepsilon\cdot b$ where $\varepsilon$ is a unit in $\pi_{*}(\cat T_{A,A}(E))$, let ${}_{p\!}F(\pi_{*}(\cat T_{A,A}(E)))/\sim$ denote the set of all equivalent classes, then ${}_{p\!}F(\pi_{*}(\cat T_{A,A}(E)))/\sim$ is an abelian group.
\end{Lem}
\begin{proof}
By Theorem \ref{fpjr}, we have
$${}_{p\!}F(\pi_{*}(\cat T_{A,A}(E)))\cong\{{\alpha_{(w_1,w_2,\cdots,w_m)}}\in\pi_{*}(\cat T_{A,A}(E))\mid(w_1,w_2,\cdots,w_m)\in A\}.$$
To prove that ${}_{p\!}F(\pi_{*}(\cat T_{A,A}(E)))$ is a $|{}_{p\!}F(\pi_{*}(\cat T_{A,A}(E)))|$-tuple of $[p]_{E}(x)$ in $\pi_{*}(\cat T_{A,A}(E))$, we first check that ${}_{p\!}F(\pi_{*}(\cat T_{A,A}(E)))$ is a set of roots of $[p]_{E}(x)$.
By Proposition \ref{pps}, we have for $(w_1,w_2,\cdots,w_m)\in A$, $(pw_1,pw_2,\cdots,pw_m)=0$ and
\begin{align*}
[p]_{E}(\alpha_{(w_1,w_2,\cdots,w_m)})=&[p]_{E}([w_1]_{E}(x_1)+_F [w_2]_{E}(x_2)+_F \cdots +_F [w_m]_{E}(x_m))\\
=&[pw_1]_{E}(x_1)+_F [pw_2]_{E}(x_2)+_F \cdots +_F [pw_m]_{E}(x_m)=0.
\end{align*}
Then we check that the difference of any two elements of ${}_{p\!}F(\pi_{*}(\cat T_{A,A}(E)))$ is not a zero divisor in $\pi_{*}(\cat T_{A,A}(E))$.
From the formula $x-_F y=(x-y)\cdot \varepsilon(x,y)$, where $x,y\in {}_{p\!}F(\pi_{*}(\cat T_{A,A}(E)))$, $\varepsilon(x,y)$ is a unit in $\pi_{*}(\cat T_{A,A}(E))$, it follows that
\begin{align*}
&(\alpha_{(u_1,u_2,\cdots,u_m)}-\alpha_{(w_1,w_2,\cdots,w_m)})\cdot\varepsilon(\alpha_{(u_1,u_2,\cdots,u_m)},\alpha_{(w_1,w_2,\cdots,w_m)})\\
=&\alpha_{(u_1,u_2,\cdots,u_m)}-_F \alpha_{(w_1,w_2,\cdots,w_m)}=\alpha_{(u_1-w_1,u_2-w_2,\cdots,u_m-w_m)},
\end{align*}
where $\varepsilon(\alpha_{(u_1,u_2,\cdots,u_m)},\alpha_{(w_1,w_2,\cdots,w_m)})$ is a unit in $\pi_{*}(\cat T_{A,A}(E))$.
Since $\pi_{*}(\cat T_{A,A}(E))\neq 0$ and $(u_1,u_2,\cdots,u_m)\neq(w_1,w_2,\cdots,w_m)$, $\alpha_{(u_1-w_1,u_2-w_2,\cdots,u_m-w_m)}\in L_A$ is not zero or zero-divisor in $\pi_{*}(\cat T_{A,A}(E))$. So ${}_{p\!}F(\pi_{*}(\cat T_{A,A}(E)))$ is a $p^m$-tuple of $[p]_{E}(x)$ in $\pi_{*}(\cat T_{A,A}(E))$.

Finally, we give ${}_{p\!}F(\pi_{*}(\cat T_{A,A}(E)))/\sim$ an abelian group structure:
\begin{enumerate}
\item[{\rm (i)}]Addition: $\alpha_{(u_1,u_2,\cdots,u_m)}+\alpha_{(w_1,w_2,\cdots,w_m)}\sim\alpha_{(u_1+w_1,u_2+w_2,\cdots,u_m+w_m)}$;
\item[{\rm (ii)}]Inverse: $-\alpha_{(w_1,w_2,\cdots,w_m)}\sim\alpha_{(-w_1,-w_2,\cdots,-w_m)}$.
\end{enumerate}
This completes the proof.
\end{proof}

The following lemma gives a $p^m$-tuple of $g_{1,n+1-q}(x)$ in $\pi_{*}(\cat T_{A,A}(E))/I_{n+1-q}$ under the assumption that $\pi_{*}(\cat T_{A,A}(E))/I_{n+1-q}\neq 0$.
\begin{Lem}\label{ntp}
Let ${}_{p\!}F(\pi_{*}(\cat T_{A,A}(E))/I_{n+1-q})^{p^{n+1-q}}$ denote the subset
$$\{\widetilde{\alpha}^{p^{n+1-q}}_{(w_1,w_2,\cdots,w_m)}\in \pi_{*}(\cat T_{A,A}(E))/I_{n+1-q}\mid (w_1,w_2,\cdots,w_m)\in  A\}.$$
If $\pi_{*}(\cat T_{A,A}(E))/I_{n+1-q}\neq 0$, then ${}_{p\!}F(\pi_{*}(\cat T_{A,A}(E))/I_{n+1-q})^{p^{n+1-q}}$ is a $p^m$-tuple of $g_{1,n+1-q}(x)$ in $\pi_{*}(\cat T_{A,A}(E))/I_{n+1-q}$, and ${}_{p\!}F(\pi_{*}(\cat T_{A,A}(E))/I_{n+1-q})^{p^{n+1-q}}/\sim$ is an abelian group.
\end{Lem}
\begin{proof}
Note that
$$g_{1,n+1-q}(\widetilde{\alpha}^{p^{n+1-q}}_{(w_1,w_2,\cdots,w_m)})=[p]_{E}(\alpha_{(w_1,w_2,\cdots,w_m)}) \mod I_{n+1-q},$$
so $\{\widetilde{\alpha}^{p^{n+1-q}}_{(w_1,w_2,\cdots,w_m)}\mid (w_1,w_2,\cdots,w_m)\in A\}$ is a set of roots of $g_{1,n+1-q}(x)$ in $\pi_{*}(\cat T_{A,A}(E))/I_{n+1-q}$. For any two different elements $\widetilde{\alpha}^{p^{n+1-q}}_{(u_1,u_2,\cdots,u_m)}, \widetilde{\alpha}^{p^{n+1-q}}_{(w_1,w_2,\cdots,w_m)}\in {}_{p\!}F(\pi_{*}(\cat T_{A,A}(E))/I_{n+1-q})^{p^{n+1-q}}$, we have
$$0\neq\widetilde{\alpha}^{p^{n+1-q}}_{(u_1,u_2,\cdots,u_m)}
-\widetilde{\alpha}^{p^{n+1-q}}_{(w_1,w_2,\cdots,w_m)}=(\widetilde{\alpha}_{(u_1,u_2,\cdots,u_m)}-\widetilde{\alpha}_{(w_1,w_2,\cdots,w_m)})^{p^{n+1-q}}$$
for the coefficient $\mathbb{F}_{p}$. Since $\pi_{*}(\cat T_{A,A}(E))/I_{n+1-q}\neq 0$ and
\begin{align*}
\widetilde{\alpha}_{(u_1,u_2,\cdots,u_m)}-\widetilde{\alpha}_{(w_1,w_2,\cdots,w_m)}=\varepsilon^{-1}(\widetilde{\alpha}_{(u_1,u_2,\cdots,u_m)},
\widetilde{\alpha}_{(w_1,w_2,\cdots,w_m)})\cdot\widetilde{\alpha}_{(u_1-w_1,u_2-w_2,\cdots,u_m-w_m)} \in \widetilde{L}_{A,q},
\end{align*}
$(\widetilde{\alpha}_{(u_1,u_2,\cdots,u_m)}-\widetilde{\alpha}_{(w_1,w_2,\cdots,w_m)})^{p^{n+1-q}}$ is not zero or zero-divisor in $\pi_{*}(\cat T_{A,A}(E))/I_{n+1-q}$. Therefore, ${}_{p\!}F(\pi_{*}(\cat T_{A,A}(E))/I_{n+1-q})^{p^{n+1-q}}$ is a $p^m$-tuple of $g_{1,n+1-q}(x)$ in $\pi_{*}(\cat T_{A,A}(E))/I_{n+1-q}$.

${}_{p\!}F(\pi_{*}(\cat T_{A,A}(E))/I_{n+1-q})^{p^{n+1-q}}/\sim$ has an abelian group structure:
\begin{enumerate}
\item[{\rm (i)}]Addition: $\widetilde{\alpha}^{p^{n+1-q}}_{(u_1,u_2,\cdots,u_m)}+\widetilde{\alpha}^{p^{n+1-q}}_{(w_1,w_2,\cdots,w_m)}\sim\widetilde{\alpha}^{p^{n+1-q}}_{(u_1+w_1,u_2+w_2,\cdots,u_m+w_m)}$;
\item[{\rm (ii)}]Inverse: $-\widetilde{\alpha}^{p^{n+1-q}}_{(w_1,w_2,\cdots,w_m)}\sim\widetilde{\alpha}^{p^{n+1-q}}_{(-w_1,-w_2,\cdots,-w_m)}$.
\end{enumerate}
This completes the proof.
\end{proof}

For any $q\leq n+1$, there is a surjective map $\theta_q:A \rightarrow {}_{p\!}F(\pi_{*}(\cat T_{A,A}(E))/I_{n+1-q})^{p^{n+1-q}}$ that maps $(w_1,w_2,\cdots,w_m)$ to $\widetilde{\alpha}^{p^{n+1-q}}_{(w_1,w_2,\cdots,w_m)}$, then we have
\begin{Lem}\label{ijnt}
$\theta_q$ is a bijection if and only if $\pi_{*}(\cat T_{A,A}(E))/I_{n+1-q}\neq 0$.
\end{Lem}
\begin{proof}
$\Rightarrow:$ Since $\theta_q$ is a bijection, then for $(u_1,u_2,\cdots,u_m)\neq(w_1,w_2,\cdots,w_m)\in A$,
$$0\neq\widetilde{\alpha}^{p^{n+1-q}}_{(u_1,u_2,\cdots,u_m)}-\widetilde{\alpha}^{p^{n+1-q}}_{(w_1,w_2,\cdots,w_m)}\in\pi_{*}(\cat T_{A,A}(E))/I_{n+1-q},$$
which implies that $\pi_{*}(\cat T_{A,A}(E))/I_{n+1-q}\neq 0$.

$\Leftarrow:$ We only have to prove that $\theta_q$ is injective. Since $\pi_{*}(\cat T_{A,A}(E))/I_{n+1-q}\neq 0$, then for any $(w_1,w_2,\cdots,w_m)\in A^*$, $0\neq \widetilde{\alpha}^{p^{n+1-q}}_{(w_1,w_2,\cdots,w_m)}\in  \widetilde{L}_{A,q}$. So if $(u_1,u_2,\cdots,u_m)\neq(w_1,w_2,\cdots,w_m)\in A$, then
$$\widetilde{\alpha}^{p^{n+1-q}}_{(u_1,u_2,\cdots,u_m)}-\widetilde{\alpha}^{p^{n+1-q}}_{(w_1,w_2,\cdots,w_m)}=(\varepsilon^{-1}(\widetilde{\alpha}_{(u_1,u_2,\cdots,u_m)},
\widetilde{\alpha}_{(w_1,w_2,\cdots,w_m)})\cdot\widetilde{\alpha}_{(u_1-w_1,u_2-w_2,\cdots,u_m-w_m)})^{p^{n+1-q}}\neq 0,$$
thus $\theta_q$ is injective.
\end{proof}

When $q=n+1$, $I_0=(0)$ and ${}_{p\!}F(\pi_{*}(\cat T_{A,A}(E))/I_{n+1-q})^{p^{n+1-q}}={}_{p\!}F(\pi_{*}(\cat T_{A,A}(E)))$.
\begin{Lem}\label{}
${}_{p\!}F(\pi_{*}(\cat T_{A,A}(E)))$ is an abelian group and $\theta_{n+1}$ is an abelian group homomorphism. If $n< m$, then $\theta_{n+1}$ is trivial and ${}_{p\!}F(\pi_{*}(\cat T_{A,A}(E)))\cong e$.
\end{Lem}
\begin{proof}
The group structure of ${}_{p\!}F(\pi_{*}(\cat T_{A,A}(E)))$ is induced by the formal group law of $E$, and for any two elements $\alpha_{(u_1,u_2,\cdots,u_m)},\alpha_{(w_1,w_2,\cdots,w_m)}\in {}_{p\!}F(\pi_{*}(\cat T_{A,A}(E)))$, their sum is defined by
$$\alpha_{(u_1,u_2,\cdots,u_m)}+_F\alpha_{(w_1,w_2,\cdots,w_m)}=\alpha_{(u_1+w_1,u_2+w_2,\cdots,u_m+w_m)}.$$
Then $\theta_{n+1}$ is an abelian group homomorphism.

If $n<m$, we assume that $\pi_{*}(\cat T_{A,A}(E))\neq 0$. By Lemma \ref{ijnt}, $\theta_{n+1}$ is a bijection and $|{}_{p\!}F(\pi_{*}(\cat T_{A,A}(E)))|=p^m$. Then ${}_{p\!}F(\pi_{*}(\cat T_{A,A}(E)))$ is a $p^m$-tuple of $[p]_{E}(x)$ in $\pi_{*}(\cat T_{A,A}(E))$. Note that $1\in (p,v_1,\cdots,v_n)$ and $\deg_W [p]_{E}(x)=p^n < p^m$. By Corollary \ref{vanish}, we have $\pi_{*}(\cat T_{A,A}(E))=0$. Then $\theta_{n+1}$ is trivial and ${}_{p\!}F(\pi_{*}(\cat T_{A,A}(E)))\cong e$.
\end{proof}

\begin{Cor}\label{lbTAA}
$\pi_{*}(\cat T_{A,A}(E))/I_{n+1-q}=0$ for $q< m+1$, which implies that
$\mathbf{s}_{A,A;E}\geq m$.
\end{Cor}
\begin{proof}
Assume that there exists $q_0< m+1$ such that $\pi_{*}(\cat T_{A,A}(E))/I_{n+1-q_0}\neq 0$. By Lemma \ref{ijnt}, $\theta_{q_0}$ is a bijection and hence $|{}_{p\!}F(\pi_{*}(\cat T_{A,A}(E))/I_{n+1-q_0})^{p^{n+1-q_0}}|=p^m$. Then ${}_{p\!}F(\pi_{*}(\cat T_{A,A}(E))/I_{n+1-q_0})^{p^{n+1-q_0}}$ is a $p^m$-tuple of $g_{1,n+1-q_0}(x)$ in $\pi_{*}(\cat T_{A,A}(E))/I_{n+1-q_0}$. Note that $p^m> \deg g_{1,n+1-q_0}(x)=p^{q_0-1}$ and $1\in (v_{n+1-q_0},\cdots,v_n)$. So by Corollary \ref{vanish}, we have
$\pi_{*}(\cat T_{A,A}(E))/I_{n+1-q_0}=0$.
\end{proof}

Although by Corollary \ref{lbTAA} and the exactness of
$$\CD \pi_{*}(\cat T_{A,A}(E))/I_{n-m} @>\cdot v_{n-m}>> \pi_{*}(\cat T_{A,A}(E))/I_{n-m} @>>>\pi_{*}(\cat T_{A,A}(E))/I_{n+1-m}\endCD,$$
we know that $v_{n-m}$ is a unit in $\pi_{*}(\cat T_{A,A}(E))/I_{n-m}$. To achieve our main idea, here we give another proof of this fact by using Theorem \ref{Fac-tuple}. Let $q=m+1$, we have
\begin{Lem}\label{So}
Let $n\geq m$, then
\begin{enumerate}
\item[\rm (i)] $$v_{n-m}=(-1)^{p^m-1}v_n \prod_{(w_1,w_2,\cdots,w_m)\in A^*}{\widetilde{\alpha}^{p^{n-m}}_{(w_1,w_2,\cdots,w_m)}},$$
\item[\rm (ii)] $$0=(-1)^{p^m-2}v_n \sum_{w^{(1)}\neq w^{(2)}\neq\cdots\neq w^{(p^m-2)}\in A^*}
    {\widetilde{\alpha}^{p^{n-m}}_{w^{(1)}}\widetilde{\alpha}^{p^{n-m}}_{w^{(2)}}\cdots \widetilde{\alpha}^{p^{n-m}}_{w^{(p^m-2)}}},$$
 $$\vdots$$
\item[\rm (iii)]
%$$v_{n-k+1}=(-1)^{p^k-p+1}v_n \prod_{(m_1,m_2,\cdots,m_k)\in(\mathbb{Z}/p)^k-0}{\alpha^{p^{n-k}}_{(m_1,m_2,\cdots,m_k)}}
 %   \sum_{(m^{(1)}_1,m^{(1)}_2,\cdots,m^{(1)}_k),(m^{(2)}_1,m^{(2)}_2,\cdots,m^{(2)}_k),\cdots,(m^{(p-1)}_1,m^{(p-1)}_2,\cdots,m^{(p-1)}_k)
   % \in(\mathbb{Z}/p)^k-0}{\alpha^{-p^{n-k}}_{(m^{(1)}_1,m^{(1)}_2,\cdots,m^{(1)}_k)}\alpha^{-p^{n-k}}_{(m^{(2)}_1,m^{(2)}_2,\cdots,m^{(2)}_k)}
   % \cdots \alpha^{-p^{n-k}}_{(m^{(p-1)}_1,m^{(p-1)}_2,\cdots,m^{(p-1)}_k)}},$$
    $$v_{n-i}=(-1)^{p^m-p^{m-i}}v_n  \sum_{w^{(1)}\neq w^{(2)}\neq\cdots\neq w^{p^{m-i}}\in A^*}
    {\widetilde{\alpha}^{p^{n-m}}_{w^{(1)}}\widetilde{\alpha}^{p^{n-m}}_{w^{(2)}}\cdots \widetilde{\alpha}^{p^{n-m}}_{w^{p^{m-i}}}},$$
$$\vdots$$
\item[\rm (iv)]$$0=-v_n \sum_{(w_1,w_2,\cdots,w_m)\in A^*}{\widetilde{\alpha}^{p^{n-m}}_{(w_1,w_2,\cdots,w_m)}},$$
\end{enumerate}
and the right side of the top equality is invertible in $\pi_{*}(\cat T_{A,A}(E))/I_{n-m}$.
\end{Lem}

\begin{Rem}\label{}
Since $\pi_{*}(\cat T_{A,A}(E))/I_{n-m}$ may be $0$, the fact that $v_{n-m}$ is invertible in $\pi_{*}(\cat T_{A,A}(E)))/I_{n-m}$ does not imply that $\cat T_{A,A}(E))$ is $v_{n-m}$-periodic, but implies
that $\cat T_{A,A}(E)$ is at most $v_{n-m}$-periodic.
\end{Rem}
\begin{proof}
If $\pi_{*}(\cat T_{A,A}(E))/I_{n-m}=0$, obviously this is true; if $\pi_{*}(\cat T_{A,A}(E))/I_{n-m}\neq 0$, then by Lemma \ref{ijnt}, we obtain that $\theta_{m+1}$ is a bijection and $|{}_{p\!}F(\pi_{*}(\cat T_{A,A}(E))/I_{n-m})^{p^{n-m}}|=p^m$. So $\pi_{*}(\cat T_{A,A}(E))/I_{n-m}$ has a $p^m$-tuple ${}_{p\!}F(\pi_{*}(\cat T_{A,A}(E))/I_{n-m})^{p^{n-m}}$ of $g_{1,n-m}(x)$. Then by Theorem \ref{Fac-tuple}, we have
$$v_{n-m}x+v_{n-m+1}x^p+\cdots+v_{n}x^{p^m}=v_n\prod_{(w_1,w_2,\cdots,w_m)\in A}(x-\widetilde{\alpha}^{p^{n-m}}_{(w_1,w_2,\cdots,w_m)})\in \pi_{*}(\cat T_{A,A}(E))/I_{n-m}[x].$$
\end{proof}

We get the upper bound $m$ of $\mathbf{s}_{A,A;E}$ by using Lemma \ref{ind}, and delay its proof. Then by Corollary \ref{lbTAA}, we have
\begin{Thm}\label{}
Let $A$ be a elementary abelian $p$-group with ${\rm rank}_p(A)=m$, then $\mathbf{s}_{A,A;E}=m$.
\end{Thm}
To show an application of our linear equation theory  over a commutative ring in Section \ref{GRRC}, we give another way to get the upper bund of $\mathbf{s}_{A,A;E}$ for the case $E=E(n)$. Using the approach in Section \ref{GRRC}, we generalize Ando--Morava--Sadofsky's theorem \cite[Proposition 2.3]{AMS} from $\mathbb{Z}/p$ to any elementary abelian $p$-group.
\begin{Thm}\label{splite}
\begin{align*}
\pi_*(\cat T_{A,A}(BP\langle n\rangle))&\cong_{\phi} L_A'^{-1}BP\langle n-m \rangle_*\psb{x_{1},\cdots,x_{m}},
\end{align*}
where $\phi$ is the ring isomorphism constructed in the following proof, and the multiplicatively closed set $L_A'$ is generated by the set
\begin{align*}
\{&\phi(\alpha_{(w_1,\cdots,w_m)})\mid\alpha_{(w_1,\cdots,w_m)}=[w_1]_{BP\langle n \rangle}(x_1)+_F \cdots +_F [w_m]_{BP\langle n \rangle}(x_m), (w_1,\cdots,w_m)\in A^*\}.
\end{align*}
\end{Thm}
\begin{proof}
As similar to Theorem \ref{bgcoh}, replacing $E$ by $BP\langle n\rangle$, we have
$$\pi_*(\cat T_{A,A}(BP\langle n\rangle)))\cong L_A^{-1}BP\langle n\rangle^*\psb{x_1,\cdots,x_{m}}/([p]_{BP\langle n\rangle}(x_1),\cdots,[p]_{BP\langle n\rangle}(x_m)),$$
where the multiplicatively closed set $L_A$ is generated by the set
\begin{align*}
\{&\alpha_{(w_1,\cdots,w_m)}=[w_1]_{BP\langle n\rangle}(x_1)+_F \cdots +_F [w_m]_{BP\langle n\rangle}(x_m)\mid(w_1,\cdots,w_m)\in A^*\}.
\end{align*}
We always require a ring map to map 1 to 1. First, we construct a ring map
$$\phi: \pi_*(\cat T_{A,A}(BP\langle n\rangle))\rightarrow L_A'^{-1}BP\langle n-m \rangle_*\psb{x_{1},\cdots,x_{m}},$$
which send $v_i$ to $v_i$ $(0\leq i\leq n-m)$, $x_j$ to $x_j$ $(1\leq j\leq m)$, and send $[p]_{BP\langle n\rangle}(x_k)$ to $0$ for $1\leq k\leq m$, then we have a system of non-homogeneous $L_A'^{-1}BP\langle n-m \rangle_*\psb{x_{1},\cdots,x_{m}}$-linear equations $\{\phi([p]_{BP\langle n\rangle}(x_i))=0, 1\leq i\leq m\}$. We view $\phi([p]_{BP\langle n\rangle}(x_i))=0$ as a non-homogeneous linear equation
$$x_i^{p^{n-m+1}}\phi(v_{n-m+1})+x_i^{p^{n-m+2}}\phi(v_{n-m+2})+\cdots+x_i^{p^{n}}\phi(v_{n})=-(v_0x_i+v_1x_i^p+\cdots+v_{n-m}x_i^{p^{n-m}})$$
with variables $\phi(v_{n-m+1}), \phi(v_{n-m+2}), \cdots,\phi(v_{n})$. Since $x_i $ is invertible for $1\leq i\leq m$, one may use Gaussian elimination to get the unique solution of $\phi(v_{n-m+1}), \phi(v_{n-m+2}), \cdots,\phi(v_{n})$. Then we define $\phi(v_i)$ as the solution of $\phi(v_i)$ for $n-m+1\leq i\leq n$, So $\phi$ is a well-defined ring map.
There is a map
$$\varphi: L_A'^{-1}BP\langle n-m \rangle_*\psb{x_{1},\cdots,x_{m}}\rightarrow \pi_*(\cat T_{A,A}(BP\langle n\rangle))$$
defined in the obvious way, that becomes an inverse map.
\end{proof}

Since there is a map: $BP\langle n\rangle\rightarrow v_n^{-1}BP\langle n\rangle\simeq E(n)$, by Theorem \ref{splite}, we use the ring isomorphism $\phi$ to give the following ring isomorphism:
\begin{Cor}\label{isomk}
Let $A$ be an elementary abelian $p$-group with ${\rm rank}_p(A)=m$. If $n\geq m$, then
\begin{align*}
\pi_*(\cat T_{A,A}(E(n)))/I_{n-m}&\cong_{\phi}L_A'^{-1}E(n-m)^*\psb{x_{1},\cdots,x_{m}}/I_{n-m} \cong L_A'^{-1}K(n-m)_*\psb{x_{1},\cdots,x_{m}},
\end{align*}
where $\phi$ is the ring isomorphism constructed in the proof of Theorem \ref{splite}, and the multiplicatively closed set $L_A'$ is generated by the set
\begin{align*}
\{&\phi(\widetilde{\alpha}_{(w_1,\cdots,w_m)})\mid\widetilde{\alpha}_{(w_1,\cdots,w_m)}=[w_1]_{E}(x_1)+_F \cdots +_F [w_m]_{E}(x_m), (w_1,\cdots,w_m)\in A^*\}.
\end{align*}
\end{Cor}

Note that if $n\geq m$, $L_A'^{-1}BP\langle n-m \rangle_*\psb{x_{1},\cdots,x_{m}}$ is non-trivial, then by Corollary \ref{isomk},
we have
\begin{Cor}\label{ubTAA}
Let $A$ be an elementary abelian $p$-group with ${\rm rank}_p(A)=m$. If $n\geq m$, then $\pi_{*}(\cat T_{A,A}(E(n)))/I_{n-m}\neq 0$.
\end{Cor}

\subsection{Proof for the case (2) $A=C$ is a general abelian $p$-group\label{pole}}
In Subsection \ref{Ivrop}, we devise a powerful tool in the proof for the case (1), which is the $|{}_{p\!}F(\pi_{*}(\cat T_{A,A}(E)))|$-tuple ${}_{p\!}F(\pi_{*}(\cat T_{A,A}(E)))$ of $[p]_{E}(x)$ in $\pi_{*}(\cat T_{A,A}(E))$. Certainly, this tool can also be used to explain the general blue-shift phenomenon (Conjecture \ref{Gbsp}). More generally, it is natural to consider $|{}_{p^j\!}F(\pi_{*}(\cat T_{A,A}(E)))|$-tuple ${}_{p^j\!}F(\pi_{*}(\cat T_{A,A}(E)))$ of $[p^j]_{E}(x)$ in $\pi_{*}(\cat T_{A,A}(E))$ for any positive integer $j$. Then we could use this tuple of $[p^j]_{E}(x)$ to get the solution of some $v_i$, and investigate whether $v_i$ is invertible by the invertible roots of $[p^j]_{E}(x)$ in this tuple.
Recall that
$$[p]_{E}(x)=v_{n+1-q}x^{p^{n+1-q}}+v_{n+2-q}x^{p^{n+2-q}}+\cdots+v_{n}x^{p^n}\in \pi_{*}(\cat T_{A,A}(E))/I_{n+1-q}[x].$$
Then there is a natural problem of how to compute the $p^j$-series $[p^j]_{E}(x)$. There is an iteration
formula $[p^j]_{E}(x)=[p]_{E}([p^{j-1}]_{E}(x))$. However, it is too difficult to obtain an accurate formula for $[p^j]_{E}(x)$. This may be one reason why the generalization of previous work to finite abelian groups is hard. But we can deal with $[p^j]_{E}(x)$. The major key insight of our breakthrough is that instead of trying to obtain an accurate formula of $[p^j]_{E}(x)$, it only suffices to compute the leading and the last terms of $[p^j]_{E}(x)$ in $E^*/I_{n+1-q}[x]$, as indicated by the method we used in Subsection \ref{Ivrop}.

Without loss of generality, we may suppose that $A$ is $\mathbb{Z}/{p^{i_1}}\oplus \mathbb{Z}/{p^{i_2}}\oplus \cdots \oplus \mathbb{Z}/{p^{i_m}}$. From Proposition \ref{Shgg} and Theorem \ref{bgcoh}, it
follows that $$\pi_{*}(\cat T_{A,A}(E))\cong L_A^{-1}E^*\psb{x_1,\cdots,x_{m}}/([p^{i_1}]_{E}(x_1),\cdots,[p^{i_m}]_{E}(x_m)),$$
where the multiplicatively closed set $L_A$ is generated by the set
$$M_A=\{\alpha_{(w_1,w_2,\cdots,w_m)}\mid(w_1,w_2,\cdots,w_m)\in A^*\}.$$

Then for $q\leq n+1$, we have
\begin{align*}
\pi_{*}(\cat T_{A,A}(E))/I_{n+1-q}\cong\widetilde{L}_{A,n+1-q}^{-1}E^*/I_{n+1-q}\psb{x_1,\cdots,x_{m}}/([p^{i_1}]_{E}(x_1),\cdots,[p^{i_m}]_{E}(x_m)),
\end{align*}
where the multiplicatively closed set $\widetilde{L}_{A,n+1-q}$ is mod $I_{n+1-q}$ reduction of $L_A$ and generated by the set
$$\widetilde{M}_{A,n+1-q}=\{\widetilde{\alpha}_{(w_1,w_2,\cdots,w_m)} \mid (w_1,w_2,\cdots,w_m)\in A^*\}.$$

\begin{Lem}\label{ntugts}
Let $A$ be a finite abelian $p$-group. If $\pi_{*}(\cat T_{A,A}(E))\neq 0$, then ${}_{p^{\infty}\!}F(\pi_{*}(\cat T_{A,A}(E)))$ is an $|A|$-tuple of $\pi_{*}(\cat T_{A,A}(E))$, and ${}_{p^\infty\!}F(\pi_{*}(\cat T_{A,A}(E)))/\sim$ is an abelian group.
\end{Lem}
\begin{proof}
The proof is similar to the proof of Lemma \ref{nrp}. By direct checking of the definition, we conclude that ${}_{p^\infty}F(\pi_{*}(\cat T_{A,A}(E)))$ is an $|A|$-tuple of $\pi_*(\cat T_{A,A}(E))$ under the assumption that $\pi_{*}(\cat T_{A,A}(E))\neq 0$.
\end{proof}

\begin{Lem}\label{Re}
Let $V(p^j|A)$ denote the subgroup $\{a\in  A\mid p^ja=0\}$ of $A$. If $\pi_{*}(\cat T_{A,A}(E))\neq 0$, then ${}_{p^j\!}F(\pi_{*}(\cat T_{A,A}(E)))$ is a $|V(p^j|A)|$-tuple of $[p^j]_{E}(x)$ in $\pi_{*}(\cat T_{A,A}(E))$, and ${}_{p^j\!}F(\pi_{*}(\cat T_{A,A}(E)))/\sim$ is an abelian group.
\end{Lem}
\begin{proof}
The proof is similar to the proof of Lemma \ref{nrp}.
\end{proof}

The following lemma shows the expression of $[p^j]_{E}(x)$ in
$\pi_{*}(\cat T_{A,A}(E))/I_{n+1-q}$.
\begin{Lem}\label{Rp}
$$[p^j]_{E}(x)= v_{n+1-q}^{1+p^{n+1-q}+\cdots+p^{(j-1)(n+1-q)}}x^{p^{j(n+1-q)}}+\cdots+v_{n}^{1+p^{n}+\cdots+p^{(j-1)n}}x^{p^{jn}}
   \in \pi_{*}(\cat T_{A,A}(E))/I_{n+1-q}[x].$$
\end{Lem}
\begin{proof}
Recall that $[p]_{E}(x)=v_{n+1-q}x^{p^{n+1-q}}+\cdots+v_{n}x^{p^n}\in E^*/I_{n+1-q}[x]$. By Proposition \ref{pps} that
$[p^j]_{E}(x)=[p]_{E}([p^{j-1}]_{E}(x))$, we obtain the leading and the last terms of $[p^j]_{E}(x)$ by iteration.
\end{proof}

We follow the method used in Subsection \ref{Ivrop}. Let $[p^j]_{E}(x)=g_{j,n+1-q}(x^{p^{j(n+1-q)}})\in E^*/I_{n+1-q}[x],$ then by Lemma \ref{pjwp} we have $g_{j,n+1-q}(x)=g_{1,n+1-q}^j(x)=a_1x+\cdots+a_{p^{j(q-1)}}x^{p^{j(q-1)}}.$

\begin{Lem}\label{}
Let ${}_{p^j\!}F(\pi_{*}(\cat T_{A,A}(E))/I_{n+1-q})^{p^{j(n+1-q)}}$ denote the subset
$$\{{\widetilde{\alpha}^{p^{j(n+1-q)}}_{(w_1,w_2,\cdots,w_m)}}\in\pi_{*}(\cat T_{A,A}(E))/I_{n+1-q}\mid (p^jw_1,p^jw_2,\cdots,p^jw_m)=0, (w_1,w_2,\cdots,w_m)\in  A\}.$$ If $\pi_{*}(\cat T_{A,A}(E))/I_{n+1-q}\neq 0$, then ${}_{p^j\!}F(\pi_{*}(\cat T_{A,A}(E))/I_{n+1-q})^{p^{j(n+1-q)}}$ is a $|V(p^j|A)|$-tuple of $g_{1,n+1-q}^j(x)$ in $\pi_{*}(\cat T_{A,A}(E))/I_{n+1-q}$, and ${}_{p^j\!}F(\pi_{*}(\cat T_{A,A}(E))/I_{n+1-q})^{p^{j(n+1-q)}}/\sim$ is an abelian group.
\end{Lem}
\begin{proof}
The proof is similar to the proof of Lemma \ref{ntp}.
\end{proof}

There is a surjective map $\theta^j_q:V(p^j|A) \rightarrow {}_{p^j\!}F(\pi_{*}(\cat T_{A,A}(E))/I_{n+1-q})^{p^{j(n+1-q)}}$ that maps $(w_1,w_2,\cdots,w_m)$ to
$\widetilde{\alpha}^{p^{j(n+1-q)}}_{(w_1,w_2,\cdots,w_m)}$.
\begin{Lem}\label{root}
$\theta^j_q$ is a bijection if and only if $\pi_{*}(\cat T_{A,A}(E))/I_{n+1-q}\neq 0$.
\end{Lem}
\begin{proof}
The proof is similar to the proof of Lemma \ref{ijnt}.
\end{proof}

If $\pi_{*}(\cat T_{A,A}(E))/I_{n+1-q}\neq 0$, then by Lemma \ref{root}, $\theta^j_q$ is a bijection for any $j\geq 1$. Combining with Lemma \ref{Re}, we have $|{}_{p^j\!}F(\pi_{*}(\cat T_{A,A}(E))/I_{n+1-q})^{p^{j(n+1-q)}}|=|V(p^j|A)|$. Then ${}_{p^j\!}F(\pi_{*}(\cat T_{A,A}(E))/I_{n+1-q})^{p^{j(n+1-q)}}$ is a $|V(p^j|A)|$-tuple of $g_{1,n+1-q}^j(x)$ in $\pi_{*}(\cat T_{A,A}(E))/I_{n+1-q}$.

\begin{Lem}\label{lbsj}
Let $j$ be any positive integer, then $\pi_{*}(\cat T_{A,A}(E))/I_{n+1-q}=0$ for $q< \frac{\log_p|V(p^{j}|A)|}{j}+1$.
\end{Lem}
\begin{proof}
Assume that there exists $j_0$ and $q_0<\frac{\log_p|V(p^{j_0}|A)|}{j_0}+1$ such that $\pi_{*}(\cat T_{A,A}(E))/I_{n+1-q_0}\neq 0$. By Lemma \ref{root}, $\theta^{j_0}_{q_0}$ is a bijection and hence $|{}_{p^{j_0}\!}F(\pi_{*}(\cat T_{A,A}(E))/I_{n+1-q_0})^{p^{j_0(n+1-q_0)}}|=|V(p^{j_0}|A)|$. Then ${}_{p^{j_0}\!}F(\pi_{*}(\cat T_{A,A}(E))/I_{n+1-q_0})^{p^{j_0(n+1-q_0)}}$ is a $|V(p^{j_0}|A)|$-tuple of $g^{j_0}_{1,n+1-q_0}(x)$ in $\pi_{*}(\cat T_{A,A}(E))/I_{n+1-q_0}$. Note that the unit $v_{n}^{1+p^{n}+\cdots+p^{(j_0-1)n}}$ is the last coefficient of $g^{j_0}_{1,n+1-q_0}(x)$, and $q_0<\frac{\log_p|V(p^{j_0}|A)|}{j_0}+1$ implies that $|V(p^{j_0}|A)|> \deg g^{j_0}_{1,n+1-q_0}(x)=p^{j_0(q_0-1)}$. So by Corollary \ref{vanish}, we have
$\pi_{*}(\cat T_{A,A}(E))/I_{n+1-q_0}=0$, which contradicts to our assumption. This completes the proof.
\end{proof}

Recall that $A$ is $\mathbb{Z}/{p^{i_1}}\oplus \mathbb{Z}/{p^{i_2}}\oplus \cdots \oplus \mathbb{Z}/{p^{i_m}}$, then we have
\begin{Lem}\label{MD}
\[\lceil\frac{\log_p|V(p^{j}|A)|}{j}\rceil=\left\{\begin{array}{ll}
=m&~~\text{if $1\leq j\leq \min\{i_1, \cdots, i_m\}$},\\
\leq m\,&
~~\text{if $j>\min\{i_1, \cdots, i_m\}$}.
\end{array}\right.\]
\end{Lem}
\begin{proof}
Note that $\log_p|V(p|A)|$ is exactly the number of $\mathbb{Z}/p$ factors in the maximal elementary abelian subgroup of $A$, then we have
$$\log_p|V(p|A)|={\rm rank}_p(A)=m.$$
Since $V(p^j|A)$ is a subgroup of $A$ and $\mathbb{Z}/{p^{j}}\oplus \cdots \oplus \mathbb{Z}/{p^{j}}$, we obtain that
$$|V(p^j|A)|\leq p^{j\log_p|V(p|A)|}~~\text{and}~~ \log_p|V(p^{j}|A)|\leq j\log_p|V(p|A)|,$$
where the equality holds if and only if $1\leq j\leq \min\{i_1, \cdots, i_m\}$. Since $\log_p|V(p|A)|$ is an integer, we have
$$\lceil\frac{\log_p|V(p^{j}|A)|}{j}\rceil\leq \log_p|V(p|A)|.$$
This completes the proof.
\end{proof}

When $q=n+1$, $I_0=(0)$ and ${}_{p^j\!}F(\pi_{*}(\cat T_{A,A}(E))/I_{n+1-q})^{p^{j(n+1-q)}}={}_{p^j\!}F(\pi_{*}(\cat T_{A,A}(E)))$.
\begin{Lem}\label{}
${}_{p^j\!}F(\pi_{*}(\cat T_{A,A}(E)))$ is an abelian group and $\theta^j_{n+1}$ is an abelian group homomorphism. If $n< m$, then $\theta^j_{n+1}$ is trivial and ${}_{p^j\!}F(\pi_{*}(\cat T_{A,A}(E)))\cong e$.
\end{Lem}
\begin{proof}
The group structure of ${}_{p^j\!}F(\pi_{*}(\cat T_{A,A}(E)))$ is induced by the formal group law of $E$, and for any two elements $\alpha_{(u_1,u_2,\cdots,u_m)},\alpha_{(w_1,w_2,\cdots,w_m)}\in {}_{p^j\!}F(\pi_{*}(\cat T_{A,A}(E)))$, their sum is defined by
$$\alpha_{(u_1,u_2,\cdots,u_m)}+_F\alpha_{(w_1,w_2,\cdots,w_m)}=\alpha_{(u_1+w_1,u_2+w_2,\cdots,u_m+w_m)}.$$
Then $\theta^j_n$ is an abelian group homomorphism.

If $n< m$, we assume that $\pi_{*}(\cat T_{A,A}(E))\neq 0$. By Lemma \ref{root}, $\theta^j_{n+1}$ is a bijection and hence $|{}_{p^j\!}F(\pi_{*}(\cat T_{A,A}(E)))|=|V(p^j|A)|$. Then ${}_{p^j\!}F(\pi_{*}(\cat T_{A,A}(E)))$ is a $|V(p^j|A)|$-tuple of $[p^j]_{E}(x)$ in $\pi_{*}(\cat T_{A,A}(E))$. Note that
$$1\in (p,v_1,\cdots,v_n)~~\text{and}~~\deg_W [p]_{E}(x)=p^{n} < |V(p|A)|=p^m.$$
By Corollary \ref{vanish},we have $\pi_{*}(\cat T_{A,A}(E))=0$. Then $\theta^j_{n+1}$ is trivial and ${}_{p^j\!}F(\pi_{*}(\cat T_{A,A}(E)))\cong e$.
\end{proof}

By Lemma \ref{lbsj} and Lemma \ref{MD}, we have
\begin{Cor}\label{LbTAA}
$\pi_{*}(\cat T_{A,A}(E))/I_{n+1-q}=0$ for $q< m+1$, which implies that
$\mathbf{s}_{A,A;E}\geq m$.
\end{Cor}

To achieve our main idea, here we give another proof of the fact that $v_{n-m}$ is a unit in $\pi_{*}(\cat T_{A,A}(E))/I_{n-m}$ by using Theorem \ref{Fac-tuple}. Let $q=m+1$, we have
\begin{Lem}\label{fapj}
Let $n\geq m$. For $1\leq j\leq \min\{i_1, \cdots, i_m\}$, $v_{n-m}$ is a unit in $\pi_{*}(\cat T_{A,A}(E))/I_{n-m}$.
\end{Lem}
\begin{proof}
If $\pi_{*}(\cat T_{A,A}(E))/I_{n-m}=0$, obviously this is true; if $\pi_{*}(\cat T_{A,A}(E))/I_{n-m}\neq 0$, for $1\leq j\leq \min\{i_1, \cdots, i_m\}$, $V(p^j|A)\cong\mathbb{Z}/{p^{j}}\oplus \cdots \oplus \mathbb{Z}/{p^{j}}$ and $|V(p^j|A)|=p^{jm}$. Then by Lemma \ref{root}, we obtain that $\theta^j_{m+1}$ is a bijection and hence $|{}_{p^j\!}F(\pi_{*}(\cat T_{A,A}(E))/I_{n-m})^{p^{j(n-m)}}|=|V(p^j|A)|=p^{jm}$. So $\pi_{*}(\cat T_{A,A}(E))/I_{n-m}$ has a $p^{jm}$-tuple ${}_{p^j\!}F(\pi_{*}(\cat T_{A,A}(E))/I_{n-m})^{p^{j(n-m)}}$ of $g^j_{1,n-m}(x)$. Then by Theorem \ref{Fac-tuple}, we have
$$v_{n-m}^{1+p^{n-m}+\cdots+p^{(j-1)(n-m)}}x+\cdots+v_{n}^{1+p^{n}+\cdots+p^{(j-1)n}}x^{p^{jm}}=v_{n}^{1+p^{n}+\cdots+p^{(j-1)n}}
\prod_{(w_1,w_2,\cdots,w_m)\in V(p^j|A)}(x-\widetilde{\alpha}^{p^{j(n-m)}}_{(w_1,w_2,\cdots,w_m)}).$$
Then $$v_{n-m}^{1+p^{n-m}+\cdots+p^{(j-1)(n-m)}}=(-1)^{p^{jm}}v_{n}^{1+p^{n}+\cdots+p^{(j-1)n}}\prod_{(w_1,w_2,\cdots,w_m)\in V(p^j|A)^*}\widetilde{\alpha}^{p^{j(n-m)}}_{(w_1,w_2,\cdots,w_m)}\in\pi_{*}(\cat T_{A,A}(E))/I_{n-m}.$$
\end{proof}

By Lemma \ref{ind}, we have
\begin{Cor}\label{UbTAA}
Let $A$ be a finite abelian $p$-group with ${\rm rank}_p(A)=m$. If $n\geq m$, then $\pi_{*}(\cat T_{A,A}(E))/I_{n-m}\neq 0$.
\end{Cor}

By Corollary \ref{LbTAA} and Corollary \ref{UbTAA}, we have
\begin{Thm}\label{}
Let $A$ be a finite abelian $p$-group with ${\rm rank}_p(A)=m$, then $\mathbf{s}_{A,A;E}=m$.
\end{Thm}

\subsection{Proof for the case (3) $A$ is a general abelian $p$-group and $C$ is its proper subgroup.}
Without loss of generality, we may suppose that $A$ is $\mathbb{Z}/{p^{i_1}}\oplus \mathbb{Z}/{p^{i_2}}\oplus \cdots \oplus \mathbb{Z}/{p^{i_m}}$ with $i_1\leq i_2 \leq \cdots \leq i_m$ and $C$ is its subgroup $\mathbb{Z}/{p^{j_1}}\oplus \mathbb{Z}/{p^{j_2}}\oplus \cdots \oplus \mathbb{Z}/{p^{j_m}}$ with a group inclusion
\begin{align*}
\varphi:\mathbb{Z}/{p^{j_1}}\oplus \mathbb{Z}/{p^{j_2}}\oplus \cdots \oplus \mathbb{Z}/{p^{j_m}}&\rightarrow \mathbb{Z}/{p^{i_1}}\oplus \mathbb{Z}/{p^{i_2}}\oplus \cdots \oplus \mathbb{Z}/{p^{i_m}}\\
(w_1,w_2,\cdots,w_m)&\mapsto (p^{i_1-j_1}w_1,p^{i_2-j_2}w_2,\cdots, p^{i_m-j_m}w_m),
\end{align*}
otherwise we could replace a set of generators of $A$. There is also a group inclusion from $A/C$ to $A$ as follows:
\begin{align*}
\phi:\mathbb{Z}/{p^{i_1-j_1}}\oplus \cdots \oplus \mathbb{Z}/{p^{i_m-j_m}}&\rightarrow \mathbb{Z}/{p^{i_1}}\oplus \cdots \oplus \mathbb{Z}/{p^{i_m}}\\
(w_1,\cdots,w_m)&\mapsto (p^{i_1-j_1}w_1,\cdots, p^{i_m-j_m}w_m).
\end{align*}
From Theorem \ref{bgcoh}, it follows that
$$\pi_{*}(\cat T_{A,C}(E))\cong L_C^{-1}E^*\psb{x_1,\cdots,x_{m}}/([p^{i_1}]_{E}(x_1),\cdots,[p^{i_m}]_{E}(x_m)),$$
where the multiplicatively closed set $L_C$ is generated by the set
$$M_C=\{\alpha_{(w_1,w_2,\cdots,w_m)}\mid(w_1,w_2,\cdots,w_m)\in A-{\rm im}\phi(A/C)\}.$$
Then
\begin{align*}
\pi_{*}(\cat T_{A,C}(E))/I_{n+1-q}\cong\widetilde{L}_{C,n+1-q}^{-1}E^*/I_{n+1-q}\psb{x_1,\cdots,x_{m}}/([p^{i_1}]_{E}(x_1),\cdots,[p^{i_m}]_{E}(x_m)),
\end{align*}
where the multiplicatively closed set $\widetilde{L}_{C,n+1-q}$ is mod $I_{n+1-q}$ reduction of $L_C$ and generated by the set
$$\widetilde{M}_{C,n+1-q}=\{\widetilde{\alpha}_{(w_1,w_2,\cdots,w_m)} \mid (w_1,w_2,\cdots,w_m)\in A-{\rm im}\phi(A/C)\}.$$

To find tuples of $\pi_{*}(\cat T_{A,C}(E))$, we still focus on the Euler classes $\alpha_{(w_1,w_2,\cdots,w_m)}$ for $(w_1,w_2,\cdots,w_m)\in A$. Note that
$$\alpha_{(u_1,u_2,\cdots,u_m)}-\alpha_{(w_1,w_2,\cdots,w_m)}=\alpha_{(u_1-w_1,u_2-w_2,\cdots,u_m-w_m)}\cdot\varepsilon^{-1}(\alpha_{(u_1,u_2,\cdots,u_m)},\alpha_{(w_1,w_2,\cdots,w_m)}),$$
where $\varepsilon(\alpha_{(u_1,u_2,\cdots,u_m)},\alpha_{(w_1,w_2,\cdots,w_m)})$ is a unit in $\pi_{*}(\cat T_{A,C}(E))$. If $\pi_{*}(\cat T_{A,C}(E))\neq 0$ and $(u_1-w_1,u_2-w_2,\cdots,u_m-w_m)\in A-{\rm im}\phi(A/C)$, then $\alpha_{(u_1,u_2,\cdots,u_m)}-\alpha_{(w_1,w_2,\cdots,w_m)}$ is not a zero divisor in $\pi_{*}(\cat T_{A,C}(E))$. Since ${\rm im}\phi(A/C)$ is a subgroup of $A$, $A$ is the disjoint union $\bigsqcup_{1\leq i\leq |C|}\Bigl(a_i+{\rm im}\phi(A/C)\Bigr)$ of the cosets of ${\rm im}\phi(A/C)$, where $\{a_i\in A\mid 1\leq i\leq |C|\}$ is a complete set of coset representatives of ${\rm im}\phi(A/C)$ in $A$. Thus we have
\begin{Lem}
Let $A$ be a finite abelian $p$-group and $C$ be its subgroup. Let $[A:{\rm im}\phi(A/C)]$ denote a complete set of coset representatives of ${\rm im}\phi(A/C)$ in $A$, and $\mathbf{S}_{[A:{\rm im}\phi(A/C)]}$ denote the subset
$$\{\alpha_{(w_1,w_2,\cdots,w_m)}\in\pi_*(\cat T_{A,C}(E))\mid (w_1,w_2,\cdots,w_m)\in  [A:{\rm im}\phi(A/C)]\}.$$
If $\pi_{*}(\cat T_{A,C}(E))\neq 0$, then $\mathbf{S}_{[A:{\rm im}\phi(A/C)]}$ is a $|C|$-tuple of $\pi_{*}(\cat T_{A,C}(E))$.
\end{Lem}

\begin{Lem}
Let $\mathbf{S}_{[A:{\rm im}\phi(A/C)],j}$ denote the subset
$$\{\alpha_{(w_1,w_2,\cdots,w_k)}\in\pi_*(\cat T_{A,C}(E))\mid (p^jw_1,p^jw_2,\cdots,p^jw_m)=0, (w_1,w_2,\cdots,w_m)\in  [A:{\rm im}\phi(A/C)]\}.$$
If $\pi_{*}(\cat T_{A,C}(E))\neq 0$, then
$\mathbf{S}_{[A:{\rm im}\phi(A/C)],j}$ is an $|\mathbf{S}_{[A:{\rm im}\phi(A/C)],j}|$-tuple of $[p^j]_{E}(x)$ in $\pi_{*}(\cat T_{A,C}(E))$.
\end{Lem}
\begin{proof}
This proof is similar to the proof of Lemma \ref{Re}.
\end{proof}

\begin{Lem}\label{}
Let $\widetilde{\mathbf{S}}_{[A:{\rm im}\phi(A/C)],j}^{p^{j(n+1-q)}}$ denote the subset
$$\{{\widetilde{\alpha}^{p^{j(n+1-q)}}_{(w_1,w_2,\cdots,w_m)}}\in\pi_{*}( \cat T_{A,C}(E))/I_{n+1-q}\mid (p^jw_1,p^jw_2,\cdots,p^jw_m)=0, (w_1,w_2,\cdots,w_m)\in [A:{\rm im}\phi(A/C)]\}.$$
If $\pi_{*}( \cat T_{A,C}(E))/I_{n+1-q}\neq 0$, then $\widetilde{\mathbf{S}}_{[A:{\rm im}\phi(A/C)],j}^{p^{j(n+1-q)}}$ is an $|\widetilde{\mathbf{S}}_{[A:{\rm im}\phi(A/C)],j}^{p^{j(n+1-q)}}|$-tuple of $g_{1,n+1-q}^j(x)$ in $\pi_{*}( \cat T_{A,C}(E))/I_{n+1-q}$.
\end{Lem}

Let $V(p^j|[A:{\rm im}\phi(A/C)])$ denote the set
$$\{(w_1,w_2,\cdots,w_m)\in [A:{\rm im}\phi(A/C)]\mid  (p^jw_1,p^jw_2,\cdots,p^jw_m)=0\},$$
then there is a surjective map $\theta^j_{q}:V(p^j|[A:{\rm im}\phi(A/C)])\rightarrow \widetilde{\mathbf{S}}_{[A:{\rm im}\phi(A/C)],j}^{p^{j(n+1-q)}}$ that maps $(w_1,w_2,\cdots,w_m)$ to $\widetilde{\alpha}^{p^{j(n+1-q)}}_{(w_1,w_2,\cdots,w_m)}$.
\begin{Lem}\label{aijnt}
$\theta^j_{q}$ is a bijection if and only if $\pi_{*}(\cat T_{A,C}(E))/I_{n+1-q}\neq 0$.
\end{Lem}
\begin{proof}
The proof is similar to the proof of Lemma \ref{ijnt}.
\end{proof}

\begin{Cor}\label{LbTAC}
Let $A$ be a finite abelian $p$-group and $C$ be its proper subgroup. Let $[A:{\rm im}\phi(A/C)]$ denote any complete set of coset representatives of ${\rm im}\phi(A/C)$ in $A$ and $j$ be any positive integer, then
$\pi_{*}(\cat T_{A,C}(E))/I_{n+1-q}=0$ for $q< \frac{\log_p|V(p^j|[A:{\rm im}\phi(A/C)])|}{j}+1$.
\end{Cor}
\begin{proof}
Assume that there exists a complete set $[A:{\rm im}\phi(A/C)]_0$, an integer $j_0$, and an integer $q_0< \frac{\log_p|V(p^{j_0}|[A:{\rm im}\phi(A/C)]_0)|}{j_0}+1$ such that $\pi_{*}(\cat T_{A,C}(E))/I_{n+1-q_0}\neq 0$. By Lemma \ref{aijnt}, $\theta^{j_0}_{q_0}$ is a bijection. Then $\widetilde{\mathbf{S}}_{[A:{\rm im}\phi(A/C)]_0,j_0}^{p^{n+1-q_0}}$ is an $|\widetilde{\mathbf{S}}_{[A:{\rm im}\phi(A/C)]_0,j_0}^{p^{n+1-q_0}}|$-tuple of $g^{j_0}_{1,n+1-q_0}(x)$ in $\pi_{*}(\cat T_{A,C}(E))/I_{n+1-q_0}$. Note that the unit $v_{n}^{1+p^{n}+\cdots+p^{(j_0-1)n}}$ is the last coefficient of $g^{j_0}_{1,n+1-q_0}(x)$. Since $C$ is a proper subgroup of $A$, we have
$$ |V(p^{j_0}|[A:{\rm im}\phi(A/C)]_0)|>\deg g^{j_0}_{1,n+1-q_0}(x)=p^{j_0(q_0-1)}.$$
So by Corollary \ref{vanish}, we have $\pi_{*}(\cat T_{A,C}(E))/I_{n+1-q_0}=0$, which contradicts to our assumption. This completes the proof.
\end{proof}
Note that $|V(p^j|[A:{\rm im}\phi(A/C)])|$ depends on the choice of $[A:{\rm im}\phi(A/C)]$. Let $[A:{\rm im}\phi(A/C)]^{\rm max}$ denote a complete set of coset representatives of ${\rm im}\phi(A/C)$ in $A$ such that $|V(p^j|[A:{\rm im}\phi(A/C)]^{\rm max})|$ is maximal. We first simplify $|V(p^j|[A:{\rm im}\phi(A/C)]^{\rm max})|$ by the following lemma.
\begin{Lem}\label{Sivpj}
Let $A$ be a finite abelian $p$-group and $C$ be its proper subgroup. Let $A'$ denote the minimal direct summand of $A$ that contains $C$, then
$$|V(p^j|[A:{\rm im}\phi(A/C)]^{\rm max})|=|V(p^j|[A':{\rm im}\phi(A'/C)]^{\rm max})|.$$
\end{Lem}

\begin{Lem}\label{Cds}
Let $A$ be a finite abelian $p$-group and $C$ be its direct summand, then
$$|V(p^j|[A:{\rm im}\phi(A/C)]^{\rm max})|=|V(p^j|C)|.$$
\end{Lem}
\begin{proof}
Since $A=C\oplus A/C$, then $[A:{\rm im}\phi(A/C)]=\{a_i\mid 1\leq i\leq|C|\}$
where $a_i=(c_i, a'_i)$ for $c_i\in C$ and $a'_i\in A/C$. $V(p^j|[A:{\rm im}\phi(A/C)])=\{(c_i, a'_i)\mid 1\leq i\leq|C|, (p^jc_i, p^ja'_i)=0\}$, we choose $a'_i=0$ for $1\leq i\leq|C|$, then $|V(p^j|[A:{\rm im}\phi(A/C)]^{\rm max})|=|V(p^j|C)|.$
\end{proof}

To compute $|V(p^j|[A:{\rm im}\phi(A/C)]^{\rm max})|$, we need the following lemma.
\begin{Lem}\label{Vpj}
Let $A$ be a finite abelian $p$-group and $C$ be its proper subgroup. Then there is an injection of cosets
$$\bigsqcup_{1\leq i\leq\frac{|V(p^j|A)|}{|V(p^j|{\rm im}\phi(A/C))|}}\Bigl(b_i+V(p^j|{\rm im}\phi(A/C))\Bigr)\hookrightarrow \bigsqcup_{1\leq k\leq |C|}\Bigl(a_k+{\rm im}\phi(A/C)\Bigr)$$
induced by the inclusion $V(p^j|A)\hookrightarrow A$.
\end{Lem}
\begin{proof}
If $b_i\in a_{k}+{\rm im}\phi(A/C)$, then $b_i+V(p^j|{\rm im}\phi(A/C)) \subseteq a_{k}+{\rm im}\phi(A/C)$. So it suffices to prove that for any $1\leq k\leq |C|$, $a_k+{\rm im}\phi(A/C)$ contains at most one $b_i$ for  $1\leq i\leq\frac{|V(p^j|A)|}{|V(p^j|{\rm im}\phi(A/C))|}$. If $a_k+{\rm im}\phi(A/C)$ contains
$b_{i_1}$ and $b_{i_2}$ for $1\leq i_1\neq i_2\leq\frac{|V(p^j|A)|}{|V(p^j|{\rm im}\phi(A/C))|}$, then there are $a',a''\in {\rm im}\phi(A/C)$ such that
$b_{i_1}=a_k+a', b_{i_2}=a_k+a''$, which follows that $b_{i_1}- b_{i_2}=a'-a''$. Note that $a'-a''\in {\rm im}\phi(A/C)$, then $b_{i_1}- b_{i_2}\in {\rm im}\phi(A/C)$. Since
$$b_{i_1}- b_{i_2}\in V(p^j|A)-V(p^j|{\rm im}\phi(A/C))=V(p^j|A-{\rm im}\phi(A/C))\subseteq A-{\rm im}\phi(A/C),$$
this is a contradiction.
\end{proof}
By Lemma \ref{Vpj} and Lemma \ref{Sivpj}, we have
\begin{Cor}\label{CpvpjAC}
Let $A$ be a finite abelian $p$-group and $C$ be its proper subgroup. Let $A'$ denote the minimal direct summand of $A$ that contains $C$, then
$$|V(p^j|[A:{\rm im}\phi(A/C)]^{\rm max})|=\frac{|V(p^j|A)|}{|V(p^j|{\rm im}\phi(A/C))|}=\frac{|V(p^j|A')|}{|V(p^j|{\rm im}\phi(A'/C))|}$$
and
$$\max_{j\in \mathbb{N}^+}{[\frac{\log_p|V(p^j|[A:{\rm im}\phi(A/C)]^{\rm max})|}{j}]}=\max_{j\in \mathbb{N}^+}{[\frac{\log_p|V(p^j|A')|-\log_p|V(p^j|{\rm im}\phi(A'/C))|}{j}]}.$$
\end{Cor}
\begin{Rem}
$[\frac{\log_p|V(p^j|[A:{\rm im}\phi(A/C)]^{\rm max})|}{j}]$ reaches the maximum when $j\leq \log_p|A|$.
\end{Rem}

By Corollary \ref{LbTAC} and Corollary \ref{CpvpjAC}, we have
\begin{Cor}\label{LBTAC}
Let $A$ be a finite abelian $p$-group and $C$ be its proper subgroup, then
$$\pi_{*}(\cat T_{A,C}(E))/I_{n+1-q}=0~~\text{for}~~ q<\max_{j\in \mathbb{N}^+}{\frac{\log_p|V(p^j|A)|-\log_p|V(p^j|{\rm im}\phi(A/C))|}{j}}+1.$$
Which implies that
$$\mathbf{s}_{A,C;E}\geq\max_{j\in \mathbb{N}^+}{\lceil\frac{\log_p|V(p^j|A)|-\log_p|V(p^j|{\rm im}\phi(A/C))|}{j}\rceil}.$$
\end{Cor}

\begin{Lem}\label{ind}
Let $A$ be an abelian $p$-group and $C$ be its subgroup with an inclusion
\begin{align*}
\varphi:C=\mathbb{Z}/{p^{j_1}}\oplus\mathbb{Z}/{p^{j_2}}\oplus \cdots \oplus \mathbb{Z}/{p^{j_m}}&\rightarrow A=\mathbb{Z}/{p^{i_1}}\oplus \mathbb{Z}/{p^{i_2}}\oplus \cdots \oplus \mathbb{Z}/{p^{i_m}}\\
(w_1,w_2,\cdots,w_m)&\mapsto (p^{i_1-j_1}w_1,p^{i_2-j_2}w_2,\cdots, p^{i_m-j_m}w_m).
\end{align*}
Let $A'$ be the subgroup of $A$ with $A=A'\oplus\mathbb{Z}/{p^{i_m}}$ and $C'$ be the subgroup of $C$ with $C=C'\oplus\mathbb{Z}/{p^{j_m}}$.
If $E$ is Landweber exact and $\pi_{*}(\cat T_{A',C'}(E))/I_{n-k}\neq 0$, then
\begin{enumerate}
\item[\rm (i)] $\pi_{*}(\cat T_{A,C}(E))/I_{n-k-1}\neq 0$ if $j_m> 0$;
\item[\rm (ii)] $\pi_{*}(\cat T_{A,C}(E))/I_{n-k}\neq 0$ if $j_m=0$.
\end{enumerate}
\end{Lem}
\begin{proof}
We first prove the case (i): $j_m> 0$. If $E$ is Landweber exact, then by Lemma \ref{LE} we obtain that $\cat T_{A',C'}(E)$ is Landweber exact. Since $\pi_{*}(\cat T_{A',C'}(E))/I_{n-k}\neq 0$, by exactness of
$${\small\CD
  0 @>>>\pi_{*}(\cat T_{A',C'}(E))/I_{n-k-1} @>\cdot v_{n-k-1}>> \pi_{*}(\cat T_{A',C'}(E))/I_{n-k-1}@>>> \pi_{*}(\cat T_{A',C'}(E))/I_{n-k} @>>> 0,
\endCD}$$
we obtain that $v_{n-k-1}$ is not a unit in $\pi_{*}(\cat T_{A',C'}(E))/I_{n-k-1}\neq 0$. By Theorem \ref{bgcoh}, we have
$$\pi_{*}(\cat T_{A',C'}(E))\cong L^{-1}_{C'}E^*(BA'_+),$$
where the multiplicatively closed set $L_{C'}$ is generated by the set
\begin{align*}
M_{C'}=\{&\alpha_{(w_1,\cdots,w_m)}\in E^*(BA'_+)\mid (w_1,\cdots,w_m)\in A'-{\rm im}\phi(A'/C')\}.
\end{align*}
Let $\widetilde{L}_{C,i}$ denote the multiplicatively closed set generated by the set
\begin{align*}
\widetilde{M}_{C,i}=\{&\widetilde{\alpha}_{(w_1,\cdots,w_m)}\in E^*(BA'_+)\psb{x_{m}}/I_{i}\mid (w_1,\cdots,w_m)\in A-{\rm im}\phi(A/C)\}.
\end{align*}
Since $E$ is Landweber exact, by a similar proof of Lemma \ref{LE}, we deduce that for each $i$ multiplication by $v_i$ is monic on $ \widetilde{L}^{-1}_{C,i}E^*(BA'_+)\psb{x_{m}}/I_i$.

Note that $i:A'\hookrightarrow A' \times U(1)$ is the right inverse of $p:A' \times U(1)\to A'$, then the homomorphism $Bp^*:E^*(BA'_+)\to E^*(BA'_+)\psb{x_{m}}$ is injective. As $E^*(BA'_+)=E^*(BA'_+)\psb{x_{m}}/(x_m)$ is an $E^*(BA'_+)\psb{x_{m}}$-module with the module map induced by $Bi^*$, then we have
$$Bi^*(L^{-1}_{C}E^*(B(A'\times U(1)_+))= L^{-1}_{C'}E^*(BA'_+).$$
Since the localization functor is exact, there is an injective homomorphism
$$L^{-1}_{C}Bp^*: L^{-1}_{C'}E^*(BA'_+)\to L^{-1}_{C}E^*(BA'_+)\psb{x_{m}}.$$
Then we have the following commutative diagram
{\footnotesize$$\CD
  0 @>>> \widetilde{L}^{-1}_{C',i}E^*(BA'_+)/I_{i} @>\cdot v_{i}>>\widetilde{L}^{-1}_{C',i}E^*(BA'_+)/I_{i} @>>> \widetilde{L}^{-1}_{C',i+1}E^*(BA'_+)/I_{i+1} @>>> 0 \\
      @. @V  VV @V  VV @V  VV @. \\
  0 @>>> \widetilde{L}^{-1}_{C,i}E^*(BA'_+)\psb{x_{m}}/I_{i} @>\cdot v_{i}>> \widetilde{L}^{-1}_{C,i}E^*(BA'_+)\psb{x_{m}}/I_{i}  @>>> \widetilde{L}^{-1}_{C,i+1}E^*(BA'_+)\psb{x_{m}}/I_{i+1}  @>>> 0,
\endCD$$}
and deduce that the homomorphism $L^{-1}_{C}Bp^*:\widetilde{L}^{-1}_{C'}E^*(BA'_+)/I_{i}\rightarrow \widetilde{L}^{-1}_{C}E^*(BA'_+)\psb{x_{m}}/I_{i}$ is injective for each $i$. Since $\pi_{*}(\cat T_{A',C'}(E))/I_{n-k}=\widetilde{L}^{-1}_{C'}E^*(BA'_+)/I_{n-k}\neq 0$, we have $\widetilde{L}^{-1}_{C,n-k}E^*(BA'_+)\psb{x_{m}}/I_{n-k}\neq 0$. By exactness of
{\scriptsize$$\CD
 0 @>>>\widetilde{L}^{-1}_{C,n-k-1}E^*(BA'_+)\psb{x_{m}}/I_{n-k-1} @>\cdot v_{n-k-1}>> \widetilde{L}^{-1}_{C,n-k-1}E^*(BA'_+)\psb{x_{m}}/I_{n-k-1}@>>> \widetilde{L}^{-1}_{C,n-k}E^*(BA'_+)\psb{x_{m}}/I_{n-k} @>>> 0,
\endCD$$}
we obtain that $v_{n-k-1}$ is not a unit in $ \widetilde{L}^{-1}_{C,n-k-1}E^*(BA'_+)\psb{x_{m}}/I_{n-k-1}$.

Using the Gysin sequence of $S^1 \to BA\stackrel{B({\rm id}\times\rho_{\frac{1}{p^{i_m}}})}\to B(A'\times U(1))$, we have
$E^*(BA_+)\cong E^*(BA'_+)\psb{x_{m}}/([p^{i_m}]_{E}(x_m))$
and
$$\pi_{*}(\cat T_{A,C}(E))/I_{n-k-1}\cong \widetilde{L}^{-1}_{C,n-k-1}E^*(BA'_+)\psb{x_{m}}/(I_{n-k-1},[p^{i_m}]_{E}(x_m)).$$
Note that $E^*(BA_+)/I_{n-k-1}$ is an $E^*(B(A'\times U(1))_+)/I_{n-k-1}$-module and the localization functor is exact, we have a short exact sequence:
{\footnotesize$$\xymatrix@C=0.5cm{
  0 \ar[r] & \widetilde{L}^{-1}_{C,n-k-1}E^*(BA'_+)\psb{x_{m}}/I_{n-k-1}\ar[rr]^{\cdot [p^{i_m}]_{E}(x_m)} &&  \widetilde{L}^{-1}_{C,n-k-1}E^*(BA'_+)\psb{x_{m}}/I_{n-k-1}\ar[rr]^{} && \pi_{*}(\cat T_{A,C}(E))/I_{n-k-1}\ar[r] & 0 }.$$}
Now $[p^{i_m}]_{E}(x_m)$ is not a unit in $\widetilde{L}^{-1}_{C,n-k-1}E^*(BA'_+)\psb{x_{m}}/I_{n-k-1}$ since its leading coefficient $ v_{n-k-1}^{1+p^{n-k-1}+\cdots+p^{(i_m-1)(n-k-1)}}$ is not a unit. Therefore $ \pi_{*}(\cat T_{A,C}(E))/I_{n-k-1}\neq 0$.

Now we prove the case (ii): $j_m=0$, that is $C=C'$. Since $\pi_{*}(\cat T_{A',C'}(E))/I_{n-k}\neq 0$, we have $\widetilde{L}^{-1}_{C,n-k}E^*(BA'_+)\psb{x_{m}}/I_{n-k}\neq 0$. As
$$\pi_{*}(\cat T_{A,C}(E))/I_{n-k}\cong \widetilde{L}^{-1}_{C,n-k}E^*(BA'_+)\psb{x_{m}}/(I_{n-k},[p^{i_m}]_{E}(x_m)),$$
then we obtain a short exact sequence:
$$\xymatrix@C=0.5cm{
\widetilde{L}^{-1}_{C,n-k-1}E^*(BA'_+)\psb{x_{m}}/I_{n-k}\ar[rr]^{\cdot [p^{i_m}]_{E}(x_m)} &&  \widetilde{L}^{-1}_{C,n-k-1}E^*(BA'_+)\psb{x_{m}}/I_{n-k}\ar[rr]^{} && \pi_{*}(\cat T_{A,C}(E))/I_{n-k}\ar[r] & 0 }.$$
Since $x_{m}$ is not invertible in $\widetilde{L}^{-1}_{C,n-k-1}E^*(BA'_+)\psb{x_{m}}/I_{n-k}$, which implies that $\cdot [p^{i_m}]_{E}(x_m)$ is not surjective, thus $\pi_{*}(\cat T_{A,C}(E))/I_{n-k}\neq 0$.
\end{proof}

By inductively using Lemma \ref{ind}, we have
\begin{Cor}\label{UbTAC}
Let $A$ be a finite abelian $p$-group and $C$ be its proper subgroup. If $n\geq {\rm rank}_p(C)$, then $\pi_{*}(\cat T_{A,C}(E))/I_{n-{\rm rank}_p(C)}\neq 0$.
\end{Cor}

By Corollary \ref{LBTAC} and Corollary \ref{UbTAC}, we have
\begin{Thm}\label{BsAC}
Let $A$ be a finite abelian $p$-group and $C$ be its proper subgroup, then
$$t\leq \mathbf{s}_{A,C;E}\leq {\rm rank}_p(C)$$
where
$$t=\max_{j\in \mathbb{N}^+}{\lceil\frac{\log_p|V(p^j|A)|-\log_p|V(p^j|{\rm im}\phi(A/C))|}{j}\rceil}.$$
\end{Thm}

By Lemma \ref{Cds}, Lemma \ref{MD} and Theorem \ref{BsAC}, we have
\begin{Cor}\label{}
Let $A$ be a finite abelian $p$-group and $C$ be its direct summand, then
$$\mathbf{s}_{A,C;E}= {\rm rank}_p(C).$$
\end{Cor}

\newpage
\addcontentsline{toc}{section}{\numberline {}{\bf References}}

\newcommand{\etalchar}[1]{$^{#1}$}

\bigskip

\newpage
\noindent {\sc Yangyang Ruan

\noindent
Beijing Institute of Mathematical Sciences and Applications\\ No. 544, Hefangkou Village, Huaibei Town, Huairou District, Beijing, 101408 \\P.R.China

\noindent
Academy of Mathematics and Systems Science, Chinese Academy of Sciences\\
No. 55, Zhongguancun East Road, Haidian District,
Beijing, 100190\\P.R.China
}\\
\vspace{-0.3cm}
\hspace{-0.1cm}\noindent{{\it E-mails:}}
\texttt{ruanyy@amss.ac.cn}

\end{sloppypar}
\end{document}